\documentclass[tbtags,a4paper,11pt]{article}

\author{Cecilia Holmgren\footnote{Department of Mathematics, Stockholm University, 114 18 Stockholm, Sweden. Supported in part by the Swedish Research Council.}
	\and
	Svante Janson\footnote{Department of Mathematics, Uppsala University, SE-75310 Uppsala, Sweden. Supported in part by the Knut and Alice Wallenberg Foundation.}}
\title{Limit Laws for Functions of Fringe trees for Binary Search Trees and
  Recursive Trees} 
\date{26 June, 2014}

\usepackage{bbold}
\usepackage{mathrsfs}
\usepackage{mathrsfs}
\usepackage{amsmath}
\usepackage{amsfonts}
\usepackage{amsthm}
\usepackage{amssymb, graphicx, epsfig}
\usepackage{times}
\usepackage{hyperref}
\usepackage[numbers,sort&compress]{natbib}

\usepackage{latexsym}

\numberwithin{equation}{section}

\hypersetup{
    bookmarks=true,         % show bookmarks bar?
    unicode=false,          % non-Latin characters in Acrobat’s bookmarks
    pdftoolbar=true,        % show Acrobat’s toolbar?
    pdfmenubar=true,        % show Acrobat’s menu?
    pdffitwindow=true,      % page fit to window when opened
    pdftitle={My title},    % title
    pdfauthor={Author},     % author
    pdfsubject={Subject},   % subject of the document
    pdfnewwindow=true,      % links in new window
    pdfkeywords={keywords}, % list of keywords
    colorlinks=true,       % false: boxed links; true: colored links
    linkcolor=blue,          % color of internal links
    citecolor=blue,        % color of links to bibliography
    filecolor=blue,      % color of file links
    urlcolor=blue           % color of external links
}

\newcommand\Bin{\operatorname{Bin}}

\newtheorem{thm}{Theorem}[section] 
\newtheorem{Lemma}[thm]{Lemma}%[section]
\newtheorem{cor}[thm]{Corollary}%[section]
\theoremstyle{definition}
\newtheorem{rem}[thm]{Remark}%[section]
\newtheorem{defn}[thm]{Definition}%[section]
\newtheorem{example}[thm]{Example}
\newtheorem{problem}[thm]{Problem}

\newenvironment{romenumerate}[1][0pt]{% optional argument changes indentation
\addtolength{\leftmargini}{#1}\begin{enumerate}% gives (i), (ii) etc.
 }{\end{enumerate}}

\newcounter{thmenumerate}
\newenvironment{thmenumerate}
{\setcounter{thmenumerate}{0}%
 \def\item{\par% \ifnum\thethmenumerate=0\else\par\fi %\noindent\fi
 \refstepcounter{thmenumerate}\textup{(\roman{thmenumerate})\enspace}}
}
{}

\newcommand\REM[1]{{\raggedright\texttt{[#1]}\par\marginal{XXX}}}
\newcommand\ga{\alpha}
\newcommand\gb{\beta}
\newcommand\gd{\delta}

\newcommand\gam{\gamma}
\newcommand\gG{\Gamma}

\newcommand\gl{\lambda}
\newcommand\gL{\Lambda}

\newcommand\gs{\sigma}
\newcommand\gss{\sigma^2}

\newcommand{\refT}[1]{Theorem~\ref{#1}}
\newcommand{\refC}[1]{Corollary~\ref{#1}}
\newcommand{\refL}[1]{Lemma~\ref{#1}}
\newcommand{\refR}[1]{Remark~\ref{#1}}
\newcommand{\refS}[1]{Section~\ref{#1}}
\newcommand{\refE}[1]{Example~\ref{#1}}
\newcommand{\refPQ}[1]{Problem~\ref{#1}}

\newcommand\punkt{.\spacefactor=1000}  
\newcommand\iid{i.i.d\punkt}    
\newcommand\ie{i.e\punkt}
\newcommand\eg{e.g\punkt}

\newcommand\cf{cf\punkt}

\newcommand\qw{^{-1}}

\newcommand\qq{^{1/2}}
\newcommand\qqw{^{-1/2}}

\newcommand\etta{\boldsymbol1} 
\newcommand\ett[1]{\etta\set{#1}}
\newcommand\bigett[1]{\etta\bigset{#1}}

\newcommand\set[1]{\ensuremath{\{#1\}}}
\newcommand\bigset[1]{\ensuremath{\bigl\{#1\bigr\}}}

\newcommand\xpar[1]{(#1)}
\newcommand\bigpar[1]{\bigl(#1\bigr)}
\newcommand\Bigpar[1]{\Bigl(#1\Bigr)}
\newcommand\biggpar[1]{\biggl(#1\biggr)}
\newcommand\lrpar[1]{\left(#1\right)}

\newcommand\lrabs[1]{\left|#1\right|}

\newcommand\E{\operatorname{\mathbb E{}}}
\renewcommand\P{\operatorname{\mathbb P{}}}
\newcommand\Var{\operatorname{Var}}
\newcommand\Cov{\operatorname{Cov}}

\newcommand\Exp{\operatorname{Exp}}
\newcommand\Po{\operatorname{Po}}
\newcommand\Be{\operatorname{Be}}
\newcommand{\tend}{\longrightarrow}
\newcommand\dto{\overset{\mathrm{d}}{\tend}}
\newcommand\pto{\overset{\mathrm{p}}{\tend}}

\newcommand\eqd{\overset{\mathrm{d}}{=}}

\newcommand\xfrac[2]{#1/#2}
\newcommand\xpfrac[2]{(#1)/#2}
\newcommand\xqfrac[2]{#1/(#2)}

\newcommand\Bigparfrac[2]{\Bigpar{\frac{#1}{#2}}}

\newcommand\bbR{\mathbb R}

\newcommand\bbZ{\mathbb Z}

\newcommand\subtn{\text{ is a subtree in $\cT_{n}$}}

\newcommand\N{\mathcal N}
\newcommand\hX{\hat X}
\newcommand\hZ{\hat Z}
\newcommand\hmu{\hat\mu}
\newcommand\nk{_{n,k}}
\newcommand\cA{\mathcal A}
\newcommand\cL{\mathcal L}
\newcommand\CramerWold{Cram\'er--Wold}
\newcommand\yi{I}
\newcommand\tyi{\tilde I}
\newcommand\yl{\yi^{\mathrm{L}}}

\newcommand\xxn{\tilde X_n}
\newcommand\yyn{\tilde Y_n}
\newcommand\ff{\bar f}
\newcommand\fbar{\bar f}
\newcommand\gbar{\bar g}
\newcommand\dd{\,\mathrm{d}}
\newcommand\rrt{random recursive tree}
\newcommand\bst{binary search tree}
\newcommand\absx[2]{|#2|_{#1}}
\newcommand\absn[1]{\absx{n}{#1}}
\newcommand\absni[1]{\absx{n+1}{#1}}
\newcommand\xnkp{X_{n,k}^{P}}
\newcommand\pkp{p_{k,P}}
\newcommand\prkp{\hat p_{k,P}}
\newcommand\ipik{I^P_{i,k}}
\newcommand\sumin{\sum_{i=1}^{n}}
\newcommand\sumini{\sum_{i=1}^{n+1}}
\newcommand\sumkn{\sum_{k=1}^{n}}
\newcommand\sumk{\sum_{k=1}^{\infty}}
\newcommand\iik{\yi_{i,k}}
\newcommand\ixik{\yi_{i,k}^{\phantom P}}
\newcommand\iq{\yi^{\phantom P}}
\newcommand\mtoo{\ensuremath{{m\to\infty}}}
\newcommand\ntoo{\ensuremath{{n\to\infty}}}
\newcommand\Ntoo{\ensuremath{{N\to\infty}}}
\newcommand\cT{{\mathcal T}}
\newcommand\pkt{p_{k,T}}
\newcommand\pmtx{p_{m,T'}}
\newcommand\gstt{\gs_{T,T'}}
\newcommand\hgsll{\hat\gs_{\gL,\gL'}}
\newcommand\qtt{q^{T}_{T'}}
\newcommand\gsxx{\sigma} 
\newcommand\gskm{\gsxx_{k,m}}

\newcommand\upp[1]{^{(#1)}}
\newcounter{CC}
\newcommand{\CC}{\stepcounter{CC}\CCx} %new constant C_i
\newcommand{\CCx}{C_{\arabic{CC}}}     %repeats the last C_i
\newcounter{cc}
\newcommand{\cc}{\stepcounter{cc}\ccx} %new constant c_i
\newcommand{\ccx}{c_{\arabic{cc}}}     %repeats the last c_i
\newcommand\pfcase[2]{\smallskip\noindent\emph{Case #1: #2} \noindent}
\newcommand\tia{A}

\newcommand\nn{1,\dots,n}
\newcommand\nni{1,\dots,n+1}
\newcommand\bX{\mathbf{{X}}}
\newcommand\bmu{\boldsymbol{\mu}}
\newcommand\hbX{\mathbf{\hat{X}}}
\newcommand\hbmu{\boldsymbol{\hat\mu}}
\newcommand\hgs{\hat\sigma}
\newcommand\hgss{\hat\sigma^2}
\newcommand\ctn{\cT_n}
\newcommand\ctk{\cT_k}
\newcommand\glk{\gL_k}
\newcommand\gln{\gL_n}
\newcommand\gLL{\boldsymbol{\gL}}
\newcommand\cN{\N}
\newcommand\bgam{\beta}

\newcommand\bgamx{\bgam^*}
\newcommand\hbgamx{\hat{\bgam}^*}
\newcommand\muy[1]{\mu_{Y,{#1}}}
\newcommand\muyl{\muy{\ell}} %\mu_{Y^\ell}
\newcommand\muz[1]{\mu_{Z,{#1}}}
\newcommand\muzl{\muz{\ell}} %\mu_{Z^\ell}
\newcommand\sigmay[1]{\gs_{Y,{#1}}}
\newcommand\sigmayl{\sigmay{\ell}}
\newcommand\sigmaz[1]{\gs_{Z,{#1}}}
\newcommand\sigmazl{\sigmaz{\ell}}
\newcommand\gssy[1]{\gss_{Y,{#1}}}
\newcommand\gssz[1]{\gss_{Z,{#1}}}
\newcommand\yln{Y_{\ell,n}}
\newcommand\zln{Z_{\ell,n}}
\newcommand\yxn[1]{Y_{#1,n}}
\newcommand\zxn[1]{Z_{#1,n}}
\newcommand\hD{\hat D}
\newcommand\dn[1]{D_{n,#1}}
\newcommand\hdn[1]{\hD_{n,#1}}
\newcommand\mud[1]{\mu_{D,#1}}
\newcommand\hmud[1]{\mu_{\hD,#1}}
\newcommand\gsshd[1]{\gss_{\hD,#1}}
\newcommand\ttt[1]{T_{#1}}
\newcommand\hp{\hat p}
\newcommand\hP{\hat P}

\newcommand\hq{\hat q}
\begin{document}
\maketitle

\begin{abstract}
We prove limit theorems for
sums of functions of subtrees of  binary search trees
and random recursive trees.  
In particular, we give simple new proofs of the
fact that the number of fringe trees of size $ k=k_n $
in the binary search tree and the random recursive
 tree  (of total size $ n $) asymptotically has  a Poisson distribution
if $ k\rightarrow\infty $, 
and that the distribution is asymptotically normal for $ k=o(\sqrt{n}) $.
 
Furthermore, we prove similar results for the number of subtrees of
size $ k $ with some required property $ P $,  for example
the  number of copies of a certain fixed subtree $ T $. 
Using the \CramerWold{} device, we  show also
that these random numbers for different fixed subtrees converge jointly 
to a multivariate normal distribution. 

As an application of the general results, 
we obtain a normal limit law
for the number of $\ell$-protected nodes in a binary search tree or random
recursive tree.

The proofs use a new version of a representation by Devroye, and Stein's
method (for both normal and Poisson approximation) together with certain
couplings. 
\end{abstract}



\section{Introduction}\label{S:intro}
%\subsection {Preliminaries}\label{prel}
In this paper we consider fringe trees of the random binary search
tree as well as of the random recursive tree;
recall that a \emph{fringe tree} is a subtree consisting of some node and
all its 
descendants, see \citet{Aldous-fringe} for a general theory, 
and note that fringe trees typically are
"small" compared to the whole tree. 
(All subtrees considered in the present
paper are of this type, and we will use 'subtree' and 'fringe tree' as
synonyms.) 
We will use a representation of Devroye
\cite{Devroye1,Devroye2} for the binary search tree, and a well-known
bijection between binary trees and recursive trees, 
together with different
applications of Stein's method for both normal and Poisson approximation
to give both new general results on the
asymptotic distributions for random variables depending on fringe trees, 
and more direct proofs of several earlier results in the field. 
We give also examples of applications of these 
general results, for example
to the number of protected nodes in the binary search tree 
or \rrt{} studied by Mahmoud and Ward \cite{MahmoudWard2012,MahmoudWard2014}
and to the shape functionals for the \bst{} or \rrt, see \refS{Sapp}.

The \emph{binary search tree} is the tree representation of the 
sorting algorithm Quicksort, see e.g. \cite{KnuthIII} or \cite{Drmota}.  
Starting with $ n $ distinct numbers called
keys, we draw one of the keys at random and associate it to the root. Then
we draw one of the remaining keys. We compare it with the root, and
associate it to the left child if it is smaller than the key at the root, and to
the right child if it is larger. We continue recursively by drawing new
keys until the set is exhausted. The comparison for each new  key
starts at the root, and at each node the key visits, it proceeds to the
left/right child if it is smaller/larger than the key associated to that
node; eventually, the new key is associated to the first empty node it visits. 
In the final tree, all
the $ n $ ordered numbers are sorted by size, so that smaller numbers are in
left subtrees, and larger numbers are in right subtrees.
We let $ \cT_{n} $ denote a random binary search tree with $ n $ nodes.

We use the representation of the binary search tree by Devroye
\cite{Devroye1,Devroye2}.  
We may clearly assume that the keys are $1,\dots,n$.
We assign, independently, each key $ k$ a uniform random variable $ U_{k} $
in $(0,1)$ which we regard as a time stamp indicating the time when the key is
drawn. (We may and will assume that the $U_k$ are distinct.)
%Thus, the first key is the key $k$ such that $U_k$ is minimal, and so on.
%A random permutation can be easily recovered from this data.  
The random binary search tree constructed by drawing the keys in this order,
i.e., in order of increasing $U_k$, 
then is the unique binary tree with nodes labelled by $ (1,U_1),\dots,(n,U_n) $
with the property that it is a binary search tree with respect to the first
coordinates in the pairs, and along every path down from the root the values
$ U_i $ are increasing. 
We will also use a cyclic version of this representation described in
\refS{SScyclic}. 

Recall that the \emph{random recursive tree} is constructed recursively, by
starting with a root with label $ 1 $, and at stage $ i $ ($ i=2,\dots,n $)
a new node with label $ i $ is attached uniformly at random to one of the
previous nodes $ 1,\dots,i-1 $.
We let  $ \Lambda_{n} $ denote a random recursive tree with $ n $
nodes.
We may regard the \rrt{} as an ordered tree by ordering the children of each
node by their labels, from left to right.

There is a well-known bijection between ordered 
trees of size $n$ and
binary trees of size $n-1$,
see e.g.\ Knuth \cite[Section 2.3.2]{KnuthI} who calls this the 
\emph{natural correspondence} (the same bijection  is also called  
the \emph{rotation correspondence}):
Given an ordered  tree  with $ n $ nodes, eliminate first the root, and
arrange all its children in a path from left to right, as right children of
each other. Continue recursively,
with the children of each node arranged in a path from left to right, with
the first child attached to its parent as the left child.
This yields a binary tree
with $ n-1 $ nodes, and the transformation is invertible.  
% Recall that an increasing tree is a tree with nodes labelled by
% $1,2,\dots$ such that the 
% labels are increasing along each path from the root.
% We regard an increasing tree as an ordered tree by 
% ordering the children of each node according to their labels.
As noted by \citet{Devroye1}, see also \citet{FuchsHwangNeininger},
the natural correspondence extends to a 
% bijection between
% increasing trees of order $n$ and 
% binary increasing trees  of order $n-1$,
% and since there are obvious isomorphisms between 
% random recursive trees and 
% uniformly random increasing
% trees, as well as between 
% binary search trees and
% uniformly random increasing binary trees  (in both cases, label the
% nodes in the order they are added to the trees in the constructions above,
% see e.g.\ \cite[Sections 1.3--1.4]{Drmota}),
% the natural correspondence yields a
coupling between 
the random recursive tree $\gL_{n}$ and the binary search tree $\cT_{n-1}$;
the probability distributions are equal by induction because
the $n$ possible places to add a new node to $\gL_n$ correspond to the $n$
possible places (external leaves) to add a new node to $\cT_{n-1}$,
and these places have equal probabilities for both models.
%the binary search tree and the random recursive tree, 

Note that a left child in the binary search tree corresponds to an eldest
child in the random recursive tree, while a right child corresponds to a
sibling. 
We say that a proper subtree in a binary tree is left-rooted [right-rooted]
if its root is
a left [right] child. 
Thus, for  $1<k<n$, subtrees of size $k$ in the random recursive tree $\gL_n$, 
correspond to left-rooted subtrees of size $k-1$ in the binary search tree
$\cT_{n-1}$, while subtrees of size 1 (i.e., leaves) correspond to nodes
without left child. (Alternatively, we can say that subtrees of size 1 in the
recursive tree correspond to empty left subtrees in the binary tree.)

An example of a bijection obtained from the natural correspondence is
illustrated in Figures \ref{natural1}--\ref{natural2}. 
Note that the labels in the \rrt{} correspond to the time stamps in the
\bst{} (replaced by $1,2,\dots$ in increasing order), 
while the keys in the \bst{} are determined by the tree structure and
thus redundant.

\begin{rem}\label{Rinc}
  The \bst{} with its time stamps and the \rrt{} with its labels are both
  increasing trees, \ie, labelled trees where the label of a node is greater
  than the label of its parent.
We allow the labels in an increasing tree to be arbitrary real numbers, but
we are only interested in the order relations between them and consider two
increasing trees that are isomorphic as trees and with labels in the same
order to be the same; hence we may 
freely relabel  (preserving the order), for example by 1,2,\dots.

Note that the \bst{} yields a uniformly distributed
increasing binary tree, and the \rrt{} a uniformly distributed (unordered)
increasing tree,
see e.g.\ \cite[Sections 1.3--1.4]{Drmota}.
Note also that the natural correspondene extends to a bijection between
increasing binary trees and increasing ordered trees that have the children
of each node ordered according to their labels.
\end{rem}

\begin{figure}
\centering
\begin{minipage}{.45\textwidth}
  \centering
  \includegraphics[width=.75\linewidth]{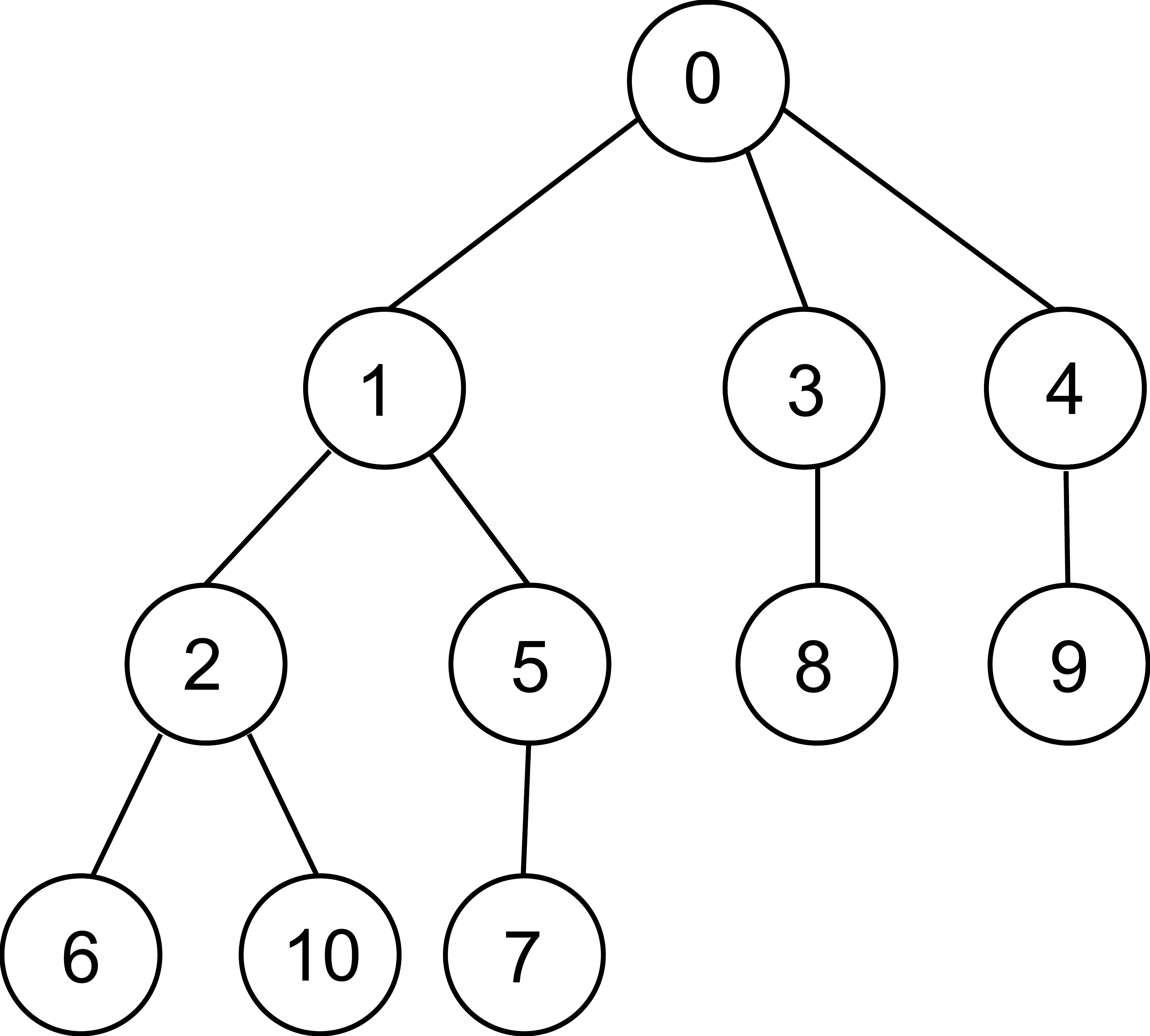}
  \caption{A recursive tree. The root has label $ 0 $ instead of 1, to
	better illustrate the bijection.} 
  \label{natural1}
\end{minipage}%
\hfill\hfill
\begin{minipage}{.45\textwidth}
  \centering
  \includegraphics[width=.7\linewidth]{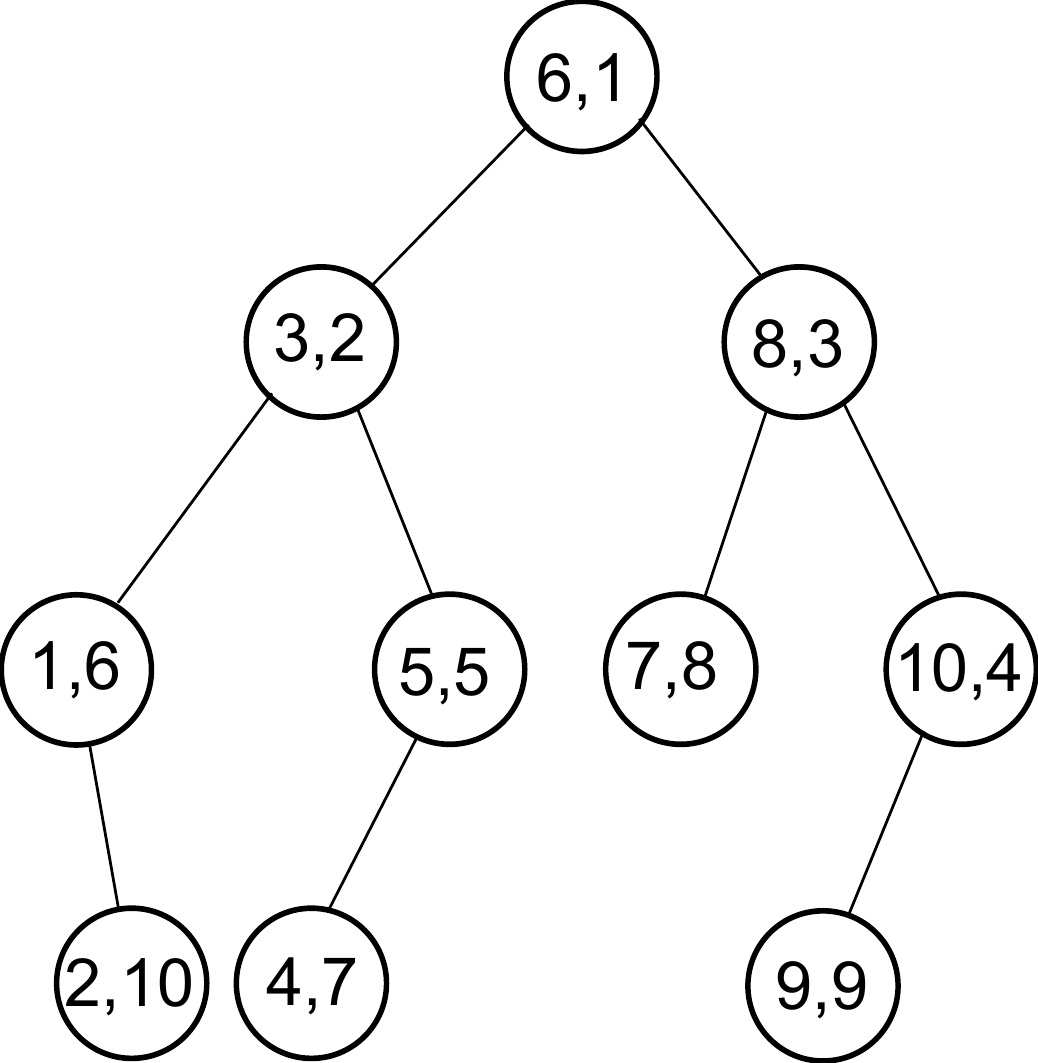}
  \caption{The corresponding binary search tree. 
The first and second labels are the keys and time stamps,
respectively, using the time stamps $1,\dots,10$ for convenience.} 
  \label{natural2}
\end{minipage}
%\label{natural}
\end{figure}

\begin{rem}\label{R0}
  We may consider a subtree of the \bst{} in two different ways; either we 
regard it as an (unlabelled) binary tree by ignoring the time stamps (and
keys), or we may regard it as an increasing binary tree by keeping the time
stamps (perhaps replacing them, in   order, by 1,2,\dots, see \refR{Rinc}).

Similarly, there are three different ways to look at a subtree of a \rrt:
as an increasing tree, as an
unlabelled ordered tree (ignoring the labels but keeping the order defined
by them), or as an unlabelled unordered rooted tree (by ignoring both labels
and ordering). 

The theorems and other results
below, and their proofs, apply (unless explicitly stated otherwise)
to all these interpretations.
(The different interpretations may be useful in different applications.)
For convenience, we state most results for the unlabelled versions (which
seem to be more common in applications), and leave the versions with
increasing trees to the reader.
Recall also that an ordered tree can be regarded as unordered by ignoring
the orderings.
\end{rem}

For simplicity, we consider 
first only the sizes of the fringe trees.
The results in the following two theorems, 
except the explicit rate in \eqref{Poissonbinary}--\eqref{Poissonrecursive},
were shown by \citet{FengMahmoudP08} and Fuchs \cite{Fuchs} by using variants of
the method of moments. 
\refT{main2} was earlier proved for fixed $k$ by 
\citet{Devroye1} (using the central limit theorem for $m$-dependent
variables), 
and the means \eqref{munk}--\eqref{murrt} are implicit in
\cite{Devroye1},
see also \cite{Devroye2} and \citet{FlajoletGM1997}.
(The corresponding, weaker, laws of large numbers were also given by
\citet{Aldous-fringe} by another method.)
The part \eqref{main2b} for binary search trees was extended to
$k=k_n$ (for a smaller range than here)
by Devroye
\cite{Devroye2} using Stein's method.
In the present paper we continue and extend
this approach, and use Stein's method for both 
Poisson and normal approximations  to provide simple proofs for the full range.

We state the main results in this section. Proofs are given in later
sections.
We let $ \mathcal{L} (X)$ denote the distribution of a random variable $ X $.
$\Po(\mu)$ denotes the Poisson distribution with mean $\mu$, and 
$\N(0,1)$ the standard normal distribution.
Convergence in distribution is denoted by $\dto$.
We recall also the definition of the total variation distance   between two
probability measures.  
\begin{defn}
Let $(\mathcal{X},\mathcal{A})$ be any measurable space. The total variation
distance $ d_{TV} $   between two probability measures $ \mu_{1} $ and $
\mu_{2} $ on $ \mathcal{X} $ is defined to be  
$$ 
d_{TV}(\mu_{1},\mu_{2}) :=\sup_{A\in \mathcal{A}}| \mu_{1}(A)-\mu_{2}(A)|.$$
\end{defn}

\begin{thm}\label{main1}
Let $ X_{n,k} $ be the number of subtrees of size $k$ in the random binary
search tree $ \cT_n $  and similarly let $ \hX_{n,k} $ be the number of
subtrees in the random recursive tree $ \Lambda_n $. Let  
$ k=k_n $ where $ k<n $.
Furthermore, let
$\mu_{n,k}:=\E(X_{n,k} )$ and $  \hmu_{n,k}:=\E(\hX_{n,k} )$. Then
\begin{align}
  \mu_{n,k}&:=\E(X_{n,k} ) =\frac{2(n+1)}{(k+1)(k+2)}, \label{munk}\\
  \hmu_{n,k}&:=\E(\hX_{n,k} ) =\frac{n}{k(k+1)}.  \label{murrt}
\end{align}
Then,  for the binary search tree,
\begin{align} \label{Poissonbinary} d_{TV} 
(\mathcal{L}(X_{n,k}),\Po(\mu_{n,k}))=\frac{1}{2}\sum_{l\geq 0}
\Bigl|\P(X_{n,k}=l)-e^{-\mu_{n,k}}\frac{(\mu_{n,k})^l}{l!}\Bigr|
=O\Bigpar{\frac1k},
\intertext{and for the random recursive tree,}  
\label{Poissonrecursive}
d_{TV}(\mathcal{L}(\hX_{n,k}),\Po(\hat{\mu}_{n,k}))
=\frac{1}{2}\sum_{l\geq 0}
\Bigl|\P(\hX_{n,k}=l)-e^{-\hat{\mu}_{n,k}}\frac{(\hat{\mu}_{n,k})^l}{l!}\Bigr|
=O\Bigpar{\frac1k}.
\end{align}

Consequently, if $ n\rightarrow\infty $ and
$ k\rightarrow\infty $, then 
$d_{TV}(\mathcal{L}({X}_{n,k}),\Po({\mu}_{n,k}))\to0$
and
$d_{TV}(\mathcal{L}(\hX_{n,k}),\Po(\hat{\mu}_{n,k}))\to0$
.
\end{thm}

\begin{thm}\label{main2}Let $ X_{n,k} $ be the number of subtrees of size
  $k$ in the binary search tree $ \cT_n $ and similarly let $ \hX_{n,k} $ be
  the number of subtrees of size $k$ in the random recursive tree $
  \Lambda_n $. Let  
$ k=k_n=o(\sqrt{n})$. Then, as $n\to \infty$,
 for the binary search tree 
 \begin{align}
\frac{X_{n,k}-\E(X_{n,k})}{\sqrt{\Var(X_{n,k})}}
&\dto \mathcal{N}(0,1) \label{main2b}
\intertext{and, similarly, for the random recursive tree}    
\frac{\hX_{n,k}-\E(\hX_{n,k})}{\sqrt{\Var(\hX_{n,k})}}
&\dto\mathcal{N}(0,1).\label{main2r}
 \end{align}
 \end{thm}
 
 \begin{rem}
   If $k/\sqrt n\to\infty$, then $\mu\nk,\hmu\nk\to0$, and the convergence
   result in \refT{main1} reduces to the trivial 
$X_{n,k}\pto0$ and 
$\hX_{n,k}\pto0$; the rate of convergence in
\eqref{Poissonbinary}--\eqref{Poissonrecursive} is still of interest.
\citet{DennertGrubel} considered instead the sum $\sum_{k\ge (1-t)n}
X_{n,k}$ and obtained a functional central limit theorem.

If $k/\sqrt n\to c\in(0,\infty)$, then $\mu\nk\to 2c^{-2}$ and $\hmu\nk\to
c^{-2}$;   and we obtain the Poisson distribution limits 
$X_{n,k}\dto\Po(2c^{-2})$ and 
$\hX_{n,k}\dto\Po(c^{-2})$
\cite{FengMahmoudP08,Fuchs}.
 \end{rem}

 \begin{rem}
The proofs yield immediately, using the classical Berry--Esseen estimate for
Poisson  distributions,  also the estimate 
$O(k/\sqrt n +1/k)$
of the convergence rates in \eqref{main2b}
   and \eqref{main2r}
for the Kolmogorov distance; 
however, for slowly growing $k$ this is inferior to
   the bound $O(k/\sqrt n)$  given by Fuchs
   \cite{Fuchs}. (We have not investigated whether  
this bound can be shown by a more careful application of Stein's method.)
Other distances might also be studied,
   but we have not done so.
 \end{rem}

It is also relevant to study  trees of a fixed size with certain
properties. For example in evolutionary biology, it is important to study
such \emph{tree patterns} in phylogenetic trees. 
A \emph{phylogenetic tree}, more precisely a \emph{cladogram}, is a
(non-ordered) tree where every node has outdegree 2 (internal nodes) or 0
(external nodes). A binary tree of size $n$ 
yields a phylogenetic tree with $n$ internal nodes by adding $n+1$ external
nodes. 
An important model for a random phylogenetic tree is the \emph{Yule model}, 
which gives the same distribution as the correspondence just described 
applied to a random binary search tree,
see \eg{} \cite{Aldous-cladograms} and \cite{Blum}.  
Hence, fringe trees in  random phylogenetic trees under the Yule model
correspond to fringe trees in random binary search trees, and our results
can be translated.
(Note that the size of a phylogenetic tree usually is defined as the number
of external nodes; a phylogenetic tree of size $k$ thus corresponds to a
binary search tree with $k-1$ nodes.)
Some important examples of tree patterns that have been studied are
\emph{k-pronged nodes} (trees of size $ k $),
\emph{k-caterpillars} (trees of size $ k $ such that the internal nodes form
a path) and  \emph{minimal clade size k} (trees of size $ k $ with either
left or right subtree of the root empty); see e.g., \cite{Chang} and the
references there.

\citet{Chang} studied fringe trees in random phylogenetic trees.
The following theorem (the binary search tree case)
improves the convergence rate in their Theorem 9
and yields in a simple way the rate stated in their Remark 1.
By a property $P$, 
in the binary tree case
we formally mean any set of binary trees; 
we let $P_k$ be the set of binary trees of size $k$ in $P$ and 
we let $\pkp:=\P(\cT_k\in P)$.
Similarly, in the random recursive tree case,
a property $P$ is any set of ordered trees,
and $\prkp:=\P(\gL_k\in P)$.
(As said in \refR{R0}, we may also, more generally, let $P$ be a set of
increasing binary or unordered trees.) 

\begin{rem}\label{Rpkp}
If $ A_k ^{P}$ is the set of permutations of length $ k $ that give rise
to binary search trees of size $ k $ with the property $ P $, then 
%$p_{k,P} $ is equal to  
$p_{k,P}=\xfrac{| A_{k} ^{P}| }{k!}$. In particular, if $P_k$ is nonempty,
then $1\ge p_{k,P} \ge 1/k!$.
Similarly, 
$1\ge \hat p_{k,P} \ge 1/(k-1)!$ unless 
$\hat p_{k,P}=0$.
\end{rem}

\begin{thm}\label{poissontreepattern}
Let $ X_{n,k}^{P} $ be the number of subtrees of size $k$ with some given
property $ P $ in the binary search tree $ \cT_n $, and similarly let $
\hX_{n,k}^{P}  $ be the number of subtrees of size $k$ with some
given property $ P $ in the random recursive tree $ \Lambda_n $.  Let $
p_{k,P} $ be the probability that a binary search tree of size $ k $ has
property $ P $, and similarly let $ \hat{p}_{k,P} $ be the probability that
a random recursive tree of size $ k $ has property $ P $.
Let $ k=k_n $ where $ k<n $.
Furthermore, let 
$\mu_{n,k}^P:=\E(X_{n,k}^P)$ and $  \hmu_{n,k}^P:=\E(\hX_{n,k}^P)$. Then
\begin{align}
  \mu_{n,k}^{P}&:=\E(X_{n,k}^{P} ) =\frac{2(n+1)p_{k,P}}{(k+1)(k+2)}, 
\label{ptp}\\
  \hmu_{n,k}^{P}&:=\E(\hX_{n,k}^{P} ) = \frac{n\hat{p}_{k,P}}{k(k+1)}. 
\label{ptpr}
\end{align}
Then,  for the binary search tree, if $ k\neq (n-1)/2 $,
\begin{equation} %\label{Poissonbinary2} 
  \begin{split}
d_{TV}(\mathcal{L}(X_{n,k}^{P}),\Po(\mu_{n,k}^{P}))
&=\frac{1}{2}\sum_{l\geq 0}
\Bigl|\P(X_{n,k}^{P}=l)-e^{-\mu_{n,k}^{P}}\frac{(\mu_{n,k}^{P})^l}{l!}\Bigr| 
\\&=
\begin{cases}
O\Bigpar{\frac{p_{k,P}}{k}} & \text{if}\quad  \mu_{n,k}^{P}\geq 1 
\\
O\Bigpar{\frac{p_{k,P}}{k}\cdot \mu_{n,k}^{P}} & \text{if}\quad
		\mu_{n,k}^{P} < 1,
\end{cases}
  \end{split}
\end{equation}
and if $ k=(n-1)/2 $,
\begin{equation}\label{specialpoissonpatternbinary}
d_{TV}(\mathcal{L}(X_{n,k}^{P}),\Po(\mu_{n,k}^{P}))=O\Bigpar{\frac{p_{k,P}^{2}}{k}}.
\end{equation}
For the random recursive tree,
  \begin{equation}%\label{Poissonrecursive2} 
	\begin{split}
d_{TV}(\mathcal{L}(\hX_{n,k}^{P}),\Po(\hat{\mu}_{n,k}^{P}))
&=\frac{1}{2}\sum_{l\geq 0}
\Bigl|\P(\hX_{n,k}^{P}=l)-e^{-\hat{\mu}_{n,k}^{P}}\frac{(\hat{\mu}_{n,k}^{P})^l}{l!}\Bigr|
\\&=
\begin{cases}
 O\Bigpar{\frac{\hat{p}_{k,P}}{k}} & \text{if}\quad  \hat{\mu}_{n,k}^{P}\geq 1 
\\
O\Bigpar{\frac{\hat{p}_{k,P}}{k}\cdot \hat{\mu}_{n,k}^{P}} &
		\text{if}\quad \hat{\mu}_{n,k}^{P} < 1.  
\end{cases}	  
	\end{split}
  \end{equation}

Consequently, if $ n\rightarrow\infty $ and
$ k\rightarrow\infty $ then
$d_{TV}(\mathcal{L}({X}_{n,k}^{P}),\Po({\mu}_{n,k}^{P}))\to0$ and similarly
% "similarly" endast pga radbrytning
$d_{TV}(\mathcal{L}(\hX_{n,k}^{P}),\Po(\hat{\mu}_{n,k}^{P}))\to0$
.
\end{thm}

Note that this theorem extends \refT{main1}, which is the case when $P$ is
the set of all trees.

\begin{rem}  \label{RPnorm}
\refT{poissontreepattern} 
implies asymptotic normality in all cases when $k\to\infty$ and
$\mu^P_{n,k}\to\infty$ or $\hat\mu^P_{n,k}\to\infty$. 
Asymptotic normality holds for $k=O(1)$ too, see 
Examples \ref{Epattern} and \ref{EPk} %\refT{TF} 
below.   
For the binary tree, the asymptotic normality in these cases 
was proved by 
\citet[Theorem 1]{Devroye1} ($k$ fixed) and 
\cite[Theorem 5]{Devroye2}
(at least for $k=o(\log n/\log\log n)$ which implies that
$\mu^P_{n,k}\to\infty$ for every $P$ unless $P_k$ is empty).
\end{rem}

So far we have considered subtrees of one size $k=k_n$ only, but we have
allowed the size to depend on $n$. In the remainder of this section we
consider subtrees of different sizes together, 
giving result on joint asymptotic normality for several sizes and
properties; however, we do not allow the sizes to depend on $n$.

An important example of $ X_{n,k}^{P} $ is the number of subtrees 
of the binary search tree $ \cT_n $
that are copies of a fixed binary tree $ T $,
which we
denote by $X^{T}_{n}$; similarly we denote by $\hX^{\Lambda}_{n}$ the
number of copies of an ordered (or unordered)
tree $ \Lambda $ in the random recursive tree 
$\Lambda_n $.  
\refT{multivariate} below shows that these numbers are asymptotically
normal, and moreover, jointly so for different trees $T$ or $\gL$.
(For a single binary tree ${T}$ this was shown by
\citet{Devroye1,Devroye2}, % see Theorem 1 in Devroye1
and by another method by \citet{FlajoletGM1997}; 
for a single unordered tree ${\Lambda}$ this was shown by \citet{FengMahmoud}.)
Before stating the theorem, we give (exact) expressions for the covariances
between these numbers. 
(The variances, \ie{} the special case $T=T'$, in the binary case were
found by \cite{FlajoletGM1997}.)
As said in \refR{R0}, these results hold also if $T$ or $\gL$ is a given
increasing tree, and for the \rrt{} we may let $\gL$ be either an ordered or
unordered tree; the results are valid for all cases, but note that \eg{}
$\hp_{k,\gL}$ and $\hq_{\gL'}^{\gL}$ depend on the version. 
(In particular, for increasing trees $T$ and $\gL$, $p_{k,T}=1/k!$ and
$\hp_{k,\gL}=1/(k-1)!$.)
\begin{thm}\label{Tcovbin} 
Let $T$ be a binary tree of size $ k $ and let $ T' $ be a binary tree of
size $m$ where $ m\leq k $.  
Let 
$ \pkt:=\P(\cT_k=T) $ and $\pmtx:=\P(\cT_m=T')$,
%be the probability that a binary search tree of size $ k $ is equal to the
%tree $ T $.
and let
$\qtt$
be the number of  subtrees of\/ $ T $ that are
copies of\/ $ T' $; further, let
\begin{multline}\label{gamma}
  \bgam(k,m):=
\frac{4(k+m+3)}{(k+1)(k+2)(m+1)(m+2)}
\\{}- \frac {4({k}^{2}+3 km+{m}^{2}+4 k+4 m+3)}
{ \left( k+1 \right) \left( m+1 \right)  
\left( k+m+1 \right) \left( k+m+2 \right) \left( k+m+3 \right) }
.
\end{multline}
If\/ $n>k+m+1$, then
the covariance between $X^{T}_{n}$ and $ X^{T'}_{n} $ is equal to
\begin{align}\label{covbin}
\Cov(X^{T}_{n},X^{T'}_{n})&=
(n+1)\gstt,
\end{align} 
where
\begin{equation}\label{covbin2}
  \gstt:=\frac{2}{(k+1)(k+2)}\qtt\pkt
-\bgam(k,m)\pkt\pmtx.
\end{equation}
\end{thm}

We note also the corresponding result for $X_{n,k}$, combining all subtrees
of the same size. 
(The variances $\gsxx_{k,k}$ are given by \citet{FengMahmoudP08}, as well as
higher moments, and the
covariances are given by \citet{DennertGrubel}.)

\begin{thm}[\citet{DennertGrubel}]\label{Tcovbin2} 
Let $k,m\ge1$ and suppose $n>k+m+1$.
The covariance between $X_{n,k }$ and $ X_{n,m} $ 
is equal to
\begin{align}\label{covksubtrees}&
\Cov(X_{n,k},X_{n,m})=
(n+1)\gskm
\end{align}
where
$\gskm=\gsxx_{m,k}$ and
\begin{align}\label{covb2a}
\gskm &=
-{\frac {4m \left( 2k+m+3 \right) }
{(k+1) \left( k+2 \right)  \left( k+m+1 \right)  \left( k+m+2 \right)
  \left( k+m+3 \right) }},
\qquad m<k,
\\
\gsxx_{k,k} &=
{\frac {2k \left( 4{k}^{2}+5k-3 \right)}
{ \left( k+1 \right) \left( k+2 \right) ^{2} \left( 2k+1 \right)
 \left( 2k+3 \right) }}.
\label{covb2b}
\end{align}
\end{thm}

For the random recursive tree we have similar results.
 
\begin{thm}\label{Tcovrec} 
Let $\Lambda$ be an ordered [or unordered] tree of size $ k$, and let $
\Lambda' $ be an ordered  [or unordered] tree of size $m$ where $ m\leq k $.  

Let 
$ \hat{p}_{k,\Lambda}:=\P(\Lambda_k= \Lambda) $ and $\hat{p}_{m,\Lambda'}:=\P(\Lambda_m= \Lambda')$,
and let
$\hq_{\Lambda'}^{\Lambda}$
be the number of  subtrees of\/ $ \Lambda' $ that are
copies of\/ $ \Lambda$; further, let
\begin{align}\label{gamma4}
\hat{\bgam}(k,m):=\frac { k^2 + k m + m^2+k+m}
{  k (k+1) m ( m+1) (k + m+1)  }
.\end{align}

If\/ $n>k+m$, then
the covariance between $\hX^{\Lambda }_{n}$ and $ \hX^{\Lambda'}_{n} $ is equal to
\begin{align}\label{covariancerecursive2}
\Cov(\hX^{\Lambda}_{n},\hX^{\Lambda'}_{n})&=
n\hat{\gs}_{\Lambda,\Lambda'},
\end{align} 
where
\begin{equation}\label{covbin3}
\hat{\gs}_{\Lambda,\Lambda'}:=\frac{1}{k(k+1)}\hq_{\Lambda'}^{\Lambda}
 \hat{p}_{k,\Lambda}
-\hat{\bgam}(k,m)\hat{p}_{k,\Lambda}\hat{p}_{m,\Lambda'}.
\end{equation}
\end{thm}

\begin{rem}
Note that using the natural correspondence  between ordered trees 
and binary trees, if $\Lambda$ corresponds to the binary tree $T$ of size
$k-1$, then  $ \hat{p}_{k,\Lambda}=p_{k-1,T} $. 
\end{rem}

\begin{thm}\label{Tcovrec2}
Let $k,m\ge1$ and suppose $n>k+m$.
The covariance between $\hX_{n,k }$ and $ \hX_{n,m} $
is equal to
\begin{align}\label{covksubtreesrecursive}&
\Cov(\hX_{n,k},\hX_{n,m})=
n\hat{\gs}_{k,m}
\end{align}
where
$\hat{\gs}_{k,m}=\hat{\gs}_{m,k}$ and
\begin{align}\label{covb2c}
\hat{\gs}_{k,m} &=
{-\frac {1}
{k\left( k+1 \right) 
  \left( k+m+1 \right) }},
\qquad m<k,
\\
\hat{\gs}_{k,k} &=
{\frac {2 k^2-1}
{ k\left( k+1 \right)^2 \left( 2k+1 \right) }}.
\label{covb2d}
\end{align}
\end{thm}

\begin{thm}\label{multivariate}%
  \begin{thmenumerate}
  \item
%Let $ \cT_n $ be a random binary search tree with $ n $ nodes.
%Let $X^{T}_{n}$ be the number of subtrees $ T $ in the random binary search
% tree $ \cT_n $. 
Let 
$T^1,\dots,T^d$ be a fixed sequence of distinct binary trees and let
$\bX_{n}=(X^{T^{1}}_{n},X^{T^{2}}_{n},\dots,X^{T^{d}}_{n})$.
Let $$ 
\bmu_{n}:=\E\bX_n= \lrpar{\E(X^{T^{1}}_{n}),\E(X^{T^{2}}_{n}),\dots,
\E(X^{T^{d}}_{n})} 
$$ 
and let\/ $\Gamma=(\gamma_{ij})_{i,j=1}^d$ denote the matrix with elements 
\begin{equation}\label{mv0}
\gamma_{ij}
=\lim_{n\rightarrow \infty}\frac{1}{n}\Cov(X^{T^{i}}_{n},X^{T^{j}}_{n})
=\gs_{T^{i},T^{j}},
\end{equation}
with notation as in 
\eqref{covbin}--\eqref{covbin2}. 
Then $\Gamma$ is non-singular and
\begin{equation}\label{mv1}
n\qqw (\bX_{n}-\bmu_{n}) \dto \N(0, \Gamma ).
\end{equation}

\item
%Let $ \Lambda_n $ be a random recursive tree with $ n $ nodes.
%Let $\hX^{\Lambda}_{n}$ be the number of subtrees $ \Lambda $ in the
%random recursive tree $ \Lambda_n $. 
Similarly,
let $\gL^1,\dots,\gL^d$ be a fixed sequence of distinct ordered 
(or unordered) trees and let
$\hbX_{n}
=(\hX^{\Lambda^{1}}_{n},\hX^{\Lambda^{2}}_{n},\dots,\hX^{\Lambda^{d}}_{n})$. 
Let 
\begin{equation*}
\hbmu_{n}:=
\E\hbX
=\lrpar{\E(\hX^{\Lambda^{1}}_{n}),\E(\hX^{\Lambda^{2}}_{n}),\dots,
\E(\hX^{\Lambda^{d}}_{n})} 
\end{equation*} 
and let\/ 
$ \hat{\Gamma}=(\hat\gamma_{ij})_{i,j=1}^d$ denote the matrix with elements
\begin{equation}
\hat{\gamma}_{ij}
=\lim_{n\rightarrow  \infty}
\frac{1}{n}\Cov\bigpar{\hX^{\Lambda^{i}}_{n},\hX^{\Lambda^{j}}_{n}}
=\hgs_{\gL^i,\gL^j}
%=(k-m-1)\delta^{*}_{k} \delta^{*}_{m}+2\gamma^{*}_{k,m}-
%(m+k+1)\delta^{*}_{k}\delta^{*}_{m},
\end{equation}
with notation as in \eqref{covariancerecursive2}--\eqref{covbin3}.
Then $\hat\Gamma$ is non-singular and 
\begin{equation}\label{rrtmv1}
n\qqw \bigpar{\hbX_{n}-\hbmu_{n}} \dto \N(0, \hat\Gamma ).
\end{equation}
  \end{thmenumerate}
 \end{thm}

For \bst{s}, \eqref{mv1} 
can be proved as the univariate case in \citet{Devroye1}, but the formula
\eqref{covbin2} for the covariances seems to be new.
For random recursive trees, 
as said above,
\citet{FengMahmoud} showed 
the univariate case $d=1$ of \eqref{rrtmv1} (for unordered $\gL$), 
together with formulas for the
mean and variance. 

\begin{rem}
Since the covariance matrices $\gG$ and $\hat\gG$ in Theorem \ref{multivariate} 
are non-singular, 
the limiting multivariate normal distributions $\N(0, \Gamma )$ and 
$\N(0,\hat{\Gamma} )$ are non-degenerate. 
Furthermore, 
let $ \Cov (\bX_n)$ denote the covariance matrix of 
$\bX_n$.
Since $n\qw \Cov (\bX_n)\to \gG$ as $n\to\infty$,
$\Cov (\bX_n)$ is non-singular
for large enough $n$ and thus
$\Cov(\bX_n)^{-1/2}$ exists and the conclusion
\eqref{mv1} is equivalent to
$\Cov(\bX_n)^{-1/2}
(\bX_n-\bmu_n) \dto \N(0, I_d)$,
where $I_d$ is the $d\times d$ identity matrix and $\N(0,I_d)$ is the $d$-dimensional
standard normal distribution with $d$ \iid{} $\N(0,1)$ components.
Similarly, 
\eqref{rrtmv1} is equivalent to 
$\Cov(\hbX_n)^{-1/2}
(\hbX_n-\bmu_n) \dto \N(0, I_d)$. 
\end{rem}

\begin{example}\label{Epattern}
For any property $P$ of binary trees and any fixed $k$, $
X_{n,k}^{P}=\sum_{T\in P_k} X^{T}_n $, 
summing over all trees $T\in P_k$; hence the joint asymptotic normality in
\refT{multivariate} implies asymptotic normality of $ X_{n,k}^{P}$, as
asserted in \refR{RPnorm}. 
Moreover, 
this also yields joint asymptotic normality for several properties
$P$; in particular, we obtain joint asymptotic normality of $X_{n,k}$ for
any finite set of $k$, as earlier shown by \citet{DennertGrubel}.
The random recursive tree case  is similar. 
\end{example}

\refE{Epattern} generalizes immediately to any finite linear combination of
subtree counts $X^T_n$ or $\hX^T_n$; in fact, this is equivalent to the
joint asymptotic normality.
Using a truncation argument, this can be further extended as follows.
%(We describe the binary search tree case in detail. The case of recursive
%trees is the same, with trivial modifications.)

Let $f$ be a functional,
\ie,  a real-valued  function,
 of (binary, ordered or unordered) rooted trees, 
(Again, we may also, more generally, consider functionals of increasing trees.)
For a  tree $T$, 
let $T(v)$ be the fringe tree rooted at the node $v\in T$, and
define the sum over all fringe trees
\begin{equation}\label{F}
  F(T)=F(T;f):=\sum_{v\in T} f(T(v)).
\end{equation}
\begin{rem}\label{RF}
Functionals $F$ that can be written as \eqref{F} for some $f$ are called
\emph{additive functionals}. They can also be defined recursively by
\begin{equation}\label{rf}
  F(T)=f(T)+F(\ttt1)+\dots+F(\ttt d),
\end{equation}
where $\ttt1,\dots,\ttt d$ are the subtrees rooted at the children of the root
of $T$.
In this context, 
$f(T)$ is often called a \emph{toll function}.  
\end{rem}

We consider the random variables
$F(\ctn)$ and $F(\gln)$, where as above $\ctn$ and $\gln$
are the binary search tree and random recursive tree, respectively.
For example, if $f(T')=\ett{T'=T}$, the indicator
function that $T'$ equals some given binary tree $T$,
then $F(\ctn)=X^{T}_n$.
Conversely, for any $f$,
\begin{equation}\label{Fnt}
  F(\ctn)=\sum_{T}f(T) X^{T}_n,
\end{equation}
summing over all binary trees $T$.
As another example, $X_{n,k}^P=F(\ctn)$ with $f(T)=\ett{T\in P_k}$;
in particular, $X_{n,k}=F(\ctn)$ with $f(T)=\ett{|T|=k}$.
The recursive tree case is similar.
We refer to Devroye \cite{Devroye2} for several other examples showing
the generality of this representation, 
and for some special cases of the following result.

\begin{thm}\label{TF}
  Let $F$ be given by \eqref{F} for some functional $f$.
\begin{romenumerate}[-10pt] 
  \item 
For the binary search tree, assume that
\begin{align}
\sumk\frac{(\Var f(\cT_k))\qq}{k^{3/2}}&<\infty, \label{tfa1}
\\
\lim_{k\to\infty}\frac{\Var f(\cT_k)}k &=0,\label{tfa2}
\\
\sumk\frac{(\E f(\cT_k))^2}{k^2}& <\infty. \label{tfa3}
\end{align}
Then, as \ntoo, 
\begin{align}
 \E (F(\ctn))/n &\to \mu_F :=  \sumk \frac{2}{(k+1)(k+2)} \E f(\ctk), \label{tfe}
\\
 \Var (F(\ctn))/n &\to  \gss_F
:=\lim_{N\to \infty} \sum_{|T|,|T'|\le N} f(T)f(T')\gstt <\infty \label{tfv}
\end{align}
and
\begin{equation}\label{tfd}
\frac{ F(\ctn)-\E F(\ctn)}{\sqrt{n}}\dto \cN(0,\gss_F).
\end{equation}

\item 
For the random recursive tree, assume that
\begin{align}
\sumk\frac{(\Var f(\glk))\qq}{k^{3/2}}&<\infty, \label{tfa1r}
\\
\lim_{k\to\infty}\frac{\Var f(\glk)}k &=0,
\\
\sumk\frac{(\E f(\glk))^2}{k^2}& <\infty.\label{tfa3r}
\end{align}
Then, as \ntoo, 
\begin{align}
\E (F(\gln))/n & \to \hmu_F := \sumk \frac{1}{k(k+1)} \E f(\glk),  \label{tfer}
\\
\Var (F(\ctn))/n &\to  
\hgss_F:=
\lim_{N\to \infty} \sum_{|\gL|,|\gL'|\le N} f(\gL)f(\gL')\hgsll <\infty
\label{tfvr}
\end{align}
and
\begin{equation}\label{tfdr}
\frac{ F(\gln)-\E F(\gln)}{\sqrt{n}}\dto \cN(0,\hgss_F).
\end{equation}
  \end{romenumerate}
\end{thm}

\begin{cor}\label{CTF}
  Let $F$ be given by \eqref{F} for some functional $f$ 
such  that $f(T)=O(|T|^\ga)$ for some $\ga<1/2$.
Then the conclusions \eqref{tfe}--\eqref{tfd} and \eqref{tfer}--\eqref{tfdr}
hold. Furthermore,
the asymptotic normality \eqref{tfd}  can be written as
\begin{equation}\label{ctfd}
\frac{ F(\ctn)-n\mu_F}{\sqrt{n}}\dto \cN(0,\gss_F)
\end{equation}
and similarly,  \eqref{tfdr} can be written
\begin{equation}\label{ctfdr}
\frac{ F(\gln)-n\hmu_F}{\sqrt{n}}\dto \cN(0,\hgss_F).
\end{equation}
\end{cor}

\begin{rem}
For the binary search tree and $f(T)$ depending on the size $|T|$ only,
\refC{CTF}
was shown by \citet{HwangN} using the contraction method
(somewhat more generally, and with a somewhat different expression for
$\gss_F$),
they also show that, for example, $f(T)=|T|^\ga$
with $\ga>1/2$ yields different limit behaviour.
This  case of \refC{CTF} was also proved (by methods similar to the ones
used here) by 
\citet[Theorem 6]{Devroye2} under somewhat stronger hypotheses.
See further
\citet{FillKapur-mary} for similar results 
(extended to general $m$-ary search trees).
Cf.\ also \citet[Theorem 13]{FillFK} for related results for the mean.

A well-known case when $f$ grows too rapidly for the results above to hold 
is $f(T)=|T|$, when $F(T)$ is the total path length in the tree.
In this case, for the binary search tree,
the expectation grows like $2n\log n$
and the limit is non-normal, see \citet{Regnier}, \citet{Roesler},
\citet{FillJanson131}.
\end{rem}

\begin{rem}  
Of course, \eqref{tfv} means that (summing over all binary trees)
  \begin{equation}
\gss_F= \sum_{T,T'} f(T)f(T')\gstt, \label{xtfv}	
  \end{equation}
provided this sum is absolutely convergent. However,
this fails in general, even if $f$ is bounded,
since, as is shown in the appendix, 
  \begin{equation}
\sum_{T,T'}|\gstt|=\infty. \label{xb}	
  \end{equation}
Similarly, for the \rrt{} in \eqref{tfvr},
\begin{equation}
  \label{xr}
\sum_{\gL,\gL'}|\hgsll| =\infty.
\end{equation}
Hence, in general, we need the less elegant expression in \eqref{tfv}.
and \eqref{tfvr}.
The same applies to the special cases in
\eqref{cpv} and \eqref{cpvr} below.

Note that if $f(T)$ depends on the size $|T|$ only (a case considered in
\cite{HwangN} and \cite{FillKapur-mary}), so $f(T)=\mu_{|T|}$ for some
sequence $\mu_k$, $k\ge1$, then \eqref{tfv} implies
\begin{equation}
\gss_F= \sum_{k,m\ge1} \mu_k\mu_m\gskm, %\label{xtfv0}	  
\end{equation}
where it is easily shown that the sum is absolutely convergent
as a consequence of \eqref{covb2a}--\eqref{covb2b} and the
assumption \eqref{tfa3}, \ie{} $\sum_k
\mu_k^2/k^2<\infty$. The analogous result for the \rrt{} holds too for such
$f$, now using \eqref{covb2c}--\eqref{covb2d}.
\end{rem}

The asymptotic means $\mu_F$ and $\hmu_F$ 
in \eqref{tfe} and \eqref{tfer} can also be written  as follows.
Let $\cT$ be the random \bst{}  $\cT_N$ with random size $N$ such that
$\P(|\cT|=k)=\P(N=k)=\frac{2}{(k+1)(k+2)}$, $k\ge1$.
Similarly, 
let $\gLL$ be the random recursive tree  $\gL_N$ with random size $N$ such that
$\P(|\gLL|=k)=\P(N=k)=\frac{1}{k(k+1)}$, $k\ge1$. Then, by definition,
\begin{align}
  \mu_F &= \E f(\cT), \label{mub}
\\
\hmu_F&=\E f(\gLL) \label{mur}.
\end{align}
%(provided these expectations exist).
%This is not just a more compact way of writing the formulas.
Moreover, as shown by \citet{Aldous-fringe},
$\cT$ is the limit 
in distribution of a uniformly random fringe tree of $\cT_n$ as \ntoo, 
and similarly  $\gLL$ is the limit 
in distribution of a uniformly random fringe tree of $\gL_n$ as \ntoo, 
see also
\cite{Devroye1} and
\cite{DevroyeJanson}.
(In fact, this is an immediate consequence of \eqref{ptp} and \eqref{ptpr}.)

\citet{Aldous-fringe} gave also direct constructions of $\cT$ and $\gLL$
using branching processes. For $\gLL$ we consider a tree $\gLL_t$ growing
randomly in 
continuous time, starting with an isolated root at time $t=0$ and 
such that each existing node gets children according to a Poisson process
with rate 1. %(This is called the \emph{Yule tree}.)
For $\cT$ we similarly grow a random binary tree $\cT_t$
by letting each node get a
left and a right child after waiting times that are independent and $\Exp(1)$.
In both cases, we stop the process at a random time $\tau\sim\Exp(1)$,
independent of everything else; this gives $\gLL$ and $\cT$, see 
\cite{Aldous-fringe}. This construction often simplifies the calculation of
$\mu_F$ and $\hmu_F$, see 
\cite{DevroyeJanson} and examples in \refS{Sapp}.
($\gLL$ and $\cT$ can be regarded as increasing trees, using the  birth
times of the nodes as labels.)

\refC{CTF} shows, in particular, that $F(\ctn)$ or $F(\gL_n)$ 
is asymptotically normal for any bounded $f$, unless $\gss_F=0$ or $\hgss_F=0$.
Letting $f$ be the indicator function of a set of trees, we obtain the
following general result.
(In the binary case, \citet[Theorem~2]{Devroye2} showed \eqref{cpe} and 
the corresponding weak law of large numbers, which is a consequence of
\eqref{cp}.
See also \citet[Lemma~4]{Devroye2} for a result similar to \eqref{cp}.)

\begin{cor}\label{CP}
Let $P$ be any property of binary  trees 
and let $X_n^P$  be the number of subtrees of $T_n$ with this property.
Then, as \ntoo,  
\begin{align}
\E X_n^P/n&\to \mu_P:=\P(\cT\in P), \label{cpe}\\
\Var X_n^P/n &\to  
\gss_P:=\lim_{N\to \infty} \sum_{T,T'\in P:\;|T|,|T'|\le N} \gstt <\infty,
\label{cpv}
\end{align}
and
\begin{equation}\label{cp}
\frac{X_n^P-\E X_n^P}{\sqrt{n}}\dto \cN(0,\gss_P).
\end{equation}

Similarly,
if $P$ is any property of ordered (or unordered) trees
and $\hX_n^P$ is the number of subtrees of $\gL_n$
with this property,
then, as \ntoo,  
\begin{align}
\E \hX_n^P/n&\to \hmu_P:=\P(\gLL\in P), \label{cper}\\
\Var \hX_n^P/n &\to  
\hgss_P:=\lim_{N\to \infty} \sum_{\gL,\gL'\in P:\;|\gL|,|\gL'|\le N} \hgsll
<\infty, \label{cpvr}
\end{align}
and
\begin{equation}\label{cpr}
\frac{\hX_n^P-\E \hX_n^P}{\sqrt{n}}\dto \cN(0,\hgss_P).
\end{equation}
Furthermore, we can replace
$\E X_n^P$ in \eqref{cp}  
and $\E \hX_n^P$ in \eqref{cpr} by 
$n\mu_P$ and $n\hmu_P$, respectively.
\qed
\end{cor}

\begin{problem}\label{PQ0}
Is the asymptotic variance $\gss_F$ or $\hgss_F$ in \refT{TF} always
non-zero except in trivial cases when $F(\ctn)$ or $F(\gln)$ is deterministic?
(We conjecture so, but have no general proof.) 
Note that by \eqref{Fnt} and
the non-singularity of
the finite covariance matrices in \refT{multivariate}, this holds for any
$f$ such that $f(T)$ is non-zero only for finitely many $T$.
Another special case where this holds is given in \refT{TG1} below.

In particular, can $\gss_P=0$ or $\hgss_P=0$ occur in \refC{CP} except in
trivial cases when $\Var X_n^P=0$ or $\Var \hX_n^P=0$, respectively,
for every $n$?

Note that $F$ may be deterministic also when $f$ is not; for example, if
$f(T)$ equals the degree of the root of $T$ minus $1$, then $F(T)=-1$ for
any rooted tree $T$. (See also \refR{Rprot} for a related example
where different functionals $f$ yield the same $F$ for binary trees.)
\end{problem}

\begin{rem}\label{RTFjoint}
\refT{TF} extends immediately to
joint asymptotic normality for several functionals $f$ and $F$
by the \CramerWold{} device. Hence
Corollaries \ref{CTF} and \ref{CP} too extend to joint asymptotic normality.
\end{rem}

\begin{example}\label{EPk}
For any property $P$,   
\refC{CP} applied to $P_k$, or 
taking $f(T)=\ett{T\in P_k}$ in \refC{CTF} or in \refT{TF}, yields
again the asymptotic normality of $X_{n,k}^P$ and $\hX_{n,k}^P$
for fixed $k$, obtained more directly 
in \refE{Epattern}.
\end{example}

\begin{rem}
  Similar results for conditioned Galton--Watson trees are given in
  \cite{Janson285}. Note, however, that for the result corresponding to
  \refT{TF} there, stronger conditions on the size of $f$ are
  required than for the results above; 
in particular, \refC{CTF} holds in that setting only for $\ga<0$.
We  believe that, similarly, the analogue of \refC{CP} does not hold for
conditioned Galton--Watson trees for arbitrary properties, although we do
  not know any counter example.
\end{rem}

We note a special case where we can give an alternative formula for the
asymptotic variance $\gss_F$ or $\hgss_F$ and
prove the conjecture in Problem \ref{PQ0}.
(\refT{TG1}, for the \bst,  is
essentially the same as the case treated by \citet[Theorem $2'$]{HwangN},
with an equivalent formula for the variance,
except for the extra randomization allowed there.
It includes the case when $F(T)$ only depends on the size $|T|$, 
where the formula is the case $m=2$ of \citet[(5.3)]{FillKapur-mary}.
In this case, a very similar result was also proved by
\citet[Lemma~5]{Devroye2}. 
Another example where Theorems \ref{TG1}--\ref{TGr} apply is provided by the
2-protected nodes in \refS{SSprotected}.) 

For a rooted tree $T$,
%let $o$ be the root and 
let $v_1,\dots,v_d$ be  children of the root (in order if $T$
is an ordered tree), where $d=d(T)$ is the degree of the root.
We call the subtrees $T(v_1), \dots, T(v_d)$ \emph{principal subtrees}
of $T$. In the case of a binary tree $T$, we 
let $T_L$ and $T_R$ by the subtrees rooted at the left and right child of
the root, and call these the \emph{left} and \emph{right subtree}; these are
thus the principal subtrees, except that $T_L$ and $T_R$ may be the empty tree
$\emptyset$. 
(We define $\cT_0=\emptyset$ and $F(\emptyset)=0$.)

\begin{thm}
  \label{TG1}
Suppose, in addition to the hypotheses of \refT{TF}(i), that
$f(T)=f(|T|,|T_L|,|T_R|)$ depends only on the sizes of $T$ and of its left
and right subtrees.
Let $\nu_k:=\E F(\ctk)$, let $I_k$ be uniformly distributed on
\set{0,\dots,k-1} and let
\begin{equation}\label{psik}
  \begin{split}
  \psi_k&:=\Var\bigpar{\nu_{I_k}+\nu_{k-1-I_k}+f(k,I_k,k-1-I_k)}
\\&\phantom:
=\E\bigpar{\nu_{I_k}+\nu_{k-1-I_k}+f(k,I_k,k-1-I_k)-\nu_k}^2.	
  \end{split}
\end{equation}
Then
\begin{equation}\label{gssfpsi}
  \gss_F=\sum_{k=1}^{\infty}\frac{2}{(k+1)(k+2)}\psi_k <\infty.
\end{equation}
Moreover, $\gss_F>0$ unless $\Var F(\ctn)=0$ for every $n\ge1$;
this happens if and only if $f(n,k,n-1-k)=a_n-a_k-a_{n-1-k}$ for some
real numbers $a_n$, $n\ge0$.
\end{thm}

Note that $|T|=|T_L|+|T_R|+1$, so two of $|T|$, $|T_L|$, $|T_R|$ determine
the third; nevertheless we write $f(|T|,|T_L|,|T_R|)$ for emphasis.

\begin{thm}
  \label{TGr}
Suppose, in addition to the hypotheses of \refT{TF}(ii), that
$f(\gL)=f(|\gL|,d(\gL),|\gL_{v_1}|,\dots,|\gL_{v_{d(\gL)}}|)$ 
depends only on the
size $|\gL|$ 
and the number and sizes of the principal subtrees.
Let $\nu_k:=\E F(\gL_k)$, 
and let
\begin{equation}\label{psikr}
  \begin{split}
  \psi_k&:=\Var\biggpar{f(k,d(\gL_k),|\gL_{k,1}|,\dots)
 +\sum_{i=1}^{d(\gL_k)}\nu_{|\gL_{k,i}|}}
\\&\phantom:
=\E\biggpar{f(k,d(\gL_k),|\gL_{k,1}|,\dots)
 +\sum_{i=1}^{d(\gL_k)}\nu_{|\gL_{k,i}|}-\nu_k} ^2
.  \end{split}
\end{equation}
Then
\begin{equation}\label{gssfpsir}
  \hgss_F=\sum_{k=1}^{\infty}\frac{1}{k(k+1)}\psi_k <\infty.
\end{equation}
Moreover, $\hgss_F>0$ unless $\Var F(\gln)=0$ for every $n\ge1$;
this happens if and only if $f(n,d,n_1,\dots,n_d)=a_n-\sum_{i=1}^d a_{n_i}$
for some 
real numbers $a_n$, $n\ge0$.
\end{thm}

The distribution of 
$(d(\gL_k),|\gL_{k,v_1}|,\dots)$ in \eqref{psikr}
is the same as the distribution of the
number of cycles in a random permutation of length $k-1$
and their lengths (taken in the order of their minimal
elements), see \citet[Section 6.1.1]{Drmota}.

\section{Representations using uniform random variables}\label{Srep}

\subsection{Devroye's representation for the binary search tree}
\label{SSlinear}

We use the representation of the binary search tree $\cT_n$ by  Devroye
\cite{Devroye1,Devroye2} described in \refS{S:intro},  
using \iid{} random time stamps $ U_i\sim U(0,1)$ assigned
to the keys $ i=1,\dots,n $.
Write,
for $1\le k\le n$ and $1\le i \le n-k+1$,
\begin{equation}\label{gsik}
\sigma(i,k)=\{(i,U_i),\dots,(i+k-1,U_{i+k-1})\},  
\end{equation}
i.e., the sequence of $k$ labels $(j,U_j)$ starting with $j=i$.
For every node $u\in \cT_n$, the fringe tree $\cT_n(u)$ rooted at $u$
consists of the nodes with labels in a set $\gs(i,k)$ for some such $i$ and $k$,
where $k=|\cT_n(u)|$, 
but note that not every set $\gs(i,k)$ 
is the set of labels of the nodes of a fringe subtree;
if it is, we say simply that \emph{$\gs(i,k)$ is a subtree}.
We define the indicator variable
$$ 
\yi_{i,k} :=\etta\{\sigma(i,k) \subtn\}.
$$ 
It is easy to see that, for convenience defining $U_0=U_{n+1}=0$,
\begin{equation}\label{yik}
  \yi_{i,k}=\etta\bigset{\text{$U_{i-1}$ and $U_{i+k}$ are the two smallest among
$U_{i-1},\dots,U_{i+k}$}}.
\end{equation}
Note that if $i=1$ or $i=n-k+1$, this reduces to
\begin{align}\label{yik1}
  \yi_{1,k}&=\etta\bigset{\text{$U_{k+1}$ is the smallest among
$U_{1},\dots,U_{k+1}$}},
\\
 \yi_{n-k+1,k}&=\etta\bigset{\text{$U_{n-k}$ is the smallest among
$U_{n-k},\dots,U_{n}$}}. \label{yik2}
\end{align}
For $k=n$, when we only consider $i=1$, we have $\yi_{1,n}=1$.

Let $f(T)$ be a function from the set of (unlabelled) binary trees to
$\bbR$. We are interested in the functional, see \eqref{F},
\begin{equation}\label{xn}
X_n:=F(\ctn)=\sum_{u\in \cT_n} f(\cT_n(u)),   
\end{equation}
summing over all fringe trees of $\cT_n$.

Since a permutation $(\gs_1,\dots,\gs_k)$ defines a binary search tree (by
drawing the keys in order $\gs_1,\dots,\gs_k$), we can also regard $f$ as a
function of permutations (of arbitrary length).
Moreover, 
any set $\gs(i,k)$ defines a permutation
$(\sigma_1,\sigma_2,\dots,\sigma_k)$ where the values $j$,  $ 1\leq j\leq
k $, are ordered according to the order of $ U_{i+j-1} $. 
We can thus also regard $ f $ as a mapping from the collection of all sets
$\sigma(i,k)$. Note that if $\gs(i,k)$ corresponds to a subtree $\cT_n(u)$ of
$\cT_n$, then %, ignoring labels, 
$\cT_n(u)$ is the binary search tree defined by
the permutation defined by $\gs(i,k)$, and thus
$f\bigpar{\cT_n(u)} = f\bigpar{\gs(i,k)}$. 
Consequently, see \cite{Devroye2}, 
\begin{equation}\label{linear}
X_n:=\sum_{u\in \cT_n} f(\cT_n(u))
=\sum_{k=1}^n\sum_{i=1}^{n-k+1} \yi_{i,k} f(\gs(i,k)).
\end{equation}

\subsection{The random recursive tree}\label{SSlinearrrt}
Consider now instead the random recursive tree $\gL_n$.
Let $f(T)$ be a function from the set of ordered rooted trees to $\bbR$.
(The case when $f$ is a functional of unordered trees is a special case,
and the case when $f$ is a functional of increasing trees is similar.)
In analogy with \eqref{xn}, 
we define
\begin{equation}\label{yn}
Y_n:=F(\gL_n)=\sum_{u\in \gL_n} f(\gL_n(u)),   
\end{equation}
summing over all fringe trees of $\gL_n$.

As said in the introduction, the natural correspondence yields a coupling 
between the random recursive tree $\gL_n$ and the binary search tree
$\cT_{n-1}$, 
where the subtrees in $\gL_n$ correspond to the left subtrees at the nodes
in $\cT_{n-1}$ together with the whole tree, including
an empty left subtree $\emptyset$ at every node in $\cT_{n-1}$ without a left
child, corresponding to a subtree of size 1 (a leaf) in $\gL_n$.
Thus, as noted by \cite{Devroye1}, the representation in \refS{SSlinear}
yields  a similar representation 
for the random recursive tree, which can be described as follows.

Define $\ff$ as the functional on binary trees corresponding to $f$ by
$\ff(T):=f(T')$, where $T'$ is the ordered tree corresponding to the binary
tree $T$ by the natural correspondence. (Thus $|T'|=|T|+1$.)
We regard the empty binary tree $\emptyset$ as corresponding to 
the (unique) ordered tree $\bullet$ with only one vertex, and thus we define
$\ff(\emptyset):=f(\bullet)$.

Assume first $1<k<n$ and 
recall that subtrees  of size $ k $ in the random recursive tree $\gL_n$
correspond to left-rooted subtrees of size $k-1$ in the binary search tree
$\cT_{n-1}$.  As said in \refS{SSlinear}, a subtree of size $k-1$ in $\cT_{n-1}$ 
corresponds to a set $\gs(i,k-1)$ for some $i\in\set{1,\dots,n-k+1}$.
The parent of the root of this subtree is either $ i-1 $ or $ i+k-1 $;
it is $i-1$, and the subtree is right-rooted, 
if $U_{i-1}>U_{i+k-1}$ and it is $i+k-1$, and the subtree is left-rooted, 
if $U_{i-1}<U_{i+k-1}$.
Thus, if we define
\begin{align}\label{ylik0}
  \yl_{i,k-1}
&:=\etta\set{\sigma(i,k-1) \text{ is a left-rooted subtree in $\cT_{n-1}$}},
\intertext{then, using \eqref{yik},}
  \yl_{i,k-1}
&\phantom:=
\etta\bigset{U_{i-1}\le U_{i+k-1}<\min_{i\le j\le i+k-2} U_j}. \label{ylik}
\end{align}
Note that, since we consider $\cT_{n-1}$, we have defined $U_0=U_{n}=0$, and
the argument above holds also in the boundary cases $i=1$ and $i=n-k+1$.
Furthermore, in the case $k=n$,
we define the whole binary tree as left-rooted, so
$\yl_{1,n-1}=1$ and \eqref{ylik} holds also for $k=n$ (and thus $i=1$).
(This is the reason for using a weak inequality $U_{i-1}\le U_{i+k-1}$ in
\eqref{ylik}; for $k<n$ we might as well require $U_{i-1}< U_{i+k-1}$ since
$U_0,\dots,U_{n-1}$ are assumed to be distinct.)

Finally, consider the case $k=1$. Subtrees of size 1 in $\gL_n$ correspond
to nodes without left child in $\cT_{n-1}$, and it is easily seen that a node
$i$ lacks a left child if and only if $U_i \ge U_{i-1}$. Hence, defining
$\yl_{i,0}:=\etta\bigset{i \text{ has no left child}}$,
\eqref{ylik} holds also for $k=1$ (with the empty minimum interpreted as
$+\infty$). 

Consequently, 
\eqref{ylik} holds for all $k$, and
the fringe trees in $\gL_n$ correspond to the 
sets $\gs(i,k-1)$ with $1\le k\le n$ and $1\le i\le n-k+1$ such that
$\yl_{i,k-1}=1$.
 It follows that, in analogy with \eqref{linear},
\begin{equation}\label{linearrrt}
Y_n:=\sum_{u\in \gL_n} f(\gL_n(u))
=\sum_{k=1}^n\sum_{i=1}^{n-k+1} \yl_{i,k-1} \ff(\gs(i,k-1)).
\end{equation}

Note that (for $k=1$) $\gs(i,0)=\emptyset$, the empty set corresponding to
the empty subtree $\emptyset$, and thus 
$\ff(\gs(i,0))=\ff(\emptyset)=f(\bullet)$.
Note also the boundary cases, because $U_0=U_n=0$,
\begin{align}\label{ylik1}
  \yl_{1,k-1}&=\etta\bigset{\text{$U_{k}$ is the smallest among
$U_{1},\dots,U_{k}$}},
\intertext{and}
\yl_{n-k+1,k-1}&=
\begin{cases}
  0, & 1\le k<n, \\
1, &k=n.
\end{cases}
\end{align}

\subsection{Cyclic representations}\label{SScyclic}
The representation \eqref{linear} of $X_n$ using a linear sequence
$U_1,\dots,U_n$ of \iid{} random variables is natural and useful, but
it has the (minor) disadvantage that terms with $i=1$ or $i=n-k+1$ have
to be treated  specially because of boundary effects, as seen in
\eqref{yik1}--\eqref{yik2}.
It will be convenient to use a related cyclic representation, where we take
$n+1$ \iid{} uniform variables $U_0,\dots,U_n\sim U(0,1)$ and extend them to
an infinite periodic sequence of random variables by
\begin{equation}\label{ukk}
  U_{i\phantom(}:=U_{i \bmod(n+1)},
\qquad i\in\mathbb Z,
\end{equation}
where $i\bmod (n+1)$ is  the remainder when $i$ is divided by $n+1$, i.e.,
the integer $\ell\in[0,n]$ such that $i\equiv \ell\pmod{n+1}$.
(We may and will assume that $U_0,\dots,U_n$ are distinct.)
We define further $\yi_{i,k}$
as in \eqref{yik}, but now for all $i$ and $k$.
%\begin{equation}\label{yyik}
%  \yy_{i,k}:=\etta\bigset{\text{$U_{i-1}$ and $U_{i+k}$ are the two smallest
%	  among $U_{i-1},\dots,U_{i+k}$}}.
%\end{equation}
Similarly, we define $\gs(i,k)$ by \eqref{gsik} for all $i$ and $k$.
We then have the following cyclic representation of $X_n$. 
(We are indebted to Allan Gut for suggesting a cyclic representation.)

\begin{Lemma}\label{Lcyclic}
Let $U_0,\dots,U_n\sim U(0,1)$ be independent and extend this sequence
periodically by
\eqref{ukk}. Then, with notations as above,
\begin{equation}
  \label{cyclic}
%X_n\eqd \xxn:= 
X_n:=\sum_{u\in \cT_n} f(\cT_n(u))
\eqd
\xxn:=
\sum_{k=1}^n\sum_{i=1}^{n+1} \yi_{i,k} f(\gs(i,k)).
\end{equation}
\end{Lemma}

\begin{proof}
The double sum in \eqref{cyclic} is invariant under a cyclic shift of
$U_0,\dots,U_n$. If we shift these values 
so that $U_0$ becomes the smallest, we obtain
the same distribution of $(U_0,\dots,U_n)$ as if we instead condition on the
event that $U_0$ is the smallest $U_i$, i.e., on $\set{U_0 = \min_i U_i}$.
Hence,
\begin{equation}
  \xxn \eqd \bigpar{\xxn \mid U_0=\min_i U_i}.
\end{equation}
Furthermore, the variables $\yi_{i,k}$  depend only on the order
relations among \set{U_i}, so if $U_0$ is minimal, they remain the same if
we put $U_0=0$. Moreover, in this case also $U_{n+1}=U_0=0$ and it follows
from \eqref{yik} that $\yi_{i,k}=0$ if $i\le n+1\le i+k-1$;
hence the terms in \eqref{cyclic} with $n-k+1 <i\le n+1$ vanish.
Note also that in the remaining terms, $f(\gs(i,k))$ does not depend on $U_0$.
Consequently,
\begin{equation}
%  \label{xxn}
\xxn\eqd
\Bigpar{
\sum_{k=1}^n\sum_{i=1}^{n-k+1} \yi_{i,k} f(\gs(i,k))\Bigm| U_0=0}
=X_n,
\end{equation}
by \eqref{linear}, showing that the cyclic and linear representations in
\eqref{linear} and \eqref{cyclic} are equivalent.
\end{proof}

\begin{rem}
In terms of the tree $\cT_{n}$,  the construction above means that we find
$i_0\in\set{0,\dots,n+1}$ such that $U_{i_0}$ is minimal and then construct
the tree $\cT_{n}$ from
the pairs $(1,U_{i_0+1}),\dots,(n,U_{i_0+n})$ by Devroye's construction.
\end{rem}

For the random recursive tree $\gL_n$ we argue in the same way, now using
\eqref{linearrrt}. We start with $n$ \iid{} uniform random variables
$U_0,\dots,U_{n-1}$ and extend them to a sequence with period $n$; we then
define $\gs(i,k-1)$ and $\yl_{i,k-1}$ by \eqref{gsik} and \eqref{ylik} for all
$i$ and $k$. This yields the following; we omit the details.

\begin{Lemma}\label{Lcyclicrrt}
Let $U_0,\dots,U_{n-1}\sim U(0,1)$ be independent and extend this sequence
periodically by $U_i:=U_{i\bmod n}$. Then, with notations as above,
\begin{equation}
  \label{cyclicrrt}
%X_n\eqd \xxn:= 
Y_n:=\sum_{u\in \gL_n} f(\gL_n(u))
\eqd
\yyn:=
\sum_{k=1}^n\sum_{i=1}^{n} \yl_{i,k-1} \ff(\gs(i,k-1)).
\end{equation}
\qed
\end{Lemma}

We may (and will) assume that the equalities in distribution in the lemmas
above are equalities.
\section{Means and variances}\label{Smean}

The cyclic representations in \refS{SScyclic}
lead to simple calculations of means and
variances.

\subsection{Random binary search tree}\label{Smeanbst}

We begin by computing the mean and variance of $X\nk$, the number of
subtrees of size $k$ in the random \bst{} $\cT_n$.
This has earlier been done using the linear representation in
\refS{SSlinear} by \citet{Devroye1} (implicitly) and \cite{Devroye2}
(explicitly);
%although some unimportant boundary terms were not evaluated exactly). 
our proof is very similar but the cyclic representation avoids the 
(asymptotically insignificant) boundary terms.
Explicit expressions have also been derived by other (analytic) methods,
see 
\citet{FengMahmoudP08}, %Also higher moments. BST + RRT
\citet{Chang},            %\marginal{Chang with $n-1,k-1$}
\citet{Fuchs,Fuchs2012}.  %\marginal{Fuchs2012 also d-ary}
We give a detailed proof for completeness, and as an introduction to later
proofs. (The lemma is a special case of later results, but we find it
convenient to start with the simplest case.)
For completeness, note also that $X_{n,k}=1$ when $k=n$ and $X_{n,k}=0$ when
$k>n$.

Note that $X\nk$ is given by \eqref{xn} with $f(T)=\etta\set{|T|=k}$, and
thus by \eqref{linear} with $f(\gs(i,\ell))=\etta\set{\ell=k}$, i.e.,
$ X\nk= \sum_{i=1}^{n-k+1} \yi_{i,k}$.
% \begin{equation}
%   \label{xnk}
% X\nk=
% \sum_{i=1}^{n-k+1} \yi_{i,k}.
% \end{equation}
However, we prefer to instead use the cyclic representation \eqref{cyclic}, 
which in this case is
\begin{equation}
  \label{xnk}
X\nk  =  %\eqd
\sum_{i=1}^{n+1} \yi_{i,k},
\end{equation}
where now $\yi_{i,k}$ are defined by \eqref{yik} with $U_i$ given by
\eqref{ukk}. 
Recall that $U_i$ thus is defined for all $i\in\bbZ$ and has period $n+1$; it
is thus 
natural to regard the index $i$ as an element of $\bbZ_{n+1}$; similarly,
$I_{i,k}$ is defined for all $i\in\bbZ$ with period $n+1$ in $i$, so we can
regard it as defined for $i\in\bbZ_{n+1}$. When discussing these variables,
we will use the natural metric on $\bbZ_{n+1}$ defined by 
\begin{equation}\label{dn+1}
 \absni{i-j}:=\min_{\ell\in\bbZ}|i-j-\ell(n+1)|. 
\end{equation}

\begin{Lemma}[Cf.~\citet{Devroye1,Devroye2} and \citet{FengMahmoudP08}]
\label{lemma1}
Let $1\le k<n $.
For the random binary search tree $\cT_n$,
\begin{align}
  \E(X_{n,k})&=\frac{2(n+1)}{(k+1)(k+2)} \label{exnk}
\end{align}
and %\intertext{and}
\begin{align}\label{vxnk0}
\Var(X_{n,k})&= 
\begin{cases}
\E X_{n,k}
%+ (n+1)\lrpar{\frac{10k+6}{(k+1)^2(2k+1)(2k+3)}-\frac{8k+12}{(k+1)^2(k+2)^2}},
-(n+1)\frac{22k^2+44k+12}{(k+1)(k+2)^2(2k+1)(2k+3)},
& k <\frac{n-1}2, \\
\E X_{n,k}
+ {\frac{2}{n}
-\frac{64}{(n+3)^2}},    % -\frac{16}{(k+2)^2}}  
& k =\frac{n-1}2, \\
\E X_{n,k} -(\E X_{n,k})^2
=
\E X_{n,k}
-\frac{4(n+1)^2}{(k+1)^2(k+2)^2}  ,
& k >\frac{n-1}2. \\
\end{cases}
\end{align}
Hence,
\begin{align}
\Var(X_{n,k})&=\E(X_{n,k})+O\Bigparfrac{n}{k^3}, \label{vxnk}
%\end{align}
\intertext{except when $k=(n-1)/2$; in this case}
%\begin{equation}
\Var(X_{n,k})
&=\E(X_{n,k})+\frac2n+O\Bigparfrac{n}{k^3} 
= \E(X_{n,k})+O\Bigparfrac1{n}
.
%= O\Bigparfrac{n}{k^2}.  
\label{vxnk2}
\end{align}
%$$\Var(X_{n,k})=n\frac{2k(4k^2+5k-3)(n+1)}{(k+2)(k+1)}+\frac{2}{k+1}.$$ 
\end{Lemma}

Another, equivalent, expression for the variance in the case $k<(n-1)/2$ is
given in \refT{Tcovbin2} with $m=k$.
(It is easily checked that when $n>2k+1$, \eqref{vxnk0} and
\eqref{covksubtrees} with \eqref{covb2b} are equivalent.)

\begin{proof}
We use \eqref{xnk}.
By \eqref{yik} and symmetry, for any $i$ and $1\le k<n$, 
\begin{equation}\label{eink}
 \E(\yi_{i,k})=\frac{2}{(k+2)(k+1)} 
\end{equation}
and thus \eqref{exnk} follows directly from \eqref{xnk}.

We now consider the variance. 
Note that by \eqref{yik},
$ \yi_{i,k}$ and $\yi_{j,k}$ are independent 
unless the sets $i-1,\dots,i+k$ and $j-1,\dots,j+k$ overlap modulo $n+1$,
%(i.e., regarded as subsets of $\bbZ_{n+1}$). 
i.e., unless $\absni{i-j}\le k+1$.
Furthermore,
if $0<\absni{i-j}\le k$, then \eqref{yik} implies
$ \yi_{i,k}\yi_{j,k}=0 $ (this says that two distinct subtrees of size $k$
are disjoint and, moreover, have their corresponding intervals of $k$
indices non-adjacent, which is obvious). 
Hence, by \eqref{xnk} and symmetry, if $k<(n-1)/2$,
\begin{equation} \label{varxnk}
  \begin{split}
\Var(X_{n,k})&=
\sum_{i=0}^n\sum_{j=0}^n \Cov(\yi_{i,k},\yi_{j,k})
\\
&=
(n+1)\Var(\yi_{0,k})+2(n+1)\sum_{j=1}^{k+1}\Cov(\yi_{0,k},\yi_{j,k})
\\ 
&=
(n+1)\Bigpar{\E\yi_{0,k} %-(\E\yi_{0,k})^2
+2\E(\yi_{0,k}\yi_{k+1,k})
-(2k+3)(\E\yi_{0,k})^2}.
  \end{split}
\end{equation}
If $k=(n-1)/2$ (and thus $n$ is odd) this has to be modified since 
$-(k+1)\equiv k+1\pmod{n+1}$, so the terms for $j-i=\pm(k+1)$ coincide and
should only be counted once; thus 
\begin{equation} \label{varxnk2}
  \begin{split}
\Var(X_{n,k})&=
(n+1)\Bigpar{\E\yi_{0,k} %-(\E\yi_{0,k})^2
+\E(\yi_{0,k}\yi_{k+1,k})
-(2k+2)(\E\yi_{0,k})^2}.
  \end{split}
\end{equation}
Finally, if $k>(n-1)/2$, then always $\yi_{i,k}\yi_{j,k}=0$ unless $i=j$
(there is not room for two distinct subtrees of size $k\ge n/2$) and 
\begin{equation} \label{varxnk3}
  \begin{split}
\Var(X_{n,k})&=
(n+1)\Bigpar{\E\yi_{0,k}-(n+1)(\E\yi_{0,k})^2}.
  \end{split}
\end{equation}
This can also be seen directly, since in this case $X_{n,k}\le1$, so
$X_{n,k}\sim\Be(\mu_{n,k})$ with $\mu\nk=\E X\nk=(n+1)\E I_{0,k}$.

It remains to compute $\E(\yi_{0,k}\yi_{k+1,k})=\E(\yi_{1,k}\yi_{k+2,k})$.
By \eqref{yik},
$\yi_{1,k}\yi_{k+2,k}=1$ when 
$U_0$ and $U_{k+1}$ are smaller than $U_1,\dots,U_{k}$ and 
$U_{k+1}$ and $U_{2k+2}$ are smaller than $U_{k+2},\dots,U_{2k+1}$.
Consider first $k<(n-1)/2$ and 
condition on $U_{k+1}=u$. Then the first condition is satisfied if either
$U_0<u$ and $U_1,\dots,U_k>u$, which has probability $u(1-u)^k$, or if
$U_0,\dots,U_k>u$ and $U_0$ is the smallest among them, which by symmetry has
the probability $\frac{1}{k+1}\P(U_0,\dots,U_k>u)=\frac{1}{k+1}(1-u)^{k+1}$.
The second condition has the same probability, and by independence we
obtain, letting $x=1-u$,
\begin{equation}\label{eikk}
  \begin{split}
\E(\yi_{1,k}\yi_{k+2,k})
&=
\int_0^1 \Bigpar{u(1-u)^k+\tfrac1{k+1}(1-u)^{k+1}}^2 \dd u
\\
&=
\int_0^1 \Bigpar{x^k-\tfrac{k}{k+1}x^{k+1}}^2 \dd x
=
\int_0^1
\Bigpar{x^{2k}-\tfrac{2k}{k+1}x^{2k+1}+\tfrac{k^2}{(k+1)^2}x^{2k+2}} \dd x
\\
&=
\frac{1}{2k+1}-\frac{2k}{(k+1)(2k+2)}+\frac{k^2}{(k+1)^2(2k+3)}
\\&=
\frac{5k+3}{(k+1)^2(2k+1)(2k+3)}.
   \end{split}
\raisetag{\baselineskip}
\end{equation}
(This can alternatively be obtain by a combinatorial argument, considering
the 6 possible orderings of $U_0, U_{k+1}, U_{2k+2}$ separately.)

In the case $k=(n-1)/2$, $U_{2k+2}=U_{n+1}=U_0$, and thus 
$I_{1,k}I_{k+2,k}=1$ if and only if $U_0$ and $U_{k+1}$ are the two smallest
among $U_0,\dots,U_n$; hence
\begin{equation}\label{eikk=}
\E(I_{1,k}I_{k+2,k})=\frac{2}{n(n+1)}.  
\end{equation}

The result \eqref{vxnk0} now follows from \eqref{eink}--\eqref{varxnk3}
by elementary calculations. 
Finally, \eqref{vxnk}--\eqref{vxnk2} follow.
\end{proof}

\refL{lemma1} is easily extended to $\xnkp$, the number of subtrees of size
$k$ with some property $P$.
(The mean and estimates of the variance are given by \citet{Devroye2}.
The special case when we count copies of a given tree $T$ was given by
\citet{FlajoletGM1997}.)

\begin{Lemma}\label{lemma1P}
Let $P$ be some property of binary trees.
Let $1\le k<n $ and let $\pkp:=\P(\cT_k\in P)$.
For the random binary search tree $\cT_n$,
\begin{align}
  \E(\xnkp)&=\frac{2(n+1)\pkp}{(k+1)(k+2)} \label{exnkP}
\end{align}
and
\begin{align}\label{vxnk0P}
\Var(\xnkp)&= 
\begin{cases}
\E \xnkp
%+ (n+1)\lrpar{\frac{10k+6}{(k+1)^2(2k+1)(2k+3)}-\frac{8k+12}{(k+1)^2(k+2)^2}},
-(n+1)\frac{22k^2+44k+12}{(k+1)(k+2)^2(2k+1)(2k+3)}\, \pkp^2,
& k <\frac{n-1}2, \\
\E \xnkp
+ \lrpar{\frac{2}{n}-\frac{64}{(n+3)^2}} \pkp^2,    % -\frac{16}{(k+2)^2}}  
& k =\frac{n-1}2, \\
\E\xnkp -(\E \xnkp)^2
=
\E \xnkp
-\frac{4(n+1)^2}{(k+1)^2(k+2)^2}\, \pkp^2  ,
& k >\frac{n-1}2. \\
\end{cases}
\end{align}
Hence,
\begin{align}
\Var(\xnkp)&=\E(\xnkp)+O\Bigpar{\frac{n}{k^3}\pkp^2}, \label{vxnkP}
%\end{align}
\intertext{except when $k=(n-1)/2$; in this case}
%\begin{equation}
\Var(\xnkp)
&= \E(\xnkp)+O\Bigpar{\frac1{n}\pkp^2}.  
\label{vxnk2P}
\end{align}
\end{Lemma}

\begin{proof}
  Let $\ipik$ be the indicator of the event that the binary search tree
  defined by the permutation defined by $\gs(i,k)$ belongs to $P$.
Then the cyclic representation
\refL{Lcyclic} with $f(T)=\etta\set{T\in P_k}$ yields
\begin{equation}\label{xpnk}
  \xnkp= %\eqd
\sumini \ixik \ipik.
\end{equation}

By \eqref{yik}, conditioning on $\iik=1$ says nothing about the relative
order of $U_i,\dots,U_{i+k-1}$; hence $\ixik$ and $\ipik$ are independent.
%Furthermore, $\E\ipik=\P(\cT_k \in P)=\pkp$.
Consequently, by \eqref{eink},
\begin{equation}\label{eixip}
  \E(\ixik\ipik)=\E(\ixik)\E(\ipik)
=\frac{2}{(k+1)(k+2)}\,\P(\cT_k\in P)
=\frac{2}{(k+1)(k+2)}\,\pkp,
\end{equation}
and \eqref{exnkP} follows immediately.

Similarly, 
for the variance we use \eqref{xpnk}, \eqref{eixip} 
and the argument in the proof of
\refL{lemma1}. Note that
$I^P_{0,k}$ and $I^P_{k+1,k}$ are independent of $I_{0,k}I_{k+1,k}$
and of each other; thus
\begin{equation*}
  \E\bigpar{\iq_{0,k}I^P_{0,k}\iq_{k+1,k}I^P_{k+1,k}}
=
  \E\bigpar{\iq_{0,k}\iq_{k+1,k}} \pkp^2.
%\qedhere
\end{equation*}
The result follows by simple calculations.
\end{proof}

To further extend this, we consider  a real-valued
functional $f(T)$ of binary trees and the sum $F(T)$ defined by
\eqref{F}. We begin with two such functionals of a special type.

\begin{Lemma}\label{Lkm}
Let $1\le m\le k$.
Suppose that
$f(T)$ and $g(T)$ are two functionals of binary trees such that
$f(T)=0$ unless $|T|=k$ and $g(T)=0$ unless $|T|=m$,
and let $F(T)$ and $G(T)$
be the corresponding sums \eqref{F} over subtrees.
Let 
\begin{align}
  \mu_f:=\E f(\cT_k) \qquad\text{and}\qquad \mu_g:=\E g(\cT_m).
\end{align}
%$\mu_f:=\E f(\cT_k)$ and $\mu_g:=\E g(\cT_m)$.
\begin{romenumerate}[-10pt]
\item\label{lkmi}
The means of $F(\cT_n)$ and $G(\cT_n)$ are given by
\begin{equation}\label{ekm}
  \E F(\cT_n)=
  \begin{cases}
\frac{2(n+1)}{(k+1)(k+2)}\mu_f, & n>k,\\
\mu_f, & n=k, \\
0, & n<k,
  \end{cases}
\end{equation}
and similarly for\/ $\E G(\cT_n)$.

\item \label{lkmii}
If\/ $n>k+m+1$, then
\begin{equation*}
  \begin{split}
  \Cov\bigpar{F(\cT_n),G(\cT_n)} 
&=(n+1)\lrpar{\frac{2}{(k+1)(k+2)}\E\bigpar{ f(\cT_k)G(\cT_k)}
-\bgam(k,m)\mu_f\mu_g}
%\sum_{j=0}^{n} \Cov\bigpar{I_{0,k} f(\gs(0,k)),I_{j,m} g(\gs(j,m))}
	\end{split}
  \end{equation*}
where $\bgam(k,m)$ is given by \eqref{gamma}. 
\item \label{lkmiii}
If\/ $n=k+m+1$, then
\begin{equation*}
  \begin{split}
  \Cov\bigpar{F(\cT_n),G(\cT_n)} 
&=(n+1)\lrpar{\frac{2}{(k+1)(k+2)}\E\bigpar{ f(\cT_k)G(\cT_k)}
-\bgam_1(k,m)\mu_f\mu_g}
%\sum_{j=0}^{n} \Cov\bigpar{I_{0,k} f(\gs(0,k)),I_{j,m} g(\gs(j,m))}
	\end{split}
  \end{equation*}
where
\begin{equation}\label{gamma1}
  \bgam_1(k,m):=
\frac{4(k+m+2)}{(k+1)(k+2)(m+1)(m+2)}
-
\frac{2}{n(n+1)}.
\end{equation}

\item \label{lkmiv}
If\/ $k<n<k+m+1$, then
\begin{equation*}
  \begin{split}
  \Cov\bigpar{F(\cT_n),G(\cT_n)} 
&=(n+1)\lrpar{\frac{2}{(k+1)(k+2)}\E\bigpar{ f(\cT_k)G(\cT_k)}
-\bgam_2(k,m)\mu_f\mu_g}
	\end{split}
  \end{equation*}
where
\begin{equation}\label{gamma2}
  \bgam_2(k,m):=
\frac{4(n+1)}{(k+1)(k+2)(m+1)(m+2)}.
\end{equation}

\item \label{lkmv}
If\/ $n=k$, then
\begin{equation*}
  \begin{split}
  \Cov\bigpar{F(\cT_n),G(\cT_n)} 
&=\E\bigpar{ f(\cT_k)G(\cT_k)}
-(n+1)\bgam_3(k,m)\mu_f\mu_g
	\end{split}
  \end{equation*}
where
\begin{equation}\label{gamma3}
  \bgam_3(k,m):=
  \begin{cases}
\frac{2}{(m+1)(m+2)}, & m<k,\\
\frac{1}{k+1}, & m=k.
  \end{cases}
\end{equation}

\item \label{lkmvi}
If\/ $n<k$, then $F(\cT_n)=0$ and thus $\Cov\bigpar{F(\cT_n),G(\cT_n)} =0$.

\end{romenumerate}
\end{Lemma}

\begin{proof}
(i):
The result is trivial for $k\ge n$ since $F(\cT_n)=F(\cT_k)=f(\cT_k)$ if $k=n$ and 
$F(\cT_n)=0$ if $k>n$.

Hence, assume $k<n$.
Using the cyclic representation \eqref{cyclic}, we find
\begin{equation}\label{eFG}
\E{F(\cT_n)} 
=\sum_{i=0}^{n} \E\bigpar{\iik f(\gs(i,k))}
=(n+1) \E\bigpar{\iik f(\gs(i,k))}
.
\end{equation}
Recalling \eqref{yik} and noting that $f(\gs(i,k))$ depends only on the
relative order of $U_i,\dots,\allowbreak U_{i+k-1}$, we see that 
$\iik$ and $f(\gs(i,k))$  are independent.
%(As in the special case \eqref{eixip}.)
Thus, using \eqref{eink},
\begin{equation}\label{emm}
  \E\bigpar{\iik f(\gs(i,k))}
=\E(\iik)  \E\bigpar{ f(\gs(i,k))}
=\E(\iik) \E\bigpar{ f(\cT_k)}
=\frac{2}{(k+1)(k+2)}\mu_f
\end{equation}
and thus
\begin{equation}%\label{eFG}
\E{F(\cT_n)} 
%=(n+1) \E(\iik)\E\bigpar{f(\gs(i,k))}
=(n+1) \E(\iik)\E\bigpar{f(\cT_k)}
=(n+1)\frac{2}{(k+1)(k+2)}\mu_f,
\end{equation}
showing \eqref{ekm} in the case $k<n$.

(ii)--(iv):
The cyclic representation \eqref{cyclic} similarly yields
\begin{equation}\label{covFG}
 \Cov\bigpar{F(\cT_n),G(\cT_n)} 
=\sum_{i=0}^{n}\sum_{j=0}^{n} \Cov\bigpar{\iik f(\gs(i,k)),I_{j,m} g(\gs(j,m))},
%\\ &
%=(n+1)\sum_{j=0}^{n} \Cov\bigpar{I_{0,k} f(\gs(0,k)),I_{j,m} f(\gs(j,m))}
\end{equation}
where
$\iik f(\gs(i,k))$ and $I_{j,m} g(\gs(j,m))$ are independent unless the
sets \set{i-1,\dots,i+k} and \set{j-1,\dots,j+m} overlap (as subsets of
$\bbZ_{n+1}$). 
Furthermore, as a consequence of \eqref{yik}, 
if these sets overlap by more than one element
but none of the sets is a subset of the other, then $I_{i,k}I_{j,m}=0$,
except in the case $k+m=n-1$ and $j-1\equiv i+k$, $i-1\equiv j+m \pmod{n+1}$
(again, this says that two subtrees cannot
overlap or be adjacent unless one is contained in the other). 

(ii):  We now assume $k+m< n-1$ and $k\ge m$.
Then \eqref{covFG}, symmetry and the observations just made yield
\begin{equation*}
  \begin{split}
  \Cov\bigpar{F(\cT_n),G(\cT_n)} 
&=(n+1)\Bigl(
%\sum_{j=0}^{n} \Cov\bigpar{I_{0,k} f(\gs(0,k)),I_{j,m} g(\gs(j,m))}
\E\bigpar{I_{0,k} f(\gs(0,k))I_{-m-1,m} g(\gs(-m-1,m))}
\\&
\hskip5em+
\sum_{j=0}^{k-m}\E\bigpar{I_{0,k} f(\gs(0,k))I_{j,m} g(\gs(j,m))}
\\&
\hskip5em+
\E\bigpar{I_{0,k} f(\gs(0,k))I_{k+1,m} g(\gs(k+1,m))}
\\&
\hskip5em
-(k+m+3)\E\bigpar{I_{0,k} f(\gs(0,k))}\E\bigpar{I_{0,m} g(\gs(0,m))}
\Bigr).
  \end{split}
\end{equation*}
As seen in the proof of (i),
$\iik$ is independent of
$f(\gs(i,k))$, and thus \eqref{emm} holds.
Similarly,
\begin{equation}\label{emmg}
  \E\bigpar{I_{j,m} g(\gs(j,m))}
=\frac{2}{(m+1)(m+2)}\mu_g,
\end{equation}
and %Independence yields also
\begin{equation}
\E\bigpar{I_{0,k} f(\gs(0,k))I_{k+1,m} g(\gs(k+1,m))}  
%=
%\E\bigpar{I_{0,k} I_{k+1,m}}
%\E\bigpar{f(\gs(0,k))}
%\E\bigpar{g(\gs(k+1,m))}  
=
\E\bigpar{I_{0,k} I_{k+1,m}}
\mu_f %\E\bigpar{f(\cT_k)}
\mu_g. %\E\bigpar{g(T)}  
\end{equation}
Furthermore, the argument for \eqref{eikk} generalizes to
  \begin{align}
\E(\yi_{0,k}\yi_{k+1,m})     
&=
\int_0^1 \Bigpar{u(1-u)^k+\tfrac1{k+1}(1-u)^{k+1}}
\Bigpar{u(1-u)^m+\tfrac1{m+1}(1-u)^{m+1}}
\dd u
\notag\\
&=
\int_0^1 \Bigpar{x^k-\tfrac{k}{k+1}x^{k+1}}\Bigpar{x^m-\tfrac{m}{m+1}x^{m+1}}
\dd x
\notag\\
&=
\frac{1}{k+m+1}-\frac{k}{(k+1)(k+m+2)}-\frac{m}{(m+1)(k+m+2)}
\notag\\
&
\hskip16em
+\frac{km}{(k+1)(m+1)(k+m+3)}
\notag\\&=
\frac {2({k}^{2}+3 km+{m}^{2}+4 k+4 m+3)}
{ \left( k+1 \right) \left( m+1 \right)  
\left( k+m+1 \right) \left( k+m+2 \right) \left( k+m+3 \right) }.
\label{eqw}
   \end{align}
(Again, this can also be obtain by a combinatorial argument.)

The term
$\E\bigpar{I_{0,k} f(\gs(0,k))I_{-m-1,m} g(\gs(-m-1,m))}$ is calculated in
the same way, and yields the same result.

Finally, for convenience shifting the indices,
\begin{equation}\label{sumfG}
  \begin{split}
&\sum_{j=0}^{k-m}\E\bigpar{I_{0,k} f(\gs(0,k))I_{j,m} g(\gs(j,m))}
\\
&\hskip4em
=
\E(I_{1,k})
\E\Bigpar{ f(\gs(1,k)) \sum_{j=1}^{k-m+1}I_{j,m} g(\gs(j,m))\Bigm| I_{1,k}=1}
\\
&\hskip4em
=
\frac{2}{(k+1)(k+2)}
\E\bigpar{ f(\cT_k)G(\cT_k)},
  \end{split}
\end{equation}
where the last equality follows because the conditioning on $I_{1,k}=1$
yields the same result as conditioning 
on $U_0=U_{k+1}=0$, and the linear representation \eqref{linear} shows that 
then the sum is $G(\cT_k)$. 
The result follows by collecting the terms above.

(iii):
In the case $k+m= n-1$, we argue in the same way, but as in the case
$k=(n-1)/2$ of \refL{lemma1} (a special case of the present lemma), 
there are only
$k+m+2=n+1$ terms to subtract and \eqref{eqw} is replaced by the simple
\begin{equation}
\E\bigpar{I_{0,k} I_{k+1,m}}=\frac{2}{n(n+1)},
\end{equation}
\cf{} \eqref{varxnk2} and \eqref{eikk=}.

(iv):
In the case $k+m>n-1$, 
there cannot be two disjoint subtrees of sizes $k$ and $m$.
Hence the arguments above yield
\begin{multline*}
  \Cov\bigpar{F(\cT_n),G(\cT_n)} 
=(n+1)\biggl(
\sum_{j=0}^{k-m}\E\bigpar{I_{0,k} f(\gs(0,k))I_{j,m} g(\gs(j,m))}
\\
{}-(n+1)\E\bigpar{I_{0,k} f(\gs(0,k))}\E\bigpar{I_{0,m} g(\gs(0,m))}
\biggr)
 \end{multline*}
and the result follows from \eqref{sumfG} and \eqref{emm}, \eqref{emmg}.

(v):
In the case $k=n$ we have $F(\cT_n)=F(\cT_k)=f(\cT_k)$, and the result follows
from \eqref{ekm}.

(vi): Trivial.
\end{proof}

This leads to the following formulas for a general functional $f$.
(Note that Lemmas \ref{lemma1}--\ref{Lkm} treat special cases.
The mean \eqref{eF} is computed by \citet{Devroye2}.)

\begin{thm}\label{TFvar}
  Let $f(T)$ be a functional of binary trees, and let $F(T)$ be the sum
  \eqref{F}. Further, let
  \begin{equation}
	\mu_k:=\E f(\cT_k)
  \end{equation}
and
\begin{equation}\label{pikn}
  \pi_{k,n}:=
  \begin{cases}
\frac{2}{(k+1)(k+2)}, & k<n,\\
\frac1{n+1}, & k=n, \\
0, & k>n.
  \end{cases}
\end{equation}
Then, for the random \bst,
\begin{equation}\label{eF}
\E F(\cT_n) = (n+1)\sumkn \pi_{k,n}\mu_k
\end{equation}
and
\begin{equation}\label{vF}
  \begin{split}
\Var\bigpar{F(\cT_n)}
&=(n+1)
\lrpar{
\sum_{k=1}^n
\pi_{k,n}\E\Bigpar{f(\cT_k)\bigpar{2F(\cT_k)-f(\cT_k)}}
-\sumkn\sum_{m=1}^n
\bgamx(k,m)\mu_k\mu_m}
  \end{split}
\end{equation}
where, using \eqref{gamma} and \eqref{gamma1}--\eqref{gamma3},
\begin{equation}\label{gamx}
  \bgamx(k,m):=
  \begin{cases}
\bgam(k,m),& k+m+1<n, \\
\bgam_1(k,m),& k+m+1=n, \\
\bgam_2(k,m),& \max\set{k,m}<n<k+m+1, \\
\bgam_3(k,m),& k=n\ge m, \\
\bgam_3(m,k),& m=n\ge k.
  \end{cases}
\end{equation}
\end{thm}

\begin{proof}
  Let $f_k(T):=f(T)\etta\set{|T|=k}$, and let $F_k$ be the corresponding sum
  \eqref{F}. 
Then $f(T)=\sum_k f_k(T)$  and $F(T)=\sum_k F_k(T)$.
Hence, using \refL{Lkm}\ref{lkmi},
\begin{equation}
  \E F(\cT_n) = \sumkn \E F_k(\cT_n)
=\sumkn (n+1)\pi_{k,n}\mu_k,
\end{equation}
which shows \eqref{eF}.

Similarly, using symmetry and \refL{Lkm}\ref{lkmii}--\ref{lkmv}, 
noting $\E f_k(\cT_k)=\E
f(\cT_k)=\mu_k$, 
\begin{equation*}
  \begin{split}
\Var\bigpar{F(\cT_n)}
&=\sum_{k=1}^n\sum_{m=1}^k(2-\gd_{km}) \Cov\bigpar{F_k(\cT_n),F_m(\cT_n)}
\\
&=\sum_{k=1}^n\sum_{m=1}^k(2-\gd_{km}) (n+1)
\Bigpar{\pi_{k,n}\E\bigpar{f_k(\cT_k)F_m(\cT_k)}-\bgamx(k,m)\mu_k\mu_m}
  \end{split}
\end{equation*}
(where $ \delta_{km} $ denotes the Kronecker delta).
Furthermore,  $F_m(\cT_k)=0$ for $m>k$,
and $F_k(\cT_k)=f_k(\cT_k)=f(\cT_k)$,
and thus
\begin{equation*}
  \begin{split}
\sum_{m=1}^k(2-\gd_{km})F_m(\cT_k)
&= 2\sum_{m=1}^\infty F_m(\cT_k) -F_k(\cT_k)
%=2F(\cT_k)-f_k(\cT_k)
%\\&
=2F(\cT_k)-f(\cT_k)
  \end{split}
\end{equation*}
and \eqref{vF} follows, noting that $\bgamx(k,m)$ by definition is symmetric
in $k$ and $m$.
\end{proof}

The formula \eqref{eF} for the expectation is also easily obtained by
induction, using a simple recurrence,
see \citet[Lemma 1]{HwangN}.

The notation above is a little cheating, since not only $\pi_{k,n}$  
but also $\bgamx(k,m)$
depends on $n$;
however, if $n>k+m+1$, neither depends on $n$, and we obtain
the following.
Define  
\begin{equation}
  \label{pik}
\pi_k:=\frac{2}{(k+1)(k+2)}
\end{equation}
and recall that $\cT$ is the random \bst{}  $\cT_N$ with random size $N$
such that $\P(|\cT|=k)=\P(N=k)=\pi_k$.

\begin{cor}\label{CFc}
  In the notation above, assume further that $f(T)=0$ when $|T|>K$,
  for some $K<\infty$. If $n>2K+1$, then
\begin{equation}\label{eFc}
\E F(\cT_n) %= (n+1)\sum_{k=1}^K \pi_k\mu_k
=(n+1)\E f(\cT)
\end{equation}
and
\begin{equation}\label{vFc}
  \begin{split}
\Var\bigpar{F(\cT_n)}
&=(n+1)
\lrpar{
\E\Bigpar{f(\cT)\bigpar{2F(\cT)-f(\cT)}}
-\sum_{k=1}^K\sum_{m=1}^K
\bgam(k,m)\mu_k\mu_m}.
  \end{split}
\end{equation}
\qed
\end{cor}

We can now prove Theorems \ref{Tcovbin} and \ref{Tcovbin2} as two %important
special cases of the results above.  

\begin{proof}[Proof of \refT{Tcovbin}]
Apply \refL{Lkm}\ref{lkmii} with $f(T_1):=\ett{T_1=T}$ and
  $g(T_1):=\ett{T_1=T'}$. 
Then $X_n^T=F(\cT_n)$ and $X_n^{T'}=G(\cT_n)$.
We have $\mu_f=\pkt$ and $\mu_g=\pmtx$.
Furthermore, if $f(\cT_k)\neq0$, then $\cT_k=T$ and $G(\cT_k)=G(T)=\qtt$.
Hence,
\begin{equation*}
  \E\bigpar{f(\cT_k)G(\cT_k)} =\qtt\E f(\cT_k) =\qtt\pkt.
\qedhere
\end{equation*}
\end{proof}

\begin{proof}[Proof of \refT{Tcovbin2}]
In principle, this follows from \refT{Tcovbin} by summing over all trees of
sizes $k$ and $m$, and evaluating the resulting sum;
however, it is easier to give a direct proof.
By symmetry we may assume $k\ge m$.
We  apply \refL{Lkm}\ref{lkmii} with $f(T):=\ett{|T|=k}$ and
$g(T):=\ett{|T|=m}$. 
Then $X_{n,k}=F(\cT_n)$ and $X_{n,m}=G(\cT_n)$. 
Furthermore, $f(\cT_k)=1$, $g(\cT_m)=1$  and $G(\cT_k)=X_{k,m}$.
Hence  $\mu_f=\mu_g=1$, and, using \eqref{exnk},  
\begin{equation}\label{aw}
\E\bigpar{f(\cT_k) G(\cT_k)}=\E X_{k,m}=
\begin{cases}
\frac{2(k+1)}{(m+1)(m+2)} , & m<k,\\
1 , & m=k.  
\end{cases}
\end{equation}
Hence, \refL{Lkm}\ref{lkmii} yields \eqref{covksubtrees} with
\begin{align}
\gskm &=
\begin{cases}
{\frac{4}{(k+2)(m+1)(m+2)} -\bgam(k,m)}, & m<k, \\  
\frac{2}{(k+1)(k+2)} -\bgam(k,k), & m=k, \\  
\end{cases}
\end{align}
which yields \eqref{covb2a}--\eqref{covb2b} by elementary calculations.
\end{proof}

\begin{Lemma}\label{LnonsingularT}
  Let $T_1,\dots,T_N$ be a finite sequence of distinct binary trees.
Then the matrix $(\gsxx_{T_i,T_j})_{i,j=1}^N$ 
in \refT{Tcovbin}
is non-singular and thus positive definite.
\end{Lemma}

\begin{proof}
  Let $K:=\max_i|T_i|$.
For any real numbers $a_1,\dots,a_N$ and any $n>2K+1$, \refT{Tcovbin} yields
\begin{equation}
  \label{vtt}
\Var\biggpar{\sum_{i=1}^N a_iX_n^{T_i}}
=\sum_{i,j=1}^Na_ia_j\Cov\bigpar{X_n^{T_i},X_n^{T_j}}
=(n+1)\sum_{i,j=1}^Na_ia_j\gsxx_{{T_i},{T_j}}.
\end{equation}
Since a variance always is nonnegative, it follows that 
the matrix $(\gsxx_{T_i,T_j})_{i,j=1}^N$ is positive semi-definite.

Suppose that the matrix is singular. Then, using \eqref{vtt}, there exist
$a_1,\dots,a_N$, not all $0$, such that
if $Z_n:=\sum_{i=1}^N a_iX_n^{T_i}$, then $\Var(Z_n)=0$ for every $n>2K+1$.
Hence $Z_n$ is a constant, \ie, it takes the same value (possibly depending
on $n$) for every realization of $\cT_n$. We shall see that this leads to a
contradiction.

We may assume that $a_i\neq0$ for every $i$ (otherwise we just ignore the
remaining trees $T_i$). We may further assume that $T_1,\dots,T_N$ are
ordered with $k:=|T_1|=\min_i|T_i|$.
For $n>K+ k+1$, let $T_{0,n}$ be the tree 
consisting of a path to the right 
from the root with $n$ nodes, and
let $T_{1,n}$ consist of a path to the right 
from the root with $n-k$ nodes together with a left subtree $T_1$ at the root.
The subtrees of $T_{1,n}$ with size in $[k,K]$ are paths to the right, one
each of each length $l\in[k,K]$, and in addition one copy of $T_1$;
$T_{0,n}$ have the same paths as subtrees but no other subtrees of these
sizes.
Thus, denoting the values of $Z_n$ for a realization $T$ of $\ctn$ by
$Z_n(T)$, and similarly for $X_n^{T_i}$,
we have $X_n^{T_1}(T_{1,n})=X_n^{T_1}(T_{0,n})+1$ and
$X_n^{T_i}(T_{1,n})=X_n^{T_i}(T_{0,n})$ for $i>1$, and hence
$Z_n(T_{1,n})=Z_n(T_{0,n})+a_1$.
This exhibits two possible realizations of $\cT_n$ with different values of
$Z_n$. Hence $\Var(Z_n)>0$, a contradiction which completes the proof.
\end{proof}

\begin{Lemma}\label{Lnonsingulark}
For every $N\ge1$, the matrix $(\gskm)_{k,m=1}^N$ 
of the values defined in \refT{Tcovbin2}
is non-singular and thus positive definite.
\end{Lemma}

\begin{proof}
  This can be proved in exactly the same way as \refL{LnonsingularT}.
Alternatively, it is an easy corollary of \refL{LnonsingularT}, since 
$X_{n,k}=\sum_{|T|=k}X_n^T$ for every $k$.
\end{proof}

In the finitely supported case in \refC{CFc},
both $\E F(\cT_n)$ and $\Var F(\cT_n)$ grow linearly in $n+1$. 
Asymptotically, this is true under much weaker assumptions.
We begin with the mean.
(The binary tree case \eqref{tlime} was shown by \citet[Lemma 1]{Devroye2}.)

\begin{thm}\label{Tlime}
Under the assumptions in \refT{TFvar},
assume further that $\E |f(\cT)|<\infty$ and 
$\mu_n=o(n)$ as $n\to\infty$.
Then
\begin{equation}\label{tlime}
  \E F(\cT_n) = n \E f(\cT) + o(n).
\end{equation}

More generally, if\/ $\E |f(\cT)|<\infty$ and 
$\mu_n=o(n^\ga)$ for some $\ga\in (0,1]$, 
then
\begin{align}
  \E F(\cT_n) = n \E f(\cT) + o(n^\ga), \label{tlimeo}
\intertext{and if\/ $\E |f(\cT)|<\infty$ and $\mu_n=O(n^\ga)$ for some 
$\ga\in [0,1)$, then
}
\E F(\cT_n) = n \E f(\cT) + O(n^\ga). \label{tlimeO}
\end{align}
\end{thm}

\begin{proof}
We have
\begin{equation}\label{tlime|}
\sumk\pi_k|\mu_k| \le
\sumk \pi_k\E |f(\cT_k)| =
  \E |f(\cT)|<\infty
\end{equation}
and similarly
\begin{equation}\label{tlime=}
  \E f(\cT)=\sumk \pi_k\E f(\cT_k)=\sumk\pi_k\mu_k,
\end{equation}
where the sum converges absolutely by \eqref{tlime|}.
%similarly, 
%$\sumk\pi_k|\mu_k|\le \sumk\pi_k\E|f(\cT_k)|=
Thus \eqref{eF} implies
\begin{equation}\label{tli}
  \lrabs{\frac{1}{n+1}\E F(\cT_n)-\E f(\cT)}
\le \sumk |\pi_{k,n}-\pi_k|\,|\mu_k|
\le \frac{|\mu_n|}{n}+\sum_{k=n+1}^\infty \pi_k|\mu_k|,
\end{equation}
which tends to 0 by the assumption $\mu_n=o(n)$ and \eqref{tlime|}.
This implies \eqref{tlime}. 

This is the case $\ga=1$ of \eqref{tlimeo}. 
For $\ga<1$,  \eqref{tli} similarly implies \eqref{tlimeo} and \eqref{tlimeO}
under the stated assumptions.
\end{proof}

For the variance we begin with
an upper bound that is uniform in $n$ and $f$.

\begin{thm}\label{Tny}
There exists a universal constant $C$ such that,
under the assumptions and notations of \refT{TFvar}, 
for all $n\ge1$,
\begin{equation}\label{vftny}
\Var(F(\cT_n))
\le C n
\lrpar{
 \biggpar{\sumk\frac{(\Var f(\ctk))\qq}{k^{3/2}}}^2
 + \sup_k \frac{\Var f(\cT_k)}k
+\sumk\frac{\mu_k^2}{k^2} 
}.
\end{equation}
\end{thm}

\begin{proof}
We split $f(T)=f\upp1(T)+f\upp2(T)$, where for a tree $T$ with $|T|=k$ we
define
$f\upp1(T):=\E f(\cT_k)=\mu_k$ and $f\upp2(T):=f(T)-\mu_k$; thus
$\E f\upp2(\cT_k)=0$. 
This yields a corresponding decomposition
$F(\cT_n)=F\upp1(\cT_n)+F\upp2(\cT_n)$,  and it suffices to estimate the
variance of each term separately.
For convenience, we drop the superscripts, and note that the two terms
correspond to the two special cases $f(T)=\mu_k$ when $|T|=k$
(i.e., $f(T)$ depends on $|T|$ only), and $\mu_k=\E
f(\cT_k)=0$, respectively.

\pfcase{1}{$f(T)=\mu_{|T|}$.}
In this case, $f(T)=\sumk \mu_k\ett{|T|=k}$ and 
$F(\cT_n)=\sum_{k=1}^n \mu_k X\nk$; furthermore, $X_{n,n}=1$ is deterministic.
Hence,
\begin{equation}\label{jk}
  \Var\bigpar{F(\cT_n)}
=\sum_{k=1}^{n-1} \sum_{m=1}^{n-1} \Cov(X\nk,X_{n,m})\mu_k\mu_m.
\end{equation}
These covariances are evaluated by \refL{Lkm}\ref{lkmii}--\ref{lkmiv}, as in
the special case $n>k+m+1$ treated in \refT{Tcovbin2}; this yields,
assuming  $m\le k<n$ and recalling \eqref{aw},
\begin{equation}
  \frac1{n+1}\Cov\bigpar{X\nk,X_{n,m}}
=
\frac{2}{(k+1)(k+2)} \E X_{k,m} -\bgamx(k,m).
\end{equation}
%If $n<k$ then $X\nk=0$, and if $n=k$ then $X\nk=1$; in
%both cases the covariance vanishes so we only have to consider $n>k$.

Suppose first that $m<k<n$. If $n>k+m+1$, then
$\Cov\bigpar{X\nk,X_{n,m}}<0$ by \eqref{covb2a}. 
If $n=k+m+1$, then, 
similar calculations 
as in the proof of \refT{Tcovbin2}, 
now using \eqref{aw} and \eqref{gamma1}, yield
\begin{align} \label{jk2}
  \frac1{n+1} \Cov\bigpar{X\nk,X_{n,m}}
&
%=\frac{4}{(k+2)(m+1)(m+2)} -\bgam_1(k,m)
=-\frac{4}{(k+1)(k+2)(m+2)} +\frac{2}{n(n+1)}
\le \frac{2}{n(n+1)},
\intertext{and when $k<n<k+m+1$,  \eqref{gamma2} similarly implies,} 
  \frac1{n+1} \Cov\bigpar{X\nk,X_{n,m}}
&
%=\frac{4}{(k+2)(m+1)(m+2)} -\bgam_2(k,m)
=-\frac{4(n-k)}{(k+1)(k+2)(m+1)(m+2)} 
<0.  \label{jk3}
\end{align}

In the case $m=k<n$ we obtain similarly, or simpler from 
\eqref{vxnk}--\eqref{vxnk2},
%\eqref{gamma}--\eqref{gamma2} 
\begin{equation}\label{jk4}
  \frac1{n+1}\Var\bigpar{X\nk}
= O\Bigparfrac{1}{k^2}.
\end{equation}

Suppose now that all $\mu_k\ge0$. The \eqref{jk},  \eqref{covb2a}
and \eqref{jk2}--\eqref{jk4} yield, for some $\CC$,
using the Cauchy--Schwarz inequality,
\begin{equation}
  \begin{split}
\frac1{n+1}
  \Var\bigpar{F(\cT_n)}
&\le 
\CCx \sum_{k=1}^{n-1} \frac{\mu_k^2}{k^2}
+ 2\sum_{k=1}^{n-2}\frac{\mu_k\mu_{n-1-k}}{n^2}
\le 
\CCx \sum_{k=1}^{n-1} \frac{\mu_k^2}{k^2}
+ \frac{2}{n^2}\sum_{k=1}^{n-2}\mu_k^2
\\&
\le(\CCx+2) \sum_{k=1}^{\infty} \frac{\mu_k^2}{k^2}	.
  \end{split}
\end{equation}
This proves \eqref{vftny} in the case $f(T)=\mu_{|T|}$, if we further assume
$\mu_k\ge0$, \ie, $f(T)\ge0$.
For a general sequence $\mu_k$, we split $f$ (and thus $\mu_k$) into its
positive and negative parts, and apply the estimate just obtained to each
part.
This yields \eqref{vftny} in general for Case 1.

\pfcase{2}{$\mu_{k}=0$, $k\ge1$.}
Let 
\begin{align}
a_k^2:=\Var (f(\cT_k))&&&  \text{and}&&
b_n^2:=\frac{1}{n+1}\Var (F(\cT_n)).  
\end{align}
Then \eqref{vF} implies, since we assume $\mu_k=0$,
using the Cauchy--Schwarz inequality and recalling \eqref{pikn},
\begin{equation}\label{kull}
  \begin{split}
b_n^2
&=
\sum_{k=1}^n
\pi_{k,n}\E\lrpar{f(\cT_k)\bigpar{2F(\cT_k)-f(\cT_k)}}
\le
2\sum_{k=1}^{n}
\pi_{k,n}\E\lrpar{f(\cT_k)F(\cT_k)}
\\&
\le
2\sum_{k=1}^{n}
\pi_{k,n} a_k (k+1)\qq b_k
\le
4\sum_{k=1}^{n-1} k^{-3/2} a_k b_k + 2 n^{-1/2} a_n b_n.
  \end{split}
\end{equation}
Now let $A:=\max\bigset{\sumk a_k k^{-3/2}, \sup_k k\qqw a_k}$.
We find from \eqref{kull}
\begin{equation}
  b_n^2 \le 4 A \max_{k<n} b_k + 2 A b_n
\end{equation}
and thus $(b_n-A)^2 \le 4 A\max_{k<n} b_k +A^2$, which by induction
implies $b_n\le 6A$.

In other words, 
\begin{equation}
\Var\bigpar{F(\cT_n)} \le 36 A^2(n+1)
\le 36  (n+1) 
\biggpar{
\biggpar{\sumk a_k k^{-3/2}}^2
+ \sup_k \frac{a_k^2}k },
\end{equation}
which proves
\eqref{vftny} in Case 2.
\end{proof}

\begin{rem}
In the proof of Case 1, it was convenient to reduce to the case $\mu_k\ge0$ in
order to require only upper bounds for 
$\Cov\bigpar{X\nk,X_{n,m}}$. This is not necessary, however. An alternative
  is to note that by \eqref{covb2a} and \eqref{jk2}--\eqref{jk3},
whenever $m<k<n$,
  \begin{equation}
\frac{1}{n+1}
\Cov\bigpar{X\nk,X_{n,m}} > - \frac{8}{k^2m}.
  \end{equation}
Hence, the proof can be concluded (for general $\mu_k$) 
by the additional estimate
\begin{equation}\label{jp}
  \sum_{k=1}^\infty \sum_{m=1}^\infty \frac{|\mu_k\mu_m|}{km\max\set{k,m}} 
\le 2  \sum_{k=1}^\infty \sum_{m=1}^\infty \frac{|\mu_k\mu_m|}{(k+m)km} 
\le \CC\sumk \frac{\mu_k^2}{k^2},
\end{equation}
which (by the substitution $x_k=|\mu_k|/k$) is an application of Hilbert's
inequality saying that the infinite matrix 
$(1/\xpar{k+m})_{k,m=1}^\infty$ defines a bounded operator on $\ell^2$,
see
\cite[Chapter IX]{HLP}.
\end{rem}

\begin{rem}
  In order for the estimate in \refT{Tny} to be useful, the three terms in
  the right-hand side of \eqref{vftny} have to be finite.
These conditions are the best possible that imply $\Var(F(\cT_n))=O(n)$ in
general, as is seen in the following examples. 
(We do not claim that these terms have to be finite in all cases
for $\Var(F(\cT_n))=O(n)$ to hold, but at least in some examples they have to.)

Consider first the case $f(T)=\mu_{|T|}$.
By \eqref{jk} and the estimates above,
\begin{equation}\label{jeppe}
  \begin{split}
\frac{\Var F(\cT_n)}{n+1}
&=
\sum_{k=1}^{n-1} \Bigpar{\frac{2}{k^2}+O\bigpar{k^{-3}}}\mu_k^2
+ \sum_{k=1}^{n-2}\frac{2}{n^2}\mu_k\mu_{n-1-k}
+\sum_{m<k<n}O\Bigparfrac{1}{k^2m}\mu_k\mu_m. 
  \end{split}
\end{equation}
Note that the factor $\xfrac{2}{k^2}+O\bigpar{k^{-3}}\ge 1/k^{2}$ unless
$k\le k_0$, for some $k_0$, and suppose for simplicity that $\mu_k=0$ for
$k\le k_0$. Then the first sum is at least 
$\sum_{k=1}^{n-1}\xfrac{\mu_k^2}{k^2}$.

Now consider instead of the sequence $(\mu_k)$ a random thinning $(\mu_k')$
obtained by letting $\mu_k'=\mu_k$ with some small fixed probability $p>0$,
and $\mu_k=0$ otherwise, independently for all $k$.
Replacing $\mu_k$ by $\mu_k'$ in \eqref{jeppe} and taking the expectation
over the thinnings yields, using the Cauchy--Schwarz inequality and \eqref{jp},
\begin{equation}%\label{jeppe}
  \begin{split}
\E\frac{\Var F(\cT_n)}{n+1}
&\ge
\sum_{k=1}^{n-1} \frac{p\mu_k^2 }{k^2} 
+  \sum_{k=1}^{n-2}\frac{2}{n^2}p^2\mu_k\mu_{n-1-k}
+\sum_{m<k<n}O\Bigparfrac{1}{k^2m}p^2\mu_k\mu_m
\\ & \ge
p\sum_{k=1}^{n-1} \frac{\mu_k^2}{k^2} 
-p^2 \CC \sum_{k=1}^{n-1} \frac{\mu_k^2}{k^2}.
  \end{split}
\raisetag{1.2\baselineskip}
\end{equation}
Choose $p \le 1/2\CCx$; then the right-hand side is at least 
$\frac p2\sum_{k=1}^{n-1} \xfrac{\mu_k^2}{k^2}$. 
Suppose now that $\sumk \mu_k^2/k^2=\infty$. By choosing $n$ large we then can
make $\E\Var F(\cT_n)/n$ arbitrarily large, so there exists an $n$ and a
thinning with $\Var F(\cT_n)/n$ arbitrarily large. This holds also if we fix
a finite number of the elements $\mu_k'$ of the thinning, and it follows by
using this argument recursively that there exists a (deterministic)
thinning $(\mu_k')$ and
a sequence $n_\nu\to\infty$ such that
$\Var F(\cT_n)/n\to\infty$ as \ntoo{} along this sequence.

For the case $\mu_k=0$, suppose that $(a_k)_3^\infty$ is a given sequence of
positive numbers. Define $f_1=f_2:=0$ and let
$g_3(T):=\#\set{\text{leaves in  $T$}}-4/3$ when $|T|=3$, 
where the constant $4/3$ is chosen such that
$\E g_3(\cT_3)=0$, cf.\ \eqref{munk}. 
Let $f_3(T)=c_3 g_3(T)$ for a constant $c_3>0$ such that
$\Var f_3(\cT_3)=a_3^2$. 
Continue recursively as follows: If we have chosen
$f_1,\dots,f_{k-1}$, let for a tree $T$ with $|T|=k$, 
$g_k(T):=\sum'_{v\in  T} f_{|T(v)|}(T(v))$, where $\sum'$ denotes summation
over all nodes except the root.
Define $f_k(T)=c_k g_k(T)$ for a constant $c_k>0$ such that
$\Var f_k(\cT_k)=a_k^2$. 
Note that, by induction using \eqref{eF}, 
$\E f_k(\cT_k)=\E g_k(\cT_k)=0$ for every $k$.

Consider $f:=\sum_k f_k$ and the corresponding $F$. By construction,
for $k>3$,
$F(T)=f_k(T)+g_k(T)=(1+c_k)g_k(T)$ for every tree $T$ with $|T|=k$.
If we let $d_k^2 := \Var g_k(\cT_k)$, we have $a_k^2=c_k^2d_k^2$ and
\begin{equation}\label{bl}
  \E\bigpar{f(\cT_k)\bigpar{2F(\cT_k)-f(\cT_k)}}
=
  \E\bigpar{c_k g_k(\cT_k)(c_k+2)g_k(\cT_k)}
=c_k(c_k+2) d_k^2
=a_k(2d_k+a_k).
\end{equation}
Similarly,
\begin{equation}
  \label{bl2}
\Var F(\cT_k) =(1+c_k)^2\Var g_k(\cT_k) 
=(1+c_k)^2b_k^2
= (a_k+d_k)^2.
\end{equation}
(For $k=3$, $F(\cT_3)=f_3(\cT_3)$, and \eqref{bl}--\eqref{bl2} hold if
we redefine $d_3:=0$.)

Note first that we have 
$\Var F(\cT_k) = (a_k+d_k)^2 \ge a_k^2=\Var f(\cT_k)$; hence, if
$\Var F(\cT_n)=O(n)$,  then $a_n^2=\Var f(\cT_n)=O(n)$, i.e., 
$\sup_k \Var f(\ctk)/k<\infty$.

Next,  \eqref{vF} and \eqref{bl}--\eqref{bl2} yield
\begin{equation}
  \begin{split}
  (d_n+a_n)^2 &
=\Var F(\cT_n) 
=(n+1)\sum_{k=3}^n \pi_{k,n} a_k(2d_k+a_k)
\\&
=
2a_nd_n+a_n^2+\sum_{k=3}^{n-1}\frac{2(n+1)}{(k+1)(k+2)} a_k(d_k+2a_k)
  \end{split}
\end{equation}
and thus
\begin{equation}\label{ws}
  d_n^2
=
\sum_{k=3}^{n-1}\frac{2(n+1)}{(k+1)(k+2)} a_k(2d_k+a_k).
\end{equation}
It follows that, for $n\ge4$,
%\begin{equation}
 $ d_n^2 > n a_3^2/10$, and thus
%\end{equation}
$d_n\ge\cc n\qq$ for  some $\ccx>0$.
Using this in \eqref{ws} we obtain
\begin{equation}
d_n^2
>
\cc n\sum_{k=4}^{n-1} k^{-3/2} a_k.
\end{equation}
Hence, if $\sumk k^{-3/2} a_k =\infty$, then $\Var F(\cT_n)/n\to\infty$ as
\ntoo. 
\end{rem}

\subsection{Random recursive tree}\label{SSmeanrrt}

For the random recursive tree we similarly  
compute mean and variance
using the cyclic representation
\eqref{cyclicrrt}.
Again, these have earlier been computed using the linear representation
(see \refS{SSlinearrrt}) 
by \citet{Devroye1}, and also by 
other (analytic) methods,
see 
\citet{FengMahmoudP08},
\citet{Fuchs}.

The representation \eqref{cyclicrrt} gives, recalling that subtrees of size $k$
correspond to subtrees of size $k-1$ in the corresponding binary tree, 
\begin{equation}\label{xnkrrt}
  \hX\nk= %\eqd
\sum_{i=1}^n\yl_{i,k-1},
\qquad 1\le k\le n.
\end{equation}

\begin{Lemma}\label{lemma1rrt}
Let $1\le k<n $.
For the \rrt,
\begin{align}
  \E(\hX_{n,k})&=\frac{n}{k(k+1)} \label{exnkrrt}
\end{align}
and %\intertext{and}
\begin{align}\label{vxnk0rrt}
\Var(\hX_{n,k})&= 
\begin{cases}
\E \hX_{n,k}
%+ n\lrpar{\frac{1}{k^2(2k+1)}-\frac{2k+1}{k^2(k+1)^2}  },
- n\frac{3k+2}{k(k+1)^2(2k+1)},
& k <\frac{n}2, \\
\E \hX_{n,k} -(\E \hX_{n,k})^2
=
\E \hX_{n,k}
-\frac{n^2}{k^2(k+1)^2}  ,
& k \ge\frac{n}2. \\
\end{cases}
\end{align}
Hence, for $1\le k<n$,
\begin{align}
\Var(\hX_{n,k})&=\E(\hX_{n,k})+O\Bigparfrac{n}{k^3}. \label{vxnkrrt}
\end{align}
\end{Lemma}
\begin{proof}
We use \eqref{xnkrrt} and argue as in the proof of \refL{lemma1}, replacing
$k$ and $n$ by $k-1$ and $n-1$.
By \eqref{ylik} and symmetry, for any $i$ and $1\le k<n$, 
\begin{equation}\label{elnk}
 \E(\yl_{i,k-1})=\frac{1}{k(k+1)} 
\end{equation}
and thus \eqref{exnkrrt} follows from \eqref{xnkrrt}.

For the variance, we obtain as in the proof of \refL{lemma1},
if $k<n/2$,
\begin{equation} \label{varxnkrrt}
  \begin{split}
\Var(\hX_{n,k})
%\sum_{i=0}^n\sum_{j=0}^n \Cov(\yl_{i,k},\yl_{j,k})\\
%&=(n+1)\Var(\yl_{0,k})+2(n+1)\sum_{j=1}^{k+1}\Cov(\yl_{0,k},\yl_{j,k})\\ 
&=
n\Bigpar{\E\yl_{0,k-1}
+2\E(\yl_{0,k-1}\yl_{k,k-1})
-(2k+1)(\E\yl_{0,k-1})^2}.
  \end{split}
\end{equation}
If $k\ge n/2$, 
then $\yl_{i,k}\yl_{j,k}=0$ unless $i=j$
and thus, or because $\hX\nk\le1$,
\begin{equation} \label{varxnk3rrt}
  \begin{split}
\Var(\hX_{n,k})&=
n\Bigpar{\E\yl_{0,k-1}-n(\E\yl_{0,k-1})^2}
=\E \hX\nk-\bigpar{\E\hX\nk}^2.
  \end{split}
\end{equation}
(There is no exceptional case when $k=n/2$, since $\yl_{0,k-1}\yl_{k,k-1}=0$
in this case.)

It remains to compute $\E(\yl_{0,k-1}\yl_{k,k-1})=\E(\yl_{1,k-1}\yl_{k+1,k-1})$.
We can argue as in the proof of \refL{lemma1},
recalling also the condition $U_{i-1}\le U_{i+k-1}$ in \eqref{ylik}, which
yields 
\begin{equation}\label{eikkrrt}
  \begin{split}
\E(\yl_{1,k-1}\yl_{k+1,k-1})
&=
\int_0^1 u(1-u)^{k-1}\cdot\frac1{k}(1-u)^{k} \dd u
\\
&=
\int_0^1 \frac{1}{k}(1-x)x^{2k-1} \dd x
%\\&
=\frac{1}{2k^2(2k+1)}.
   \end{split}
\end{equation}
Alternatively, it is this time easy to use a combinatorial argument;
$\yl_{1,k-1}\yl_{k+1,k-1}=1$ if $U_0$ is the smallest of $U_0,\dots,U_{2k}$,
$U_k$ is the smallest of the rest, and $U_{2k}$ is the smallest of
$U_{k+1},\dots,U_{2k}$; these events are independent and have probabilities
$1/(2k+1)$, $1/(2k)$ and $1/k$. 

Finally, \eqref{vxnk0rrt}--\eqref{vxnkrrt}
follow by simple calculations from \eqref{elnk}--\eqref{eikkrrt}.
 \end{proof}

\begin{Lemma}\label{lemma1P2}
Let $P$ be some property of ordered rooted trees.
Let $1\le k<n $ and let $\hat{p}_{k,P}:=\P(\gL_k\in P)$.
For the random recursive tree $\gL_n$, 
\begin{align}
  \E(\hat{X}_{n,k}^P)&=\frac{n\hat{p}_{k,P}}{k(k+1)} \label{exnkP2}
.
\end{align}
Furthermore,
\begin{align}\label{vxnk0P2}
\Var(\hat{X}_{n,k}^P)&= 
\begin{cases}
\E \hat{X}_{n,k}^P
%+ n\lrpar{\frac{1}{k^2(2k+1)}-\frac{2k+1}{k^2(k+1)^2}  },
- n\frac{3k+2}{k(k+1)^2(2k+1)}\, \hat{p}_{k,P}^2,
& k <\frac{n}2, \\
\E \hat{X}_{n,k}^P -(\E \hat{X}_{n,k}^P)^2
=
\E \hat{X}_{n,k}^P
-\frac{n^2}{k^2(k+1)^2} \, \hat{p}_{k,P}^2 ,
& k \ge\frac{n}2,
\end{cases}
\end{align}
and hence
\begin{align}
\Var(\hat{X}_{n,k}^P)&=\E(\hat{X}_{n,k}^P)+O\Bigparfrac{n \hat{p}_{k,P}^2}{k^3}. \label{vxnkP2}
\end{align}
\end{Lemma}

\begin{proof}
Let $I_{i,k-1}^{\bar{P}}$ be the indicator of the event that the binary
search tree 
defined by the permutation defined by $\gs(i,k-1)$ belongs to $\bar{P}$,
where $\bar{P}$ is the property of binary trees corresponding to
(by the natural correspondence) the property $ P $ of ordered rooted trees.
Then the cyclic representation
\refL{Lcyclicrrt} with   
$f(\Lambda)=\etta\set{\Lambda\in P_k}$
and thus $\bar{f}(T)=\etta\set{T\in \bar{P}_{k-1}}$ yields
\begin{equation}\label{xpnk2}
\hat{X}_{n,k}^P= %\eqd
\sum_{i=1}^{n}\yl_{i,k-1}I_{i,k-1}^{\bar{P}}.
\end{equation}

The rest of the proof is analogous to the proof of Lemma \ref{lemma1P}.
\end{proof}

\begin{Lemma}\label{Lkm2}
Let $1\le m\le k$.
Suppose that
$f(\Lambda)$ and $g(\Lambda)$ are two functionals of ordered rooted trees such that
$f(\Lambda)=0$ unless $|\Lambda|=k$ and $g(\Lambda)=0$ unless $|\Lambda|=m$,
and let $F(\Lambda)$ and $G(\Lambda)$
be the corresponding sums \eqref{F} over subtrees.
Let 
\begin{align}
  \lambda_f:=\E f(\gL_k) \qquad\text{and}\qquad \lambda_g:=\E g(\gL_m).
\end{align}
\begin{romenumerate}[-10pt]
\item\label{lkmi2}
The means of $F(\gL_n)$ and $G(\gL_n)$ are given by
\begin{equation}\label{ekm2}
  \E F(\gL_n)=
  \begin{cases}
\frac{n}{k(k+1)}\lambda_f, & n>k,\\
\lambda_f, & n=k, \\
0, & n<k,
  \end{cases}
\end{equation}
and similarly for\/ $\E G(\gL_n)$.

\item \label{lkmii2}
If\/ $n>k+m$, then
\begin{equation*}
  \begin{split}
  \Cov\bigpar{F(\gL_n),G(\gL_n)} 
&=n\lrpar{\frac{1}{k(k+1)}\E\bigpar{ f(\gL_k)G(\gL_k)}
-\hat{\bgam}(k,m)\lambda_f\lambda_g}
%\sum_{j=0}^{n} \Cov\bigpar{I_{0,k} f(\gs(0,k)),I_{j,m} g(\gs(j,m))}
	\end{split}
  \end{equation*}
where $\hat{\bgam}(k,m)$ is given by \eqref{gamma4}. 

\item \label{lkmiv2}
If\/ $k<n\leq k+m$, then
\begin{equation*}
  \begin{split}
  \Cov\bigpar{F(\gL_n),G(\gL_n)} 
&=n\lrpar{\frac{1}{k(k+1)}\E\bigpar{ f(\gL_k)G(\gL_k)}
-\hat{\bgam}_2(k,m)\lambda_f\lambda_g}
	\end{split}
  \end{equation*}
where
\begin{equation}\label{gamma2rec}
  \bgam_2(k,m):=
\frac{n}{k(k+1)m(m+1)}.
\end{equation}

\item \label{lkmv2}
If\/ $n=k$, then
\begin{equation*}
  \begin{split}
  \Cov\bigpar{F(\gL_n),G(\gL_n)} 
&=\E\bigpar{ f(\gL_k)G(\gL_k)}
-n\hat{\bgam}_3(k,m)\gl_f\gl_g
	\end{split}
  \end{equation*}
where
\begin{equation}\label{gamma3rec}
  \hat{\bgam}_3(k,m):=
  \begin{cases}
\frac{1}{m(m+1)}, & m<k,\\
\frac{1}{k}, & m=k.
  \end{cases}
\end{equation}

\item \label{lkmvi2}
If\/ $n<k$, then $F(\gL_n)=0$ and thus $\Cov\bigpar{F(\gL_n),G(\gL_n)} =0$.

\end{romenumerate}
\end{Lemma}

\begin{proof} The proof is similar to the proof of Lemma \ref{Lkm}.

\ref{lkmi2}:
The result is trivial for $k\ge n$ since $F(\Lambda_n)=F(\Lambda_k)=f(\Lambda_k)$ if $k=n$ and 
$F(\Lambda_n)=0$ if $k>n$.
Hence, assume $k<n$.
Using the cyclic representation \eqref{cyclicrrt}, we find
by similar calculations as in \eqref{emm},  using \eqref{elnk},
\begin{equation}\label{emm2}
  \begin{split}
\E{F(\Lambda_n)} 
&=\sum_{i=1}^{n} \E\bigpar{\yl_{i,k-1} \ff(\gs(i,k-1))}
=n\E\bigpar{\yl_{i,k-1} \ff(\gs(i,k-1)}
\\&
=n \E(\yl_{i,k-1})\E\bigpar{f(\Lambda_k)}
=\frac{n}{k(k+1)}\lambda_f,	
  \end{split}
\end{equation}
showing \eqref{ekm2} in the case $k<n$.

\ref{lkmii2}--\ref{lkmiv2}:
The cyclic representation \eqref{cyclicrrt} similarly yields
\begin{equation}\label{covFG2}
 \Cov\bigpar{F(\Lambda_n),G(\Lambda_n)} 
=\sumin\sumin \Cov\bigpar{\yl_{i,k-1} \ff(\gs(i,k-1)),
\yl_{j,m-1} \gbar(\gs(j,m-1))},
\end{equation}
where
$\yl_{i,k-1} \ff(\gs(i,k-1))$ and $\yl_{j,m-1} \gbar(\gs(j,m-1))$ are independent unless the
sets \set{i-1,\dots,i+k-1} and \set{j-1,\dots,j+m-1} overlap (as subsets of
$\bbZ_{n}$). 
Furthermore, as a consequence of \eqref{yik}, 
if these sets overlap by more than one element
but none of the sets is a subset of the other, then $\yl_{i,k-1}\yl_{j,m-1}=0$.
(Note that there is no exception with  $k+m=n$;
there is not room for two disjoint subtrees of sizes $k$ and $m$.)   
%(Note that there is no exception with $(k-1)+(m-1)=(n-1)-1$, i.e., $k+m=n$,
%since the corresponding binary tree does not have space for two left-rooted
%subtrees of sizes $k-1$ and $m-1$.)   

\ref{lkmii2}:  
We now assume $k+m< n$ and $k\ge m$.
Then \eqref{covFG2}, symmetry and the observations just made yield
\begin{equation*}
  \begin{split}
  \Cov\bigpar{F(\Lambda_n),G(\Lambda_n)} 
&=n\Bigl(
\E\bigpar{\yl_{0,k-1}  \ff(\gs(0,k-1))\yl_{-m,m-1} \gbar(\gs(-m,m-1))}
\\&
\hskip3em+
\sum_{j=0}^{k-m}\E\bigpar{\yl_{0,k-1} \ff(\gs(0,k-1))\yl_{j,m-1} \gbar(\gs(j,m-1))}
\\&
\hskip3em+
\E\bigpar{\yl_{0,k-1}  \ff(\gs(0,k-1))\yl_{k,m-1} \gbar(\gs(k,m-1))}
\\&
\hskip-1em
-(k+m+1)\E\bigpar{\yl_{0,k-1} \ff(\gs(0,k-1))}\E\bigpar{\yl_{0,m-1} \gbar(\gs(0,m-1))}
\Bigr).
  \end{split}
\end{equation*}
As seen in the proof of \ref{lkmi2},
$\yl_{0,k-1}$ is independent of
$\ff(\gs(i,k-1))$, and thus, \cf{} \eqref{emm2},
\begin{align}\label{emmf2}
&\E\bigpar{\yl_{i,k-1} \fbar(\gs(i,k-1))}=\frac{1}{k(k+1)}\lambda_f;
\end{align}
similarly,
\begin{gather}
\E\bigpar{\yl_{j,m-1} \gbar(\gs(j,m-1))}=\frac{1}{m(m+1)}\lambda_g,\label{emmg2}
\\
\E\bigpar{\yl_{0,k-1} \ff(\gs(0,k-1))\yl_{k,m-1} \gbar(\gs(k,m-1))}  
%=
%\E\bigpar{\yl_{0,k-1} \yl_{k,m-1}}
%\E\bigpar{\ff(\gs(0,k-1))}
%\E\bigpar{\gbar(\gs(k,m-1))}  
=
\E\bigpar{\yl_{0,k-1} \yl_{k,m-1}}
\gl_f %\E\bigpar{f(\cT_k)}
\gl_g. \label{emmg3}%\E\bigpar{g(T)}  
\end{gather}
Furthermore, the argument for \eqref{eikkrrt} generalizes to
  \begin{align}
\E(\yl_{0,k-1}\yl_{k,m-1})     
&=
\int_0^1 x (1 - x)^{k - 1}\frac{1}{m} (1 - x)^m \dd x
%\\&
=
\frac {1}
{ m \left(k + m\right)\left (1 + k + m\right)}.
\label{eqw2}
   \end{align}
(Again, this can also be obtain by a combinatorial argument.)
By  analogous calculations we obtain
$$
\E\bigpar{\yl_{0,k-1} \ff(\gs(0,k-1))\yl_{-m,m-1} \gbar(\gs(-m,m-1))}  
=\frac {1}{ k \left(k + m\right)\left (1 + k + m\right)}\gl_f\gl_g.
$$
(Note that this differs from \eqref{emmg3}--\eqref{eqw2}, unlike
the corresponding terms for
the binary search tree case where \eqref{eqw} is symmetric in $k$ and $m$.)
Finally, for convenience shifting the indices,
\begin{equation}\label{sumfG2}
  \begin{split}
&\sum_{j=0}^{k-m}\E\bigpar{\yl_{0,k-1} \ff(\gs(0,k-1))\yl_{j,m-1} \gbar(\gs(j,m-1))}
\\
&\hskip4em
=
\E(\yl_{1,k-1})
\E\Bigpar{ \ff(\gs(1,k-1)) \sum_{j=1}^{k-m+1}\yl_{j,m-1} \gbar(\gs(j,m-1))\Bigm| \yl_{1,k-1}=1}
\\
&\hskip4em
=
\frac{1}{k(k+1)}
\E\bigpar{ f(\Lambda_k)G(\Lambda_k)},
  \end{split}
\raisetag{1\baselineskip}
\end{equation}
where the last equality follows from the linear representation in \eqref{linearrrt}. 
The result follows by collecting the terms above.

\ref{lkmiv2}:
In the case $k+m\geq n$, 
there cannot be two disjoint subtrees of sizes $k$ and $m$.
Hence the arguments above yield
{\multlinegap=0pt
\begin{multline*}
  \Cov\bigpar{F(\Lambda_n),G(\Lambda_n)} 
=n\biggl(
\sum_{j=0}^{k-m}\E\bigpar{\yl_{0,k-1} \ff(\gs(0,k-1))\yl_{j,m-1} \gbar(\gs(j,m-1))}
\\
{}-n\E\bigpar{\yl_{0,k-1}\ff(\gs(0,k-1))}\E\bigpar{\yl_{0,m-1} \gbar(\gs(0,m-1))}
\biggr)
\end{multline*}}
and the result follows from \eqref{sumfG2} and \eqref{emmf2}, \eqref{emmg2}.

\ref{lkmv2}:
In the case $k=n$ we have $F(\Lambda_n)=F(\Lambda_k)=f(\Lambda_k)$, and the result follows
from \eqref{ekm2}.

\ref{lkmvi2}:
Trivial.
\end{proof}

\begin{thm}\label{TFvar2}
  Let $f$ be a functional of ordered rooted trees, and let $F$ be the sum
  \eqref{F}. Further, let
  \begin{equation}
	\lambda_k:=\E f(\Lambda_k)
  \end{equation}
and
\begin{equation}\label{pikn2}
  \hat{\pi}_{k,n}:=
  \begin{cases}
\frac{1}{k(k+1)}, & k<n,\\
\frac1{n}, & k=n, \\
0, & k>n.
  \end{cases}
\end{equation}
Then, for the random recursive tree,
\begin{equation}\label{eF2}
\E F(\Lambda_n) = n\sumkn \hat{\pi}_{k,n}\lambda_k
\end{equation}
and
\begin{equation}\label{vF2}
  \begin{split}
\Var\bigpar{F(\gL_n)}
&=n
\lrpar{
\sum_{k=1}^n
\hat{\pi}_{k,n}\E\Bigpar{f(\gL_k)\bigpar{2F(\gL_k)-f(\gL_k)}}
-\sumkn\sum_{m=1}^n
{\hbgamx}(k,m)\lambda_k\lambda_m}
  \end{split}
\end{equation}
where, using \eqref{gamma4} and \eqref{gamma2rec}--\eqref{gamma3rec},
\begin{equation}\label{gamx2}
  {\hbgamx}(k,m):=
  \begin{cases}
\hat{\bgam}(k,m),& k+m<n, \\
\hat{\bgam}_2(k,m),& \max\set{k,m}<n\leq k+m, \\
\hat{\bgam}_3(k,m),& k=n\ge m, \\
\hat{\bgam}_3(m,k),& m=n\ge k.
  \end{cases}
\end{equation}
\end{thm}

\begin{proof} 
Analogous to the proof of Theorem \ref{TFvar}, using
 Lemma \ref{Lkm2}.
\end{proof}

Recall that $\gLL$ is the \rrt{}  $\gL_N$ with random size $N$
such that $\P(|\gLL|=k)=\P(N=k)=\hat\pi_k:=1/(k(k+1))$.

\begin{cor}\label{CFc2}
In the notation above, assume further that $f(\Lambda)=0$ 
when $|\Lambda|>K$,  for some $K<\infty$. If $n> 2K$, then
\begin{equation}\label{eFc2}
\E F(\gL_n) %= n\sum_{k=1}^K \hat{\pi}_k\lambda_k
=n\E f(\gLL)
\end{equation}
and
\begin{equation}\label{vFc2}
  \begin{split}
\Var\bigpar{F(\gL_n)}
&=n
\lrpar{
\E\Bigpar{f(\gLL)\bigpar{2F(\gLL)-f(\gLL)}}
-\sum_{k=1}^K\sum_{m=1}^K
\hat{\bgam}(k,m)\lambda_k\lambda_m}.
  \end{split}
\end{equation}
\qed
\end{cor}

We can now prove Theorems \ref{Tcovrec} and \ref{Tcovrec2} as two %important
special cases of the results above.  The proofs are analogous to the proofs of Theorems \ref{Tcovbin} and \ref{Tcovbin2}, but we include them for completeness.

\begin{proof}[Proof of \refT{Tcovrec}]
Apply \refL{Lkm2}\ref{lkmii2} with $f(\Lambda_n(u)):=\ett{\Lambda_n(u)=\Lambda}$ and
  $g(\Lambda_n(u)):=\ett{\Lambda_n(u)=\Lambda'}$. 
Then $\hat{X}_n^\Lambda=F(\gL_n)$ and $\hat{X}_n^{\Lambda'}=G(\gL_n)$.
We have $\lambda_f=\hat{p}_{k,\Lambda}$ and
$\lambda_g=\hat{p}_{m,\Lambda'}$. 
Furthermore, if $f(\gL_k)\neq0$, then $\gL_k=\Lambda$ and $G(\gL_k)=G(\Lambda)=\hq_{\Lambda'}^{\Lambda}$.
Hence,
\begin{equation*}
  \E\bigpar{f(\gL_k)G(\gL_k)} =\hq_{\Lambda'}^{\Lambda}\E f(\gL_k) =
  \hq_{\Lambda'}^{\Lambda}\hat{p}_{k,\Lambda}.
\qedhere
\end{equation*}
\end{proof}

\begin{proof}[Proof of \refT{Tcovrec2}]
In principle, this follows from \refT{Tcovrec} by summing over all trees of
sizes $k$ and $m$, and evaluating the resulting sum;
however as noted for the binary search tree, it is easier to give a direct proof.
By symmetry we may assume $k\ge m$.
We  apply \refL{Lkm2}\ref{lkmii2} with $f(\Lambda):=\ett{|\Lambda|=k}$ and
$g(\Lambda):=\ett{|\Lambda|=m}$. 
Then $\hat{X}_{n,k}=F(\gL_n)$ and $\hat{X}_{n,m}=G(\gL_n)$. 
Furthermore, $f(\gL_k)=1$, $g(\gL_m)=1$  and $G(\gL_k)=\hat{X}_{k,m}$.
Hence  $\lambda_f=\lambda_g=1$, and, using \eqref{exnkrrt},  
\begin{equation}\label{aw2}
\E\bigpar{f(\gL_k) G(\gL_k)}=\E \hat{X}_{k,m}=
\begin{cases}
\frac{k}{m(m+1)} , & m<k,\\
1 , & m=k.  
\end{cases}
\end{equation}
Hence, \refL{Lkm2}\ref{lkmii2} yields \eqref{covksubtreesrecursive} with
\begin{align}
\gskm &=
\begin{cases}
{\frac{1}{(k+1)m(m+1)} -\hat{\bgam}(k,m)}, & m<k, \\  
\frac{1}{k(k+1)} -\hat{\bgam}(k,k), & m=k, \\  
\end{cases}
\end{align}
which yields \eqref{covb2c}--\eqref{covb2d} by elementary calculations.
\end{proof}

\begin{Lemma}\label{LnonsingularT2}
  Let $\Lambda_1,\dots,\Lambda_N$ be a finite sequence of distinct ordered 
or unordered 
rooted trees.
Then the matrix $(\hat{\sigma}_{\Lambda_i,\Lambda_j})_{i,j=1}^N$ 
in \refT{Tcovrec}
is non-singular and thus positive definite.
\end{Lemma}

\begin{proof}
The proof is analogous to the proof of Lemma \ref{LnonsingularT}.
\end{proof}

\begin{Lemma}\label{Lnonsingulark2}
For every $N\ge1$, the matrix $(\hat{\sigma}_{k,m})_{k,m=1}^N$ 
of the values defined in \refT{Tcovrec2}
is non-singular and thus positive definite.
\end{Lemma}

\begin{proof}
The proof is analogous to the proof of Lemma \ref{Lnonsingulark}.
\end{proof}

In the finitely supported case in \refC{CFc2},
both $\E F(\gL_n)$ and $\Var F(\gL_n)$ grow linearly in $n$. 
Asymptotically, this is true under much weaker assumptions.

\begin{thm}\label{Tlime2}
Under the assumptions in \refT{TFvar2},
assume further that $\E |f(\gLL)|<\infty$ and 
$\lambda_n=o(n)$ as $n\to\infty$.
Then
\begin{equation}\label{tlime2}
  \E F(\gL_n) = n \E f(\gLL) + o(n).
\end{equation}

More generally, if\/ $\E |f(\gLL)|<\infty$ and 
$\lambda_n=o(n^\ga)$ for some $\ga\le1$,
then
\begin{align}
  \E F(\gL_n) = n \E f(\gLL) + o(n^\ga), \label{tlimeo2}
\intertext{and if\/ $\E |f(\gLL)|<\infty$ and $\lambda_n=O(n^\ga)$ for some
  $\ga<1$, then} 
\E F(\gL_n) = n \E f(\gLL) + O(n^\ga). \label{tlimeO2}
\end{align}
\end{thm}

\begin{proof}
The proof is analogous to the proof of Theorem \ref{Tlime}.
\end{proof}

\begin{thm}\label{Tny2}
There exists a universal constant $C$ such that,
under the assumptions and notations of \refT{TFvar2}, 
for all $n\ge1$,
\begin{equation}\label{vftny2}
\Var(F(\gL_n))
\le C n
\lrpar{
 \biggpar{\sumk\frac{(\Var f(\gL_k))\qq}{k^{3/2}}}^2
 + \sup_k \frac{\Var f(\gL_k)}k
+\sumk\frac{\lambda_k^2}{k^2} 
}.
\end{equation}
\end{thm}

\begin{proof}
The proof is analogous to the proof of Theorem \ref{Tny2}.
\end{proof}

\section{Poisson approximation by Stein's method and couplings}

To prove Theorems \ref{main1} and \ref{poissontreepattern} 
we use Stein's method with couplings as 
described by  \citet{Janson}. 
In general, let $ \cA$ be a finite index set and let 
$(I_{\alpha},~\alpha\in\cA)$  be indicator random variables. We write 
$W:=\sum_{\alpha\in\cA} I_{\alpha}$ and $ \lambda:=\E(W) $. 
To approximate
$ W $ with a Poisson distribution $ \Po(\lambda) $, this method uses
a coupling for each $\alpha\in\cA$
between $ W $ and  a random variable  
 $ W_{\alpha} $
which is defined on the same probability space as $ W $ and has the property
\begin{align}\label{coupling2}
\mathcal{L}(W_{\alpha} )=\mathcal{L}(W-I_{\alpha}\mid I_{\alpha}=1).
\end{align}  
%$ \mathcal{L}(W_{\alpha})=\mathcal{L}(W-I_{\alpha}\mid  I_{\alpha}=1)$.  
A common way to construct such a coupling $ (W,W_{\alpha} )$ is to find
 random variables $ (J_{\beta\alpha},~\beta\in\cA)$  defined on the same
 probability space as $ (I_{\alpha},~\alpha\in\cA)$ in such a way that
 for each $ \alpha\in \cA $, and jointly for
all  $ \beta\in \cA $,
 \begin{align}\label{coupling}
\mathcal{L}(J_{\beta\alpha})=\mathcal{L}(I_{\beta}\mid I_{\alpha}=1 ).
\end{align}
Then $ W_{\alpha}=\sum_{\beta\neq\alpha}J_{\beta\alpha} $ is defined on the
same probability space as $ W $ and \eqref{coupling2} holds.

Suppose that
$J_{\gb\ga}$ are such random variables, and that,
for each $ \alpha $, the set $ \cA_{\alpha}:=\cA\backslash \{\alpha\}  $ is
partitioned into $ \cA_{\alpha}^{-} $ and $\cA_{\alpha}^{0} $ 
in such a way that  
\begin{equation}
  \label{jb}
 J_{\beta\alpha} 
   \leq I_{\beta}  \quad \text{if } \beta\in \cA_{\alpha}^{-}, 
\end{equation} 
with no condition if $\beta\in \cA_{\alpha}^{0} $. 
We will use the following result from \cite{Janson}
(with a slightly simplified constant).
(\cite{Janson} also contain similar results using a third part $\cA_\ga^+$
of $\cA_\ga$, where \eqref{jb} holds in the opposite direction; we will not
need them and note that it is always possible to 
include $\cA_\ga^+$ in $\cA_\ga^0$ and then use the following result.)

\begin{thm}[{\cite[Corollary 2.C.1]{Janson}}]\label{Tstein} 
Let\/ $ W=\sum_{\alpha\in\cA} I_{\alpha}$ and $ \lambda=\E(W) $. Let\/ $
\cA_{\alpha}=\cA\backslash \{\alpha\} $ and   
$ \cA_{\alpha}^{-},\cA_{\alpha}^{0}$  
be defined as above.
Then 
\begin{equation*}
  \begin{split}
d_{TV}(\mathcal{L}(W),\Po(\lambda)) 
&\leq (1\wedge \lambda^{-1})
\Bigpar{\lambda-\Var(W)+2\sum_{\alpha\in \cA}\sum_{\beta\in \cA_{\alpha}^{0}}
\E(I_{\alpha}I_{\beta})}.
	  \end{split}
\end{equation*}
\vskip-1.5\baselineskip
\qed
\end{thm}

\subsection{Couplings for proving Theorem \ref{main1} and Theorem \ref{poissontreepattern}}

Returning to the binary search tree, we use the cyclic representation 
$ X\nk= \sum_{i=1}^{n+1} \yi_{i,k}$ in \eqref{xnk}.
Recall the construction of $\yi_{i,k}$ in \eqref{yik} and the
distance
$\absni{i-j}$ on $\bbZ_{n+1}$ given by \eqref{dn+1}.

\begin{Lemma}\label{negativebinary}
Let 
$k\in\{1,\dots,n-1\}$ and let
$ \yi_{i,k}$ be as in \refS{SScyclic}.
Then for each $ i\in\{1,\dots,n+1\} $,
there exists a coupling  
$((\yi_{j,k})_j,(Z_{ji}^{k})_j )$ such that 
$\mathcal{L}(Z_{ji}^{k})=\mathcal{L}(\yi_{j,k}\mid \yi_{i,k}=1 )$  jointly
for all 
$ j \in \{1,\dots,n+1\}$. 
Furthermore, 
\begin{equation*}
  \begin{cases}
Z_{ji}^k = \yi_{j,k} & \text{if}\quad   \absni{j-i}>k+1, 
\\
Z_{ji}^k \ge \yi_{j,k} & \text{if}\quad   \absni{j-i}=k+1 ,
\\
Z_{ji}^k=0 \le \yi_{j,k} & \text{if}\quad   0<\absni{j-i}\le k. 
  \end{cases}
\end{equation*}
\end{Lemma}
\begin{proof}
We define $Z_{ji}^k$ as follows. (Indices are taken modulo $n+1$.)
Let $m$ and $m'$ be the indices in $i-1,\dots,i+k$ such that $U_m$ and
$U_{m'}$ are the two smallest of $U_{i-1},\dots,U_{i+k}$; 
if one of these is $i-1$ we choose $m=i-1$, and if one of them is $j+k$ we
choose $m'=j+k$, otherwise, we
randomize the choice of $m$ among these two indices so that
$\P(m<m')=\frac12$, independently of everything else.
Now exchange $U_{i-1}\leftrightarrow U_m$ and $U_{i+k}\leftrightarrow U_{m'}$,
i.e., let $U'_{i-1}:=U_m$, $U'_{m}:=U_{i-1}$, $U'_{i+k}:=U_{m'}$,
$U'_{m'}:=U_{i+k}$, 
and $U'_l:=U_l$ for all other indices $l$. 
Finally, let, cf.\ \eqref{yik},
\begin{equation}\label{zik}
  Z_{ji}^k=\etta\bigset{\text{$U'_{j-1}$ and $U'_{j+k}$ are the two smallest
	  among $U'_{j-1},\dots,U'_{j+k}$}}.
\end{equation}
Then,
$\cL\bigpar{U'_1,\dots,U'_n}=\cL\bigpar{(U_1,\dots,U_n)\mid \yi_{i,k}=1}$ and
thus 
$\mathcal{L}(Z_{ji}^{k})=\mathcal{L}(\yi_{j,k}\mid \yi_{i,k}=1 )$  jointly for
all $j$.

Note that $U'_l=U_l$ if $l\notin\set{j-1,\dots,j+k} $ and thus
$Z_{ji}^k=\yi_{j,k}$ if $\absni{j-i}>k+1$. 
On the other hand, if $0<j-i<k+1$, then
$Z_{ji}^k=0$ since  $i+k$ lies in $\set{j,\dots,j+k-1}$
and  $U'_{i+k}$ is smaller than $U'_{j-1}$ by construction; the case
$-k-1<j-i<0$ is similar. (This says simply that two different fringe trees
of the same size cannot overlap, which is obvious.)

Finally, if $j=i+k+1$ with $j+k+1<i+n+1$ (i.e., $k+1<(n+1)/2$),
then $j-1=i+k$ and thus $U'_{j-1}\le U_{j-1}$ while $U'_l=U_l$ for 
$l\in j,\dots,j+k$; hence $Z_{ji}^{k}\ge \yi_{i,k}$.
The cases $j=i+k+1$ with $j+k+1=i+n+1$ and
$j=i-k-1$ with $j-k-1>i-n-1$ are similar.
\end{proof}
 
Figures \ref{binartid}--\ref{binartidkopp} show an example
of this coupling, illustrated by the corresponding binary search trees;
in this example $i=4$, $k=3$, $m=i-1=3$, $m'=6$ and $U_0=0$.
\begin{figure}[ht]
\begin{minipage}[b]{0.4\linewidth}
\centering
\includegraphics[width=\textwidth]{binartidgrey.pdf}\vspace{0.7cm}
\caption{A binary search tree with no fringe subtree of size three containing the keys $ \{4,5,6\} $.}
\label{binartid}
\end{minipage}
\hspace{3cm}
\begin{minipage}[b]{0.4\linewidth}
\centering
\includegraphics[width=\textwidth]{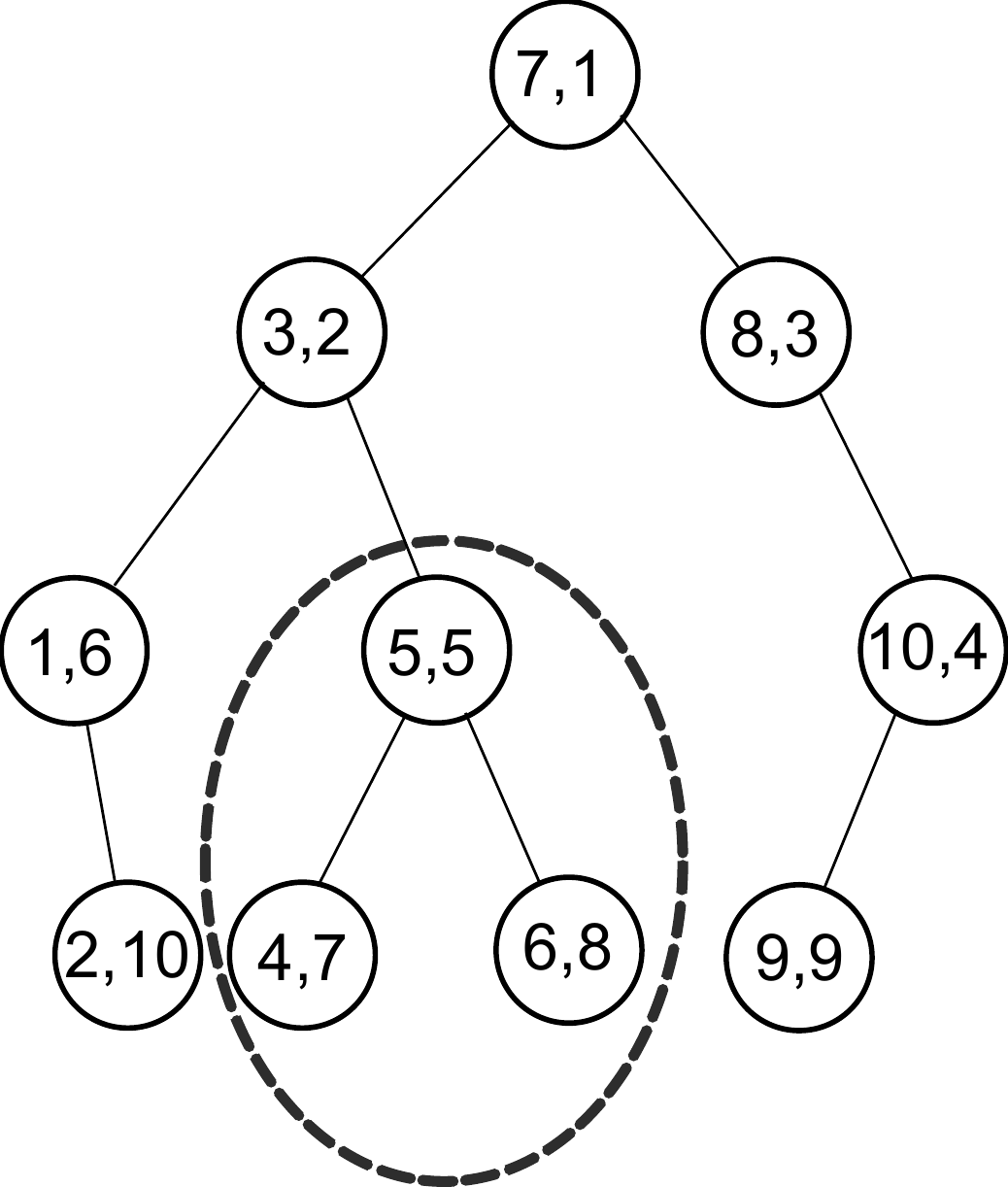}
\caption{A coupling forcing a fringe subtree of size three containing the keys $ \{4,5,6\} $ in the tree in Fig. \ref{binartid}.}
\label{binartidkopp}\end{minipage}
\end{figure}

For proving the Poisson approximation result in (\ref{Poissonrecursive}) for
the random recursive tree there is a similar coupling using 
the representation \eqref{xnkrrt} where
$\yl_{i,k-1}$
is defined by \eqref{ylik} and 
the indicators $U_i$ have period $n$:
$U_i:=U_{i\bmod n}$.

\begin{Lemma}\label{negativerecursive}
Let $k\in\{1,\dots,n-1\}$ and let
$ \yl_{i,k-1}$ be as in \refS{SScyclic}.
Then for each $ i\in\set{\nn}$,
there exists a coupling  
$((\yl_{j,k-1})_j,(Z_{ji}^{k-1})_j )$ such that 
$\mathcal{L}(Z_{ji}^{k-1})=\mathcal{L}(\yl_{j,k-1}\mid \yl_{i,k-1}=1 )$
jointly for all $ j \in \set{\nn}$. 
Furthermore, 
\begin{equation*}
  \begin{cases}
\hZ_{ji}^{k-1} = \yl_{j,k-1} & \text{if}\quad   \absn{j-i}>k, 
\\
\hZ_{ji}^{k-1}=0 \le \yl_{j,k-1} & \text{if}\quad   0<\absn{j-i}< k. 
  \end{cases}
\end{equation*}
\end{Lemma}

In contrast to \refL{negativebinary}, there is no monotonicity (in any
direction) between 
$\hZ_{ij}^{k-1} $ and $\yl_{i,k-1}$ when $\absn{j-i}=k$, as easily is seen
by simple examples.

\begin{proof} 
We use the same construction as in the proof of \refL{negativebinary} except
that if $U'_{j-1}>U'_{j+k}$ then we make a final additional interchange
$U'_{j-1}\leftrightarrow U'_{j+k}$. Denote the result by
$U''_1,\dots,U''_{n-1}$. 
The rest of the argument is as above, now defining
\begin{equation}\label{hzik}
  \hZ_{ji}^{k-1}=
\etta\bigset{U''_{j-1}<U''_{j+k-1}<\min_{j\le l\le j+k-2} U''_l}.
\end{equation}
\end{proof}

\begin{proof}[Proof of Theorem \ref{main1}] 
The means are given in Lemmas \ref{lemma1} and \ref{lemma1rrt}.

We prove the Poisson approximation result  first for 
the binary search tree,
using the representation 
$ X_{n,k}=\sum_{i=1}^{n+1} \yi_{i,k}$ in
\eqref{xnk}.
Let 
$\cA:=\set{\nni}$. 
From Lemma
\ref{negativebinary} we see that for each $i\in \cA $ we can 
apply \refT{Tstein} with
\begin{align*}
  \cA_{i}^{-}:= \cA\setminus \set{i,i\pm(k+1)},
\qquad
%&&&
\cA_{i}^{0}:=  \set{i\pm(k+1)};
\end{align*}
this yields,  
using Lemma \ref{lemma1} and \eqref{eikk}, 
provided $k\neq(n-1)/2$, 
%with  $ \mu_{n,k}=\E(X_{n,k}) $,
\begin{equation*}%\label{binarypoisson}
  \begin{split}
d_{TV}(\mathcal{L}(X_{n,k}),\Po(\mu_{n,k}))
%\\
&
\leq \bigpar{1\wedge \mu_{n,k}^{-1}}\Bigl(\mu_{n,k}-\Var(X_{n,k})
+4
\sum_{1\leq i\leq n+1}
\E(\yi_{i,k}\yi_{i+k+1,k})\Bigr)
\\&
=O\Bigpar{\frac{1}{\mu_{n,k}}\cdot\frac{n}{k^3}}
=O\Bigparfrac{1}{k},
  \end{split}
\end{equation*}
which shows \eqref{Poissonbinary}; the case $k=(n-1)/2$ follows similarly
from Lemma \ref{lemma1} and \eqref{eikk=}. 
%the case $k=(n-1)/2$ is trivial since then $\mu_{n,k}=O(1/k)$. 

For the random recursive tree, we argue similarly, using the representation
$\hX\nk = %\eqd
\sum_{i=1}^n\yl_{i,k-1}$
in \eqref{xnkrrt} and \refT{Tstein} together with
Lemmas \ref{negativerecursive} and \ref{lemma1rrt}, and \eqref{eikkrrt}.
\end{proof}

Lemmas \ref{negativebinary} and \ref{negativerecursive} can be extended to
include a property $P$. We state only the \bst{} case, and leave the \rrt{}
to the reader.
Recall that $\ipik$ is the indicator of the event that the binary search tree
defined by the permutation defined by $\gs(i,k)$ belongs to $P$.

\begin{Lemma}\label{negativebinaryP}
Let $k\in\{1,\dots,n-1\}$, and let 
$\tyi_{i,k}^P:=\yi_{i,k}^{\phantom P} \yi_{i,k}^{P}$.
Then for each  $ i\in \{1,\dots,n+1\}$ ,
there exists a coupling  
$((\tyi_{j,k}^{P})_{j}^{\phantom P},(W_{ji}^{k})_j^{\phantom P} )$ such that 
$\mathcal{L}(W_{ji}^{k})=\mathcal{L}(\tyi_{j,k}^{P}\mid \tyi_{i,k}^{P}=1 )$  jointly for all
$ j \in \{1,\dots,n+1\}$. Furthermore, 
\begin{equation}\label{cases}
  \begin{cases}
W_{ji}^k = \tyi_{j,k}^{P} & \text{if}\quad   |j-i|_{n+1}>k+1, 
\\
W_{ji}^k \ge \tyi_{j,k}^{P} & \text{if}\quad   |j-i|_{n+1}=k+1 ,
\\
W_{ji}^k=0 \le \tyi_{j,k}^{P} & \text{if}\quad   0<|j-i|_{n+1}\le k. 
  \end{cases}
\end{equation}
\end{Lemma}

\begin{proof}
We use the same notations as in the proof of \refL{negativebinary}. 
(In particular, indices
are taken modulo $ n+1 $.) 
Let $m$ and $m'$ be the indices in $i-1,\dots,i+k$ 
defined in proof of \refL{negativebinary},
and exchange $U_{i-1}\leftrightarrow U_m$ and $U_{i+k}\leftrightarrow U_{m'}$.
So far we
have used exactly the same  coupling as in  Lemma \ref{negativebinary}.
However, since we want $ \sigma(i,k) $ to have the property $ P$, 
we also exchange the values $ U'_i,\dots,U'_{i+k-1} $ with each other 
so that this property is fulfilled (choosing uniformly at random between the
orderings satisfying $P$).
We abuse notation and write $ U'_i,\dots,U'_{i+k-1} $ for the new
values after this exchange. Write  
\begin{equation*}
\sigma^{i}(j,k)=\{(j,U'_j),\dots,(j+k-1,U_{j+k-1}')\}  
\end{equation*} 
and note that $ \sigma^{i}(j,k) = \sigma(j,k)$ if $|j-i|\geq k+1  $.
Finally, let
\begin{equation}\label{wik}
  W_{ji}^k:=Z_{ji}^k\cdot\bigett{\sigma^{i}(j,k) \textrm{ has property } P},
\end{equation}
where $  Z_{ji}^k$ is defined by \eqref{zik}.
Then,
$\cL\bigpar{U'_1,\dots,U'_n}=\cL\bigpar{(U_1,\dots,U_n)\mid \tyi_{i,k}^{P}=1}$ and
thus 
$\mathcal{L}(W_{ji}^{k})=\mathcal{L}(\tyi_{j,k}^{P}\mid \tyi_{i,k}^{P}=1 )$  jointly for
all $j$. 
To see that (\ref{cases}) holds, we argue as in the proof of Lemma
\ref{negativebinary}. 
\end{proof}

\begin{proof}[Proof of  Theorem \ref{poissontreepattern}] We  prove the
  result for  
$ X_{n,k}^{P}$, the result for $ \hX_{n,k}^{P}$ follows by similar
  calculations.  

The mean $ \mu_{n,k}^{P}:=\E(X_{n,k}^{P})$ is given by \refL{lemma1P}.
From \refT{Tstein} together with 
\refL{negativebinaryP}, \refL{lemma1P}
and \eqref{eikk}--\eqref{eikk=},
we deduce that for $ k\neq (n-1)/2 $,
\begin{align*}
d_{TV}\bigpar{&\mathcal{L}(X_{n,k}^{P}),\Po(\mu_{n,k}^{P})}
\\&
\leq (1\wedge (\mu_{n,k}^{P})^{-1})\Big(\mu_{n,k}^{P}-\Var(X_{n,k}^{P})+4
\sum_{1\leq i\leq n+1}\big(\E(\tyi_{i,k}^{P}\tyi_{i+k+1,k}^{P})\Big)
\\
&=
\begin{cases}
 O\Bigpar{\frac{p_{k,P}}{k}} & \mbox{if }  \mu_{n,k}^{P}\geq 1 \\
 O\Bigpar{\frac{p_{k,P}}{k}}\cdot \mu_{n,k}^{P} & \mbox{if } \mu_{n,k}^{P} < 1
\end{cases}
\end{align*}
and for $ k=(n-1)/2 $, 
\begin{align*}
d_{TV}\bigpar{&\mathcal{L}(X_{n,k}^{P}),\Po(\mu_{n,k}^{P})}
= O\Bigpar{\frac{p_{k,P}^{2}}{k}},
\end{align*}
which shows Theorem \ref{poissontreepattern} in the binary tree case.
\end{proof}

\section{Normal approximation by Stein's method}

In this section we will prove Theorem \ref{main2} and Theorem
\ref{multivariate}.
As in \cite[Theorem 5]{Devroye2} we
use Stein's method in the following form, see \eg{}
\cite[Theorem 6.33]{Janson2} for a proof, and for the definition of
dependency graph.

\begin{Lemma}\label{jansonnormal} 
Suppose that $(S_n)_1^\infty$ is a sequence
of random variables such that
$S_n = \sum_{\alpha \in V_n} Z_{n\alpha}$, where for each $n$,
$\{ Z_{n\alpha} \}_\alpha$ is a family of random variables
with dependency graph $(V_n, E_n)$.
Let $N(\cdot)$ denote the closed neighborhood of
a node or set of nodes in this graph.
Suppose further that there exist numbers
$M_n$ and $Q_n$ such that
$$
\sum_{\alpha \in V_n} \E( | Z_{n\alpha} | ) \le M_n
$$
and for every $\alpha, \alpha' \in V_n$,
$$
\sum_{\beta \in N(\alpha,\alpha')}
   \E( | Z_{n\beta} | \mid Z_{n\alpha}, Z_{n\alpha'} ) \le Q_n~.
$$
Let $\sigma_n^2 = \Var (S_n )$. 
If 
\begin{equation}\label{mq}
\lim_{n\to\infty} \frac{M_n Q_n^2}{\sigma_n^3} = 0~,
\end{equation}
then
$$
\frac{S_n - \E( S_n )}{ \sqrt{ \Var ( S_n ) } }
\dto {\cal N}(0,1).
$$
\end{Lemma}

\begin{proof}[Proof of Theorem \ref{main2}]
We consider the binary search tree. 
The random recursive tree is similar. 

From Lemma \ref{lemma1} we have 
\begin{align}
  \E(X_{n,k})&=\frac{2(n+1)}{(k+2)(k+1)} \label{lexe}
\intertext{and}
\Var(X_{n,k})&=\E(X_{n,k})+O\Bigpar{\frac{n}{k^3}}. \label{lexv}
\end{align}

By the usual argument with subsequences, it suffices to consider the two
cases $k\to\infty$ and $k=O(1)$.

If $k\to\infty$ and $ k=o(\sqrt{n}) $,  
then Theorem \ref{main1}  shows that
$ X_{n,k} $ can be approximated by a random variable with a
$\Po(\E(X_{n,k})) $ distribution, where by 
\eqref{lexe}--\eqref{lexv},
$\Var(X_{n,k})\sim \E(X_{n,k})\to\infty $ as \ntoo.
Thus, from Theorem \ref{main1} and the central limit theorem for
Poisson distributions,
it follows that 
$\frac{X_{n,k}-\E(X_{n,k})}{\sqrt{\Var(X_{n,k})}} \dto \mathcal{N}(0,1) $ 
%distribution 
as $ n\rightarrow \infty $.

Thus, it remains to only show Theorem \ref{main2} for $ k=O(1) $. We repeat
the arguments used in \cite[Theorem 5]{Devroye2}, but using the
representation \eqref{xnk}.
(In fact, it suffices to consider a fixed $k$ and then the result follows by
\refT{multivariate}. However, we prefer to give a direct, and somewhat more
general, proof.)

We define the dependency graph $ (V_n,E_n) $ for the collection of random
variables $\{\yi_{i,k},~1\leq i\leq n+1\}$
by taking
$$V_{n}=\set{1,\dots,n+1}
$$
and $E_n:=\set{(i,j):0<\absni{i-j}\le k+1}$. Then $|N(\ga,\ga')|\le2(2k+3)$
for all $\ga,\ga'\in V_n$, and thus we may take $Q_n=4k+6$ in
\refL{jansonnormal}. We further take $M_n=\E X_{n,k}=O(n/k^2)$.
Thus, $M_nQ_n^2=O(n)$, and 
to show \eqref{mq} and thus Theorem \ref{main2} for the binary search tree
it is enough to show that 
\begin{align}\label{steinnormal}
%\frac{k^2 \E(X_{n,k})}{\Var(X_{n,k})^{\xfrac{3}{2}}}
%=
\frac{n}{\Var(X_{n,k})^{\xfrac{3}{2}}}
\stackrel{n\rightarrow \infty}{\longrightarrow}0.
\end{align}

For $k=O(1)$, \refT{Tcovbin2} shows that $\Var(X_{n,k})\ge c n$, and
\eqref{steinnormal} follows, which completes the proof.

More generally, \refT{Tcovbin2} shows that $\Var(X_{n,k}) \ge c n/k^2$ for
all $k<(n-1)/2$. Thus $n/\Var(X_{n,k})^{3/2}=O(k^3/n\qq)$, and it follows
that \eqref{steinnormal} 
holds if $ k=o(n^{1/6}) $. 
\end{proof}

\begin{proof}[Proof of Theorem \ref{multivariate}]
We show the result for the binary search tree, for the random recursive tree
the  proof follows by analogous calculations. Recall that $
\bX_n=(X^{T^{1}}_n,X^{T^{2}}_n,\dots,X^{T^{d}}_n)$ 
and  let $ \mathcal{Z}_{d} =(Z_1,\dots,Z_d)$,  where $ \mathcal{Z}_{d}$ is
multivariate normal with the distribution  $\N(0, \Gamma )$, where 
$\Gamma$ is the
matrix with elements $ \gamma_{ij}=\lim_{n\rightarrow
  \infty}\frac{1}{n}\Cov(X^{T^{i}}_n,X^{T^{j}}_n)$, see \eqref{mv0}.
Note that $\gG$ is non-singular by \refL{LnonsingularT}.

By the \CramerWold{} device \cite[Theorem 7.7]{billingsley}, to show that
$ n^{-\frac{1}{2}}(\bX_n-\bmu_n)$ converges in
distribution to $ \mathcal{Z}_{d} $, it is enough to show that for every
fixed vector  $( t_1,\dots,t_{d})\in \mathbb{R}^d $ we have
\begin{align}\label{binarymultinormal}
\frac{\sum_{j=1}^{d} t_j
	X^{T^{j}}_n-\E\lrpar{\sum_{j=1}^{d} t_j X^{T^{j}}_n}}{\sqrt{n}}\dto
  \sum_{j=1}^{d} t_j Z_{j}, 
\end{align}
where 
$ \sum_{j=1}^{d} t_jZ_{j} \sim 
\N\bigpar{0,\gam^2}$ with
\begin{equation}\label{gam2}
  \gam^2 :=\sum_{j,k=1}^{d} t_jt_k \gam_{jk}.
\end{equation}

Let $S_n:=\sum_{j=1}^{d}t_{j}X^{T^{j}}_n$.
Theorem \ref{Tcovbin} implies that, as \ntoo,
\begin{equation}\label{barb}
\Var(S_n) \sim n \sum_{j,k=1}^{d} t_jt_k \gs_{T^{i},T^{j}} =
 n \sum_{j,k=1}^{d} t_jt_k \gam_{ij} =n\gam^2.    
\end{equation}
In particular, if $\gam^2=0$, then \eqref{binarymultinormal} is trivial,
with the limit 0.

To show that (\ref{binarymultinormal}) holds when $\gam^2>0$, we will use
the same 
method as was used in \cite[Theorem 5]{Devroye2} for proving this theorem
(in a more general form)
%for the binary search tree 
in the 1-dimensional case $d=1$.
Let 
$| T^{j} | =k_j$, $1\leq j\leq d $. 
We use the cyclic representation \eqref{cyclic}, which in this case can be
written as, see \eqref{xpnk},
\begin{equation*}
   X_n^{T^j} = \sumini I_{i}^j
\end{equation*}
for some indicator variable $I_i^j=I_{i,k_j}I_{i,k_j}^{T^j}$ depending only on
$U_{i-1},\dots,U_{i+k_j}$. 
We define
$$
V_{n}:=\{(i,{j}):1\leq i\leq n+1,  1\leq j\leq d\}
$$ 
and let for each $(i,j)\in V_n$, $\tia_{i,j}$ be the set
\set{i-1,\dots,i+k_j}, regarded as a subset of $\bbZ_{n+1}$.
Thus $I_i^j$ depends only on $\set{U_k:k\in\tia_{i,j}}$,
and thus we can define
a dependency graph $L_n$ with vertex set $V_n$ by connecting $(i,j)$ 
and $(i',j')$ when $\tia_{i,j}\cap\tia_{i',j'}\neq\emptyset$.

Let $K:=\max\{k_1,k_2,\dots, k_d\}$ and $ M:=\max\{t_1,t_2,\dots,t_d\}
$. It is easy to see that for the sum 
$$
S_n:=
\sum_{j=1}^{d}t_{j}X^{T^{j}}_n
=\sum_{i=1}^{n+1}\sum_{j=1}^{d}t_{j}I_i^j
=\sum_{(i,j)\in V_n}t_{j}I_i^j,
$$ 
we can
choose the numbers $M_n$ and  $ Q_n $ in Lemma \ref{jansonnormal}
as $M_n=(n+1)dM$ and
$$
Q_n=2M\sup_{(i,{j})\in V_n}| N((i,j))|
\le 2Md(2K+3).
$$ 
Since $\gs_n\sim n\qq$ by \eqref{barb}, \eqref{mq} holds and
\refL{jansonnormal} shows that
\eqref{binarymultinormal} holds.
%
%Similarly, the result holds for the random recursive tree by applying 
%\refL{jansonnormal} and using Theorem \ref{covariancerecursive}. 
\end{proof}

\section{Truncations}

As said in the introduction, we combine \refT{multivariate} with a
truncation argument to deal with more general additive functionals.

\begin{proof}[Proof of \refT{TF}]
We consider again the binary search tree. The random recursive tree is
similar.

Note first that \eqref{tfa1}--\eqref{tfa2} imply
  \begin{align}
\sumk\frac{\Var f(\cT_k)}{k^{2}}
\le
\lrpar{\sup_k \frac{\Var f(\cT_k)}{k}}\qq
\sumk\frac{(\Var f(\cT_k))\qq}{k^{3/2}}
<\infty,
  \end{align}
and thus, using also \eqref{tfa3},
  \begin{align}
\sumk\frac{\E|f(\cT_k)|^2}{k^{2}}
=\sumk\frac{\Var f(\cT_k)}{k^{2}}
+\sumk\frac{(\E f(\cT_k))^2}{k^{2}}
<\infty.
  \end{align}
It follows that 
$\sumk\frac{\E|f(\cT_k)|}{k^{2}}<\infty$, and thus, see \eqref{tlime|} and
\eqref{pik}, that $\E|f(\cT)|<\infty$.
Since \eqref{tfa3} also implies $\E f(\ctk)/k\to0$ as $k\to\infty$,
\eqref{tfe} follows by \refT{Tlime} and \eqref{tlime=}.

Next, define the truncations $f^N(T):=f(T)\ett{|T|\le N}$, and the
corresponding sums $F^N(T)$. Then 
\begin{equation}
  \label{FN}
F^N(\ctn)=\sum_{|T|\le N} f(T) X_n^T,
\end{equation}
and thus \refT{Tcovbin} yields, as \ntoo,
\begin{align}\label{sam}
  \Var F^N(\ctn)/n \to \gss_{F,N}
:= \sum_{|T|,|T'|\le N} f(T)f(T')\gstt.
\end{align}
Moreover,  \refT{Tny} applied to $f-f^N$ yields
\begin{equation}\label{FNx1}
\frac{1}n\Var\bigpar{F(\cT_n)-F^N(\ctn)}
\le C \gd_N
\end{equation}
where
\begin{equation}\label{FNx2}
\gd_N:=
{
 \biggpar{\sum_{k>N}\frac{(\Var f(\ctk))\qq}{k^{3/2}}}^2
 + \sup_{k>N} \frac{\Var f(\cT_k)}k
+\sum_{k>N}\frac{\mu_k^2}{k^2} 
}.
\end{equation}
Note that $\gd_N$  is independent of $n$, 
and by the assumptions \eqref{tfa1}--\eqref{tfa3}, 
$\gd_N\to0$ as $N\to\infty$. 
%Hence 
It follows by Minkowski's
inequality that the sequences $(\Var (F^N(\ctn))/n)_{n\ge1}$  converge uniformly
to $(\Var (F(\ctn))/n)_{n\ge1}$. This and \eqref{sam} imply \eqref{tfv}
(including the existence of the limit in \eqref{tfv}).

For the convergence in distribution \eqref{tfd}, we use again the
truncation $f^N$ and $F^N$, and note that \refT{multivariate} implies
\begin{equation}
  \frac{ F^N(\ctn)-\E F^N(\ctn)}{\sqrt{n}}\dto \cN(0,\gss_{F,N})
\end{equation}
as \ntoo{}, for each fixed $N$.
 This together with the uniform  bound \eqref{FNx1} where $\gd_N\to0$
as \Ntoo{}
implies \eqref{tfd}, see \eg{} \cite[Theorem 4.2]{billingsley}.
\end{proof}

\begin{proof}[Proof of \refC{CTF}]
The assumption  $f(T)=O(|T|^\ga)$ with $\ga<1/2$ implies
\eqref{tfa1}--\eqref{tfa3} and \eqref{tfa1r}--\eqref{tfa3r}.
The result follows by \refT{TF}.
The version \eqref{ctfd} of the asymptotic normality \eqref{tfd} follows by
\eqref{tlimeO} and \eqref{mub}, and similarly for the random recursive case
\eqref{ctfdr}.
\end{proof}

\section{Proofs of Theorems \ref{TG1}--\ref{TGr}}

Finally, we prove Theorems \ref{TG1} and \ref{TGr}, beginning with  exact
formulas for finite $n$.  

\begin{Lemma}
  \label{Lpsi}
Let $F(T)$ be given for binary trees $T$ by \eqref{F}, with a functional
$f(T)=f(|T|,|T_L|,|T_R|)$ that depends only on the sizes of $T$ and of its left
and right subtrees.
Let %$\nu_k$, $I_k$ and 
$\psi_k$ be as in \refT{TG1}.
%
%$\nu_k:=\E F(\ctk)$, let $I_k$ be uniformly distributed on
%\set{0,\dots,k-1} and let
%\begin{equation}\label{psik}
%  \psi_k:=\Var\bigpar{\nu_{I_k}+\nu_{k-1-I_k}+f(k,I_k,k-1-I_k)}
%=\E\bigpar{\nu_{I_k}+\nu_{k-1-I_k}+f(k,I_k,k-1-I_k)-\nu_k}^2.
%\end{equation}
Then
\begin{equation}\label{ema}
  \Var(F(\ctn)) = 
(n+1)\sum_{k=1}^{n-1}\frac{2}{(k+1)(k+2)}\psi_k +\psi_n.
\end{equation}
\end{Lemma}

\begin{proof}
We use the notation in \refT{TG1}, and let $\gss_n:=\Var(F(\ctn))$.

Condition  the random tree $\cT_n$  on having a left subtree of size
  $|\cT_{n,L}|=k$; then the two subtrees $\cT_{n,L}$ and $\cT_{n,R}$
are independent random trees with the distributions
$\cT_{n,L}\eqd\cT_k$ and $\cT_{n,R}\eqd\cT_{n-1-k}$.
Hence, \eqref{rf} implies that the conditional distribution of $F(\ctn)$ is
given by
\begin{equation}\label{em}
  \bigpar{F(\ctn)\mid|T_{n,L}|=k}
\eqd
f(n,k,n-1-k)+F(\ctk)+F(\cT'_{n-k-1}),
\end{equation}
where $\cT'_{n-k-1}\eqd\cT_{n-k-1}$ is independent of $\cT_k$.

Taking the expectation in \eqref{em} we obtain the conditional expectation
of $F(\ctn)$ as
\begin{equation}
  \E\bigpar{F(\ctn)\mid|T_{n,L}|=k}
=g(k)
:=f(n,k,n-1-k)+\nu_k+\nu_{n-1-k}.
\end{equation}
Since $|T_{n,L}|\eqd I_n$, it follows that
\begin{equation}\label{emw}
  \E\bigpar{F(\ctn)\mid|T_{n,L}|}
\eqd g(I_n).
\end{equation}
Consequently,
\begin{equation}\label{emvare}
 \Var\bigpar{\E\bigpar{F(\ctn)\mid|T_{n,L}|}}
=\Var( g(I_n))=\psi_n
\end{equation}
by \eqref{psik}; the last equality in \eqref{psik} follows because
taking the expectation in \eqref{emw} yields
\begin{equation}
\E g(I_n)=\E F(\ctn)=\nu_n.
\end{equation}

Furthermore,
taking the variance in \eqref{em} we obtain the conditional variance
\begin{equation}\label{emvar}
  \Var\bigpar{F(\ctn)\mid|T_{n,L}|=k}
=\Var(F(\ctk))+\Var(F(\cT'_{n-1-k}))
=\gss_k+\gss_{n-1-k}.
\end{equation}

Consequently, by a standard variance decomposition formula
(``the law of total variance''),
see, e.g., \cite[Exercise 10.17-2]{Gut}, together with \eqref{emvare} and
\eqref{emvar},
\begin{equation}\label{em4}
  \begin{split}
  \gss_n=\Var(F(\ctn))
&=
  \E\bigpar{\Var\bigpar{F(\ctn)\mid|T_{n,L}|}}
+  \Var\bigpar{\E\bigpar{F(\ctn)\mid|T_{n,L}|}}
\\&
=\E\bigpar{\gss_{I_n}+\gss_{n-1-I_n}}+\psi_n.	
  \end{split}
\raisetag\baselineskip
\end{equation}
If we define $\Psi(T)$ by \eqref{rf} using the toll function
$\psi(T):=\psi_{|T|}$, it follows from \eqref{em4} and induction that
$\E(\Psi(\ctn)) = \gss_n$, and thus 
\eqref{ema} follows from \eqref{eF}.
\end{proof}

\begin{Lemma}  \label{Lpsir}
Let $F(\gL)$ be given for rooted trees $T$ by \eqref{F}, with a functional
$f(\gL)=f(|\gL|,d(\gL),\gL_{v_1},\dots,\gL_{v_{d(\gL)}})$ that
depends only on the
size $|\gL|$ 
and the number and sizes of the principal subtrees.
Let $\psi_k$ be as in \refT{TGr}.
Then
\begin{equation}\label{emar}
  \Var(F(\gln)) = 
n\sum_{k=1}^{n-1}\frac{1}{k(k+1)}\psi_k +\psi_n.
\end{equation}
\end{Lemma}

\begin{proof}
  Similar to the proof of \refL{Lpsi} with mainly notational changes,
now conditioning on the degree
$d=d(\gL_n)$ and the sizes of the principal subtrees
$\gL_{n,v_1},\dots,\gL_{n,v_{d}}$, and using \eqref{eF2}.
\end{proof}

\begin{proof}[Proof of \refT{TG1}]
  By \refT{TF}, $\Var(F(\ctn))/(n+1)\to\gss_F<\infty$. Since $\psi_n\ge0$,
  this and \eqref{ema} imply that
  \begin{equation}\label{emu}
\sum_{k=1}^{\infty}\frac{2}{(k+1)(k+2)}\psi_k <\infty	
  \end{equation}
and
\begin{equation}\label{emv}
  \gss_F
=\sum_{k=1}^{\infty}\frac{2}{(k+1)(k+2)}\psi_k+\lim_{\ntoo}\frac{\psi_n}{n+1},
\end{equation}
where the limit has to exist. However, if
$\lim_{\ntoo}\xqfrac{\psi_n}{n+1}\neq0$, then
%$\sum_{k=1}^{\infty}\frac{2}{(k+1)(k+2)}\psi_k=\infty$. 
\eqref{emu} cannot hold.
Hence 
$\lim_{\ntoo}\xqfrac{\psi_n}{n+1}=0$ and \eqref{gssfpsi} follows from
\eqref{emv}. 

It follows from \eqref{gssfpsi} and \eqref{ema} that
$\gss_F=0\iff \psi_k=0 \,\forall k \iff \Var(F(\ctn))=0\,\forall n$.
 The final conclusion follows by \eqref{psik}. (If
$f(n,k,n-1-k)=a_n-a_k-a_{n-1-k}$, then $F(T)=a_{|T|}-(|T|+1)a_0$ is
deterministic.) 
\end{proof}

\begin{proof}[Proof of \refT{TGr}]
Similar.  
\end{proof}

\section{Applications}\label{Sapp}

In this section we give some simple examples of applications of the
results above.

\subsection{Outdegrees}\label{SSdegrees}

First we consider the number of nodes 
in $\ctn$ or $\gL_n$
of a certain outdegree (number of children) $d\ge0$; we denote these numbers
by $\dn{d}$  and $\hdn{d}$, respectively.
These equal $X_n^P$ and $\hX_n^P$, where $P$ is the property that the root
has degree $d$.
Consequently, \refC{CP} immediately yields convergence of the expectation
and variance divided by $n$, and asymptotic normality provided the
asymptotic variance does not vanish.
By \refR{RTFjoint}, this extends to joint convergence for several outdegrees
$d$.

The case $d=0$ is simple; the vertices with outdegree $0$ are the leaves, and
thus $\dn0=X_{n,1}$ and $\hdn0=\hX_{n,1}$ with 
means given by \eqref{munk}--\eqref{murrt} and
variances given in
Theorems \ref{Tcovbin2} and \ref{Tcovrec2}. 
(In this case, the asymptotic normality also follows by \refT{main2}
or \ref{multivariate}.)
To find the asymptotic variances
for $d>0$ (and covariances) directly from \refC{CP} seems much more difficult.
However, as noted already by \citet{Devroye1}, for the \bst, when the only
outdegrees are $0,1,2$, it is possible to reduce to the case $d=0$,
because
\begin{align}\label{nodeschildren}
\dn0+\dn1+\dn2=n
\qquad\text{and}\qquad 
\dn1+2\dn2=n-1,
\end{align}
%where the first equality follows since the total number of nodes is $ n $
%and the second equality follows since the total number of children is $
%n-1$. 
and hence
\begin{align}\label{outdegreetwo}
\dn2=\dn0-1
\qquad\text{and}\qquad 
\dn1=n+1-2\dn0.
\end{align}   
Hence we recover the result by  \citet[Theorem 2]{Devroye1}: 
%(proved by the same method):
\begin{example}[{\citet{Devroye1}}]\label{outdegreesbin} 
$ \dn{d} $, the number of vertices 
with outdegree $ d $ in the \bst,
$ d={0,1,2} $,
has expectation (for $n>1$)
\begin{align}\label{expbindegrees}
  \E(\dn{0})&=\E(\dn{1})= \frac{n+1}{3},&\E(\dn{2})&= \frac{n-2}{3}
\end{align}
and variance (for $n>3$)
\begin{align}\label{varbindegrees}
\Var{\dn{0}}&=\Var{\dn{2}}=\frac{2}{45}(n+1),
&\Var{\dn{1}}&=\frac{8}{45}(n+1)
\end{align} and  for each $ d\in\{0,1,2\} $, as \ntoo,
\begin{align}
\label{normalbindegrees}\dfrac{\dn{d}-\E(\dn{d})}{\sqrt{\Var(\dn{d})}}
\dto
\mathcal{N}\bigpar{0,1}.  
\end{align}
\end{example}

\begin{rem} 
The asymptotic means $\mud{d}:=\lim_{\ntoo}\E \dn{d}/n$ can also be
calculated by \eqref{cpe} or \eqref{tfe}. For \eqref{cpe}, we note that the
growing binary tree $\cT_t$ has root degree distributed as
$\Bin(2,1-e^{-t})$, and thus, by the definition of $\cT:=\cT_\tau$,
\begin{equation}
  \mud{d}:=\int_0^\infty \binom 2d (1-e^{-t})^d(e^{-t})^{2-d} e^{-t}\dd t
=\binom2d\int_0^1(1-x)^dx^{2-d}\dd x =\frac13,
\end{equation}
for each $d=0,1,2$, see \citet{Aldous-fringe}.
If we instead use \eqref{tfe}, we obtain 
$$
\mud{d}=\sumk\frac{2}{(k+1)(k+2)}p_{k,d},
$$ 
where $p_{k,d}$ is the
probability that the root of $\cT_k$ has degree $d$.
For $d=0$ we have $p_{1,0}=1$ and $p_{k,0}=0$ for $k>1$; hence
$\mud0=\frac{2}{2\cdot 3}=\frac{1}3$. For $d=1$ we have $p_{1,1}=0$ and
$p_{k,1}=2/k$ for $k\ge2$, since 
the binary search tree generated by a sequence of keys has root degree 1 if
and only if the first key is either the  largest or the  smallest.
Hence \eqref{tfe} yields 
\begin{align}\label{expbindegreesCTF}
\mud1=\sum_{k=2}^{\infty}\frac{2}{(k+2)(k+1)}\cdot \frac{2}{k}=\frac{1}{3}.
\end{align}
We can similarly show $\mud2=\frac13$  too by \eqref{tfe}.
\end{rem}

For the random recursive tree, \refC{CP} yields the following,
which was proved (using an urn model) by \citet{Janson2005},
extending earlier results by
Mahmoud and Smythe \cite{MS:trees}.
In fact, \cite{Janson2005} gave also a generating function for the variances
$\gsshd{d}$ (and the covariances), enabling us to calculate them; as said
above, it seems difficult to obtain $\gsshd{d}$ by the methods of this
paper except for $d=0$, when $\gsshd0=\hgs_{1,1}=\frac{1}{12}$ by
\eqref{covb2d}.
(The asymptotic formula \eqref{convergencep} for the expectation was shown
earlier by \citet{NaR}. The convergence in probability
$\hdn{d}/n\pto2^{-d-1}$, which follows from \eqref{normalrecursivedegrees},
was shown by \citet{MM88}.)
\begin{thm}\label{outdegreesrecursive} 
For $\hdn{d} $, the number of vertices with outdegree $ d\ge0 $ in the \rrt,
it holds that, as \ntoo,
\begin{align}\label{convergencep}
 \frac{\E\hdn{d}}{n}\to 2^{-d-1}
\end{align}
and furthermore 
\begin{align}
\label{normalrecursivedegrees}
\dfrac{\hdn{d}-2^{-d-1}n}{\sqrt{n}}
\dto
\mathcal{N}\lrpar{0,\gsshd{d}} 
\end{align}
for some constant $ \gsshd{d}\ge0 $.
\end{thm}

\begin{proof}
By \refC{CP}, it remains only to calculate $\hmud{d}:=\lim_\ntoo\E\hdn{d}/n$.
We use \eqref{cper} and note that the growing random tree $\gLL_t$ 
has root degree with the Poisson distribution $\Po(t)$.
Since we stop the process at a random time $ \tau\sim \Exp(1) $ it follows that
$$ 
\hmud{d}=\int_{0}^{\infty}\frac{t^d e^{-t}}{d!}\cdot e^{-t}\dd t=2^{-d-1},
$$
as calculated by \citet{Aldous-fringe}.
\end{proof}

\begin{rem}
An alternative approach  for finding $\hmud{d}$ %in \eqref{normalmuoutdegree} 
is to use the the natural correspondence between the recursive tree and the
binary search tree.  
A node of outdegree $ d $ in the recursive tree 
(except the root of the whole tree) 
corresponds to a left-rooted subtree in the binary search tree with a rightmost
path of length $ d-1 $, and  thus to a left-rooted right path of
length $d-1$, considering here only paths that cannot be continued further
to the right.
By symmetry, the expected number of such paths equals the expected number of
rightrooted right paths of length $d-1$, but these paths are the right paths
of length $d$.
By symmetry again, on the average half of these paths (except paths from
the root) are left-rooted, and thus
$\E \hdn{d+1}=\frac12\E\hdn{d}+O(1)$.
Hence, $\hmud{d} =2^{-d-1} $ follows by induction
since $\hmud0=\frac{1}{2} $ (see  \eqref{murrt}). 
\end{rem}

\subsection{Protected nodes}\label{SSprotected}

We proceed to use fringe trees to study the so-called protected nodes that
recently have been studied in several types of random trees,  see e.g.\
\cite{CheonShapiro, MahmoudWard2012, Bona, DevroyeJanson,  MahmoudWard2014}
and the references there.  
 A node is $ \ell$-\emph{protected} 
 if the shortest distance to a descendant that is a leaf is at least $\ell $.
The most studied case is $\ell=2$: a node is \emph{two-protected} if it is
neither a leaf nor the parent of a leaf. 
\begin{rem}
The case $\ell=1$ is a bit trivial, at least for the random trees studied
here: a node is 1-protected if and only if it 
is a non-leaf. Hence, for \bst{s} and \rrt{s}, where the number of nodes is
given, it is equivalent to study the number of leaves, which was done in
\refS{SSdegrees}. (However, for random trees with a random number of nodes,
for example the ternary search tree studied in \cite{HolmgrenJanson},
this case too is interesting.)  
\end{rem}

\refC{CP} implies immediately that for any $ \ell $, the number of 
$\ell$-protected nodes is asymptotically normal
in both the binary search tree and the random recursive tree,
at least provided the asymptotic variances below are non-zero, which is an
obvious conjecture although we have no rigorous proof for $\ell\ge3$,
\cf{} \refPQ{PQ0}.

\begin{thm} \label{protected3} 
Let $\ell\ge1$ and 
let $\yln$ denote the number of $ \ell $-protected nodes in a binary
search tree $ \cT_{n} $. 
Then,
for some constants 
$\muyl=\P(\text{the root of $\cT$ is $\ell$-protected})>0$ 
and
$ \sigmayl^2 \ge0$,
with at least $\sigmay2^2>0$,
\begin{align}\label{lprotbinexp}
 \frac{\E(\yln)}{n}
&\rightarrow \muyl
%=\P(\text{the root of $\cT$ is $\ell$-protected})
,
\\
\label{lprotbinvar} 
\frac{\Var(\yln)}{n}
&\rightarrow \sigmayl^{2} ,
\intertext{and}
\label{lprotbinnormal}
\frac{\yln-n\muyl}{\sqrt{n}},\;
\dfrac{\yln-\E(\yln)}{\sqrt{n}}
&\dto
\mathcal{N}\lrpar{0,\sigmayl^{2}}. 
\end{align} 

Similarly,
let $\zln$ denote the number of $ \ell $-protected nodes in a random
recursive tree $ \gL_{n} $. 
Then,
for some constants 
$\muzl=\P(\text{the root of $\gLL$ is $\ell$-protected})>0$ 
and
$ \sigmazl^2 \ge0$,
with at least $\sigmaz2^2>0$,
 \begin{align}\label{lprotrecexp}
 \frac{\E(\zln)}{n}&\rightarrow \muzl, 
\\
\label{lprotrecvar} 
\frac{\Var(\zln)}{n}&\rightarrow \sigmazl^{2}
\intertext{and}
\label{lprotrecnormal}
\dfrac{\zln-\E(\zln)}{\sqrt{n}}
,\;
\dfrac{\zln-n\muzl}{\sqrt{n}}
&\dto
\mathcal{N}\lrpar{0,\sigmazl^{2}}.
\end{align}
\end{thm}

\begin{proof}
Let $ P $ be the class of trees such that the root is $\ell$-protected
and apply \refC{CP}, noting that $\yln = X_n^P$ and $\zln = \hX_n^P$.

That $\sigmay2^2>0$ and $\sigmaz2^2>0$ follows from Theorems \ref{TG1} and
\ref{TGr}.
\end{proof}

By \refR{RTFjoint}, we also obtain joint normality for several $\ell$.

\refT{protected3} includes several earlier results, proved by several 
different methods:
\eqref{lprotbinexp} was shown by 
\citet{MahmoudWard2012} for $\ell=2$ and by
\citet{Bona} and \citet{DevroyeJanson} in general;
\cite{MahmoudWard2012} also shows 
\eqref{lprotbinvar}--\eqref{lprotbinnormal} for $\ell=2$;
\eqref{lprotrecexp} was shown by 
\citet{MahmoudWard2014} for $\ell=2$ and by
\citet{DevroyeJanson} in general;
\cite{MahmoudWard2014} also shows for $\ell=2$ the weaker version of 
\eqref{lprotrecvar} that $\Var(\zxn2)=O(1/n)$.

To calculate the asymptotic means and variances is more complicated, however.
For the binary search tree, 
the asymptotic means $\muyl$ 
were calculated for $\ell\le4$
by \citet{Bona} (using generating functions)
and \citet{DevroyeJanson} (using the formula \eqref{cpe} as here)
to be
$\muy1=\frac{2}3$, $\muy2=\frac{11}{30}$,
$\muy3={\frac {1249}{8100}}$,   % 0.1541975309
$\muy4={\frac {103365591157608217}{2294809143026400000}}$;
the
methods in these papers apply to arbitrary $\ell$ (and yield rational
numbers)
but explicit calculations
quickly become cumbersome.

For the \rrt, $\muz1=\frac12$ 
as a consequence of \refT{outdegreesrecursive} (with $d=0$)
and $\muz2=\frac12-e\qw$ by \citet{MahmoudWard2014} 
and \citet{DevroyeJanson};
the method in 
\cite{DevroyeJanson} is based on \eqref{cper} as here and yields
(recursively) a complicated integral expression for every $\ell$, but we do
not know any closed form for $\ell\ge3$.

For the asymptotic variances, the formulas \eqref{cpv} and \eqref{cpvr} do
not seem to easily yield explicit formulas (although they might be useful
for numerical approximations). The only value that we know, 
except for $\ell=1$ when $\gssy1=\gs_{1,1}=\frac{2}{45}$ 
(\cf{} \eqref{varbindegrees}) and $\gssz1=\hgs_{1,1}=\frac{1}{12}$,
is $\gssy2=\frac{29}{225}$. 
In fact, for the \bst{} and $\ell=2$ we can compute the mean and variance 
of $\yxn2$ exactly by a special trick; the result is stated
in the following theorem
earlier shown by \citet[Theorems 2.1, 2.2 and 3.1]{MahmoudWard2012} 
(using generating functions and recurrences),
which is a more precise version of the special case $\ell=2$
of \refT{protected3} for the \bst.
(See also \cite[Theorem 1.2]{HolmgrenJanson} for a different proof of the
asymptotic normality using P\'olya urns.)
\begin{thm}[\citet{MahmoudWard2012}]  \label{protected1} 
Let $ \yxn2 $ denote the number of two-protected nodes in a binary search tree $
\cT_{n} $. Then \begin{align}
\label{expbinprot} \E(\yxn2)=\frac{11}{30}n-\frac{19}{30},
\qquad\text{for} \quad n\geq 4, 
 \end{align}
and
\begin{align}
\label{varbinprot} \Var(\yxn2)=\frac{29}{225}(n+1),
\qquad\text{for} \quad n\geq 8. 
\end{align}
Furthermore, as \ntoo,
\begin{align}
\label{normalbinprot}\dfrac{\yxn2-\frac{11}{30}n}{\sqrt{n}}
\dto
\mathcal{N}\Bigl(0,\frac{29}{225}\Bigr).  
\end{align}
\end{thm}
We provide a simple proof of this theorem using our results on fringe trees.
Moreover, our approach using fringe trees also allows us to provide a simple
proof of the following result which was conjectured in 
\cite[Conjecture 2.1]{MahmoudWard2012}.

\begin{thm}\label{protected2} 
%Let $ \yxn2 $ denote the number of protected nodes in a binary search tree 
%$\cT_n $. 
For each fixed integer $ k\geq 1 $, there exists a polynomial $
p_{k}(n) $ of degree $ k $, the leading term of which is $ (\frac{11}{30})^k
$, such that $ \E(\yxn2^{k})=p_{k}(n) $ for all $ n\geq 4k $. 
\end{thm}

\begin{proof}[Proof of Theorem \ref{protected1}] 

 In a binary tree (with at least 2 nodes), 
the number of nodes that are not two-protected
 equals two times the number of leaves (counting all the leaves and all the
 parents of the leaves) minus the number of cherry subtrees, i.e., subtrees
 consisting of a root with one left and one right child that both are leaves
(since these are the only cases when a parent is counted twice). 
Thus, writing $ L $ for a tree that is a single leaf and 
$ C $ for a tree that is a cherry, 
\begin{align}\label{binprotected}
\yxn2=n-2X^L_n+X^C_n.
\end{align}
Hence,
 \begin{align}\label{expbinprotected}\E(\yxn2)=n-2\E(X^L_n)+\E(X^C_n)
\end{align}
and
\begin{align}\label{varbinprotected}
\Var(\yxn2)=4\Var(X^L_n)+\Var(X^C_n)-4\Cov(X^L_n,X^C_n).
\end{align}

By \eqref{ptp},
the expected number of subtrees of $\ctn$ isomorphic to a tree $T$ of size
$|T|=k$ is 
\begin{equation}
  \label{pastoral}
\E(X^{T}_n)=\frac{2(n+1)}{(k+2)(k+1)}p_{k,T},
\end{equation}
where  
$ p_{k,T}=\P(\cT_k= T)=\xfrac{| A_{k}^T | }{k!}$ where 
$ A_k ^{T}$ is the set of permutations of length $ k $ that give rise to the
binary search tree $ T $.  Evidently $p_{1,L}=1$, and for
the cherry $C$ we have $|C|=3$ and $p_{3,C}=\frac{1}{3}$. 
Thus $\E(X^L_n)=(n+1)/3$ (for $n\ge2$), as already seen in
\eqref{expbindegrees}, 
and $\E(X^C_n)=\xpfrac{n+1}{30}$ (for $n\ge4$).
Hence \eqref{expbinprotected} yields \eqref{expbinprot}.

To calculate $ \Var(\yxn2) $ we use Theorem \ref{Tcovbin}. 
Using the notations there $q_L^L=q_{C}^{C}=1$ and $q_L^C=2$, and simple
calculations yield (for $n\ge8$)
$
\Var(X^L_n)=\frac{2}{45}(n+1)
$
(as shown in \eqref{varbindegrees}),
$
\Cov(X^L_n,X^C_n)= \frac{2}{105}(n+1)
$
and
$
%\begin{align}\label{varcherry}
\Var(X^C_n)=\frac{43}{1575}(n+1),
%\end{align}
$
which together with \eqref{varbinprotected} yield \eqref{varbinprot}.

Since any linear combination of the components in a random vector with a
multivariate normal distribution is normal, the asymptotic normality
\eqref{normalbinprot} follows from \eqref{binprotected}
and Theorem \ref{multivariate}. 
\end{proof}

\begin{rem}\label{Rprot}
Alternatively, \eqref{binprotected} shows that $\yxn2=F(\ctn)$ for the
functional 
\begin{equation}
  f(T):=1-2\cdot\ett{T=L}+\ett{T=C}
\end{equation}
and the results follow by \refT{TF} (with the same calculations as above).
\end{rem}

\begin{proof}[Proof of Theorem \ref{protected2}]
We use again \eqref{binprotected} and the cyclic
representation \eqref{cyclic}, which show that
\begin{equation}\label{myr}
 \yxn2=n+\sumini g(\gs(i-1,4))
\end{equation}
for some functional $g$ defined by
$g(\gs(i-1,5)):=-2I_{i,1}+I_{i,3}f_C(\gs(i,3))$
where $f_C$ is the indicator that the permutation defines a cherry.
Thus $\E\yxn2^k$ can be calculated by substituting \eqref{myr} and expanding,
and the result follows easily by collecting terms that are equal
since
the random variables $g(\gs(i-1,5))$ are \iid{} and 4-dependent.
\end{proof}

\begin{rem}\label{remexpprotected} 
The asymptotic mean $\muy2$ in \eqref{lprotbinexp} can also be directed
directly from \eqref{tfe} in \refT{TF}. We give this alternative calculation
to illustrate our results, although in this case \eqref{mub} (see
\cite{DevroyeJanson}) or \eqref{expbinprotected}  yield simpler calculations.
Let $p_k$ be the probability that the root of $\ctk$ is two-protected.
Since the complement of the two-protected nodes consists of the leaves and
the parents of the leaves we obtain (for $k\ge2$), that the root is not
two-protected if and only if it has a child that is a leaf, which going
back to the construction of the \bst{} by a sequence of keys means that the
first key is either the second smallest or the second largest key.
Hence, $p_k=1-2/k$ for $k\ge4$. Furthermore,
$p_1=p_2=0$ and $p_3=2/3$. Consequently, \eqref{tfe} yields
\begin{align}\label{expbinprotCTF}
\muy2=\frac{2}{4\cdot5}\cdot\frac{2}3
+\sum_{k=4}^{\infty}\frac{2}{(k+2)(k+1)}\cdot\Bigpar{1- \frac{2}{k}}
=\frac{11}{30}.
\end{align}  
\end{rem}

\citet{MahmoudWard2012} also discuss the two-protected nodes in the extended
binary search tree. Recall that an extended binary search tree
is a binary search tree where the $ n+1 $ external children are added. 
The leaves in the extended binary tree are the external vertices; 
hence the two-protected nodes are those that have at least distance two to
an external vertex, \ie, the internal vertices that have no external
children. In other words, the two-protected nodes are precisely the nodes in
the
\bst{} that have outdegree 2.
Thus Example \ref{outdegreesbin} directly implies the following theorem in
\cite{MahmoudWard2012}.

\begin{thm}[\citet{MahmoudWard2012}]\label{protectedextended} 
Let $ Z_n $ denote the number of two-protected nodes in an extended binary search tree. Then \begin{align}
\label{expbinprot2} \E(Z_n)=\frac{n}{3}-\frac{2}{3},~\qquad\text{for} \quad n\geq 2, 
 \end{align}
and
\begin{align}
\label{varbinprot2} \Var(Z_n)=\frac{2}{45}(n+1),~\qquad\text{for} \quad n\geq 4. 
\end{align}
Furthermore, as \ntoo,
\begin{align}
\label{normalbinprot2}\dfrac{Z_n-\frac{n}{3}}{\sqrt{n}}
\dto
\mathcal{N}\Bigl(0,\frac{2}{45}\Bigr).  
\end{align}
\end{thm}

We can also show
the following result which was conjectured in 
\cite[Conjecture 4.1]{MahmoudWard2012}. 

\begin{thm}\label{protectedextended2} Let $ Z_n $ denoted the number of
  two-protected nodes in the extended binary search tree. For each fixed
  integer $ k\geq 1 $, there exists a polynomial $ p_k(n) $ of degree $ k $,
  the leading term of which is $ \frac{1}{3^{k}} $, such that $
  \E(Z_n^{k})=p_k(n) $ for all $ n\geq 2k $. 
\qed
\end{thm}

\begin{proof}[Proof of Theorem \ref{protectedextended2}] 
By the comments above, \eqref{outdegreetwo} and \eqref{xnk},
\begin{equation}
Z_n
=\dn{2}
=\dn0-1
=X_{n,1}-1
=\sumini I_{i,1}-1.
\end{equation}
The result follows
from the fact that the indicator functions $I_{i,1}$ are 2-dependent,
\cf{}  the proof of \refT{protected2}.
\end{proof}

\subsection{Shape functionals}

\subsubsection{Binary search trees}
Consider first a binary tree $T$ with $|T|=n$, and define
$P(T):=p_{n,T}=  \P(\ctn=T)$.
It is easy to see that
\begin{equation}\label{shapebin}
P(T):=  \P(\ctn=T)
=\prod_{v\in T} |T(v)|\qw,
\end{equation}
see \eg{}  (more generally for $m$-ary search trees) \citet{DobrowFill}.
The functional $P(T)$
%, or its negative logarithm $F(T)$, 
is known as the
\emph{shape functional} for binary trees.

By \eqref{shapebin}, the functional
$F(T):=-\log{P(T)}$ is given by \eqref{F} with $f(T)=\log|T|$.

\begin{example}[\citet{Fill1996}]
  \label{Eshapebin}
Theorems \ref{TF} and \ref{TG1} apply to $F(T)=-\log P(T)$,
and it follows immediately that, as shown by
\citet{Fill1996} (with some further details), see also
\citet{FillKapur-mary} for $m$-ary search trees, 
as \ntoo,
\begin{equation}\label{eshapebin}
  -\E \log P(\ctn) \sim n \sum_{k=2}^\infty \frac{2\log k}{(k+1)(k+2)}
\end{equation}
and
\begin{equation}
  \frac{\log P(\ctn)-\E \log P(\ctn)}{\sqrt n}\dto \cN(0,\gss)
\end{equation}
for some $\gss>0$ which  can be computed from \eqref{gssfpsi}.
(We have $\gss>0$ by \refT{TG1}, since \eg{} $P(\cT_3)$ is not deterministic.)
\end{example}

\subsubsection{Unordered random recursive trees}
For an unordered rooted tree $\gL$ with $|\gL|=n$, we similarly define
$\hP(\gL):=\hp_{n,\gL}=\P(\gL_n=\gL)$ (regarding $\gL_n$ as an unordered tree)..
Then, see \citet{FengMahmoud},
\begin{equation}\label{shaperec}
  \hP(\gL):=\P(\gL_n=\gL)=n\prod_{v\in\gL} s(\gL,v)\qw|\gL(v)|\qw,
\end{equation}
where $s(\gL,v)$ is the number of permutations of the children of $v$ that 
can be extended to
automorphisms of the tree $\gL$, \ie, if $v$ has $\nu_1$ children $v_{1i}$
such that $\gL(v_{1i})\cong \gL_1$ for some rooted tree $\gL_1$,
$\nu_2$ children $v_{2i}$
such that $\gL(v_{2i})\cong \gL_2$ for some different rooted tree $\gL_2$,
\dots, then $s(\gL,v)=\prod_j \nu_j!$.
This functional $\hP(T)$ is 
the {shape functional} for unordered rooted trees.

By \eqref{shaperec}, 
the functional
$-\log{\hP(\gL)}=F(\gL)-\log|\gL|$, where 
$F(\gL)$ is given by \eqref{F} with 
$f(\gL)=\log|\gL|+\log s(\gL,o)$, where $o$ is the root.
Note that this functional is more complicated that the corresponding one for
binary trees, and that $f(\gL)$ no longer depends only on the size $|\gL|$
(nor only on the size of $\gL$ and of the principal subtrees as in
\refT{TGr}). Nevertheless, \refT{TF} applies and yields the
following. (It seems obvious that $\hgss>0$, but we have no rigorous
proof. We have not attempted any numerical estimate.)

\begin{thm}
  As \ntoo,
\begin{equation}\label{eshaperec}
  -\E \log \hP(\gln) \sim n \hmu
\end{equation}
and
\begin{equation}
  \frac{\log \hP(\gln)-\E \log \hP(\gln)}{\sqrt n}\dto \cN(0,\hgss)
\end{equation}
for some $\hmu>0$ and $\hgss\ge0$ which in principle can be computed from
\eqref{tfer}--\eqref{tfvr}. 
\end{thm}

\begin{proof}
  Let $d(\gL)$ be the degree of the root $o$ of $\gL$. Then, crudely,
$s(o,\gL)\le d(\gL)!$ and thus 
  \begin{equation}\label{so}
\log s(o,\gL)\le d(\gL) \log d(\gL)
\le d(\gL) \log |\gL|.	
  \end{equation}
From the definition of the \rrt, $d(\gL_k)\eqd \sum_{i=1}^{k-1} I_i$,
where $I_i\sim \Be(1/i)$ are i.i.d., and a simple calculation shows that
\begin{equation}\label{ed}
  \E d(\gL_k)^2 = O\bigpar{\log^2 k}.
\end{equation}
By \eqref{so} and \eqref{ed} we have, rather crudely,
\begin{equation}
  \begin{split}
  \E \bigpar{f(\gL_k)^2}
&=  \E\bigpar{\xpar{\log s(o,\gL_k)+\log k}^2}
%\\&
\le 2\E\bigpar{{d(\gL_k)^2\log^2 k}}+2\log^2 k
\\&
= O(\log^4 k).	
  \end{split}\label{ek}
\end{equation}
Hence, \eqref{tfa1r}--\eqref{tfa3r} hold, and \refT{TF}(ii) applies.
\end{proof}

\subsubsection{Ordered random recursive trees}
Now consider the \rrt{} $\gL_n$ as an ordered tree.
For an ordered rooted tree $\gL$ with $|\gL|=n$, we define
$\hP(\gL):=\hp_{n,\gL}=\P(\gL_n=\gL)$.
It is easily seen that if we denote the children of a node $v$ by
$v_1,\dots,v_{d(v)}$, then
\begin{equation}\label{shapereco}
  \hP(\gL):=\P(\gL_n=\gL)
=\prod_{v\in\gL} \prod_{i=1}^{d(v)} \lrpar{\sum_{j=i}^{d(v)} |\gL(v_j)|}\qw
\end{equation}
This functional $\hP(T)$ is
the {shape functional} for ordered rooted trees.

By \eqref{shapereco}, 
the functional
$-\log{\hP(\gL)}=F(\gL)$, where 
$F(\gL)$ is given by \eqref{F} with 
\begin{equation}
  f(\gL)=
\sum_{i=1}^{d}\log {\sum_{j=i}^{d} |\gL_j|}
\end{equation}
where $d$ is the degree of the root and $\gL_1,\dots,\gL_d$ are the
principal subtrees.
This functional is of the type
in \refT{TGr}, and we obtain the following.

\begin{thm}
  As \ntoo,
\begin{equation}\label{eshapereco}
  -\E \log \hP(\gln) \sim n \hmu
\end{equation}
and
\begin{equation}
  \frac{\log \hP(\gln)-\E \log \hP(\gln)}{\sqrt n}\dto \cN(0,\hgss)
\end{equation}
for some $\hmu>0$ and $\hgss>0$ which in principle can be computed from
\eqref{tfer} and \eqref{gssfpsir}. 
\end{thm}

\begin{proof}
  Let again $d(\gL)$ be the degree of the root of $\gL$. Then, 
$f(\gL)\le d(\gL)\log|\gL|$, and \eqref{ed} shows that \eqref{ek} holds in
  the ordered 
  case too; thus the result follows by Theorems \ref{TF} and \ref{TGr}.
\end{proof}

\appendix
\section{Appendix: proof of \eqref{xb}--\eqref{xr}}

Let $|T|=k$ and $|T'|=m<k$. By \eqref{gamma},  
\begin{equation}
  \bgam(k,m)=\frac{4}{(k+2)(m+1)(m+2)} + O\Bigpar{\frac{1}{k^2m}}
\end{equation}
and thus \eqref{covbin2} yields
\begin{equation}\label{50a}
    \gstt=\frac{2}{(k+1)(k+2)}\pkt
\lrpar{\qtt-\frac{2(k+1)}{(m+1)(m+2)}\pmtx}
+ O\lrpar{\frac{\pkt\pmtx}{k^2m}}.
\end{equation}
Note first that we may ignore the $O$ term in \eqref{50a}, since
\begin{equation}\label{50b}
\sum_{m=1}^\infty\sum_{k>m}\sum_{|T|=k}\sum_{|T'|=m}
 \frac{\pkt\pmtx}{k^2m}
=
\sum_{m=1}^\infty\sum_{k>m} \frac{1}{k^2m}
<\infty;
\end{equation}
hence it suffices to show that
\begin{equation}
S_1:=
\sum_{m=1}^\infty\sum_{k>m}\sum_{|T|=k}\sum_{|T'|=m}
\frac{2}{(k+1)(k+2)}\pkt
\lrabs{\qtt-\frac{2(k+1)}{(m+1)(m+2)}\pmtx}
=\infty.
\end{equation}
Note that $\qtt$ is an integer. Thus, if
$\frac{2(k+1)}{(m+1)(m+2)}\pmtx\le\frac12$, then 
\begin{equation}
\lrabs{\qtt-\frac{2(k+1)}{(m+1)(m+2)}\pmtx}\ge
\frac{2(k+1)}{(m+1)(m+2)}\pmtx . 
\end{equation}
Hence,
\begin{equation}\label{emx}
  \begin{split}
S_1
&\ge
\sum_{m=1}^\infty\sum_{|T'|=m}\sum_{k=m+1}^{m^2/(4\pmtx)}\sum_{|T|=k}
\frac{2}{(k+1)(k+2)}\pkt
\frac{2(k+1)}{(m+1)(m+2)}\pmtx
\\&=	
\sum_{m=1}^\infty\sum_{|T'|=m}
\frac4{(m+1)(m+2)}\pmtx
\sum_{k=m+1}^{m^2/(4\pmtx)}
\frac{1}{k+2}
\\&\ge
\sum_{m=1}^\infty\sum_{|T'|=m}
\frac4{(m+1)(m+2)}\pmtx
\log\frac{m^2}{4\pmtx(m+3)}
\\&\ge
\sum_{m=6}^\infty\frac1{m^2}\sum_{|T'|=m}
\pmtx
\log\frac{1}{\pmtx}
=
\sum_{m=6}^\infty\frac1{m^2}\E \log\frac{1}{p_{m,\cT_m}}.
  \end{split}
\end{equation}
However, we saw in \eqref{eshapebin} that
$\E \log{p_{m,\cT_m}\qw}\sim \mu m$ as \mtoo{}
for some constant $\mu>0$, and thus 
the sum in \eqref{emx} diverges, which as said above implies \eqref{xb} by
\eqref{50a}--\eqref{50b}. 

The proof of the random recursive case \eqref{xr} is similar, using
\eqref{covbin3}, \eqref{gamma4} and \eqref{eshaperec}.
(It suffices to consider the versions with unlabelled binary trees and
unordered rooted trees, since the versions with increasing trees or ordered
trees have larger 
$\E \log{p_{m,\cT_m}\qw}$ and
$\E \log{\hp_{m,\gL_m}\qw}$.)
\qed

\begin{rem}
  The proof above needs only lower bounds for 
$\E \log{p_{m,\cT_m}\qw}$
and $\E \log{\hp_{m,\gL_m}\qw}$. We can use the simple bounds
$ \log{p_{m,\cT_m}\qw}\ge X_{m,2}\log 2$
and
$ \log{\hp_{m,\gL_m}\qw}\ge \hX_{m,2}\log 2-\log m$,
see \eqref{shapebin} and \eqref{shaperec},
together with \eqref{munk}--\eqref{murrt} instead of \eqref{eshapebin} and
\eqref{eshaperec}.
\end{rem}

\begin{rem}
  Recalling the notation in \refS{S:intro}, $\qtt$ is the value of
  $X_k^{T'}$ when $\ctk=T$, and it follows, using \eqref{ptp}, that
\begin{equation}
\sum_{|T|=k}
\pkt
\lrabs{\qtt-\frac{2(k+1)}{(m+1)(m+2)}\pmtx}
=\E\lrabs{X_k^{T'}-\E X_k^{T'}}
\end{equation}
and thus, dropping the prime,
\begin{equation}
S_1=
\sum_{T}\sum_{k>|T|}
\frac{2}{(k+1)(k+2)}
\E\lrabs{X_k^{T}-\E X_k^{T}}
.
\end{equation}
A similar sum appears in the proof for the \rrt.
\end{rem}

\end{document}